%% file: llg+dmi+mps.tex
\documentclass[a4paper,12pt,reqno]{amsart}

\usepackage[T1]{fontenc}
\usepackage[utf8]{inputenc}
\usepackage[english]{babel}
\usepackage{amsmath,amsthm,amssymb}
\usepackage{fullpage}
\usepackage{xcolor}
\usepackage{siunitx}
\usepackage{pgfplotstable}
\pgfplotsset{compat = 1.13, tick label style = {font = \tiny}}
\tikzset{
    invisible/.style={opacity=0},
    visible on/.style={alt=#1{}{invisible}},
    alt/.code args={<#1>#2#3}{%
      \alt<#1>{\pgfkeysalso{#2}}{\pgfkeysalso{#3}} 
    },
  }
\usepackage[width=0.85\textwidth]{subcaption}
\usepackage{hyperref}
\usepackage{catchfile}

\makeatletter
\def\@seccntformat#1{%
  \protect\textup{\protect\@secnumfont
    \ifnum\pdfstrcmp{subsection}{#1}=0 \bfseries\fi
    \csname the#1\endcsname
    \protect\@secnumpunct
  }%
}
\makeatother

\usepackage{fancyhdr}
\advance\footskip0.4cm
\advance\textheight-0.4cm
\calclayout
\newtheorem{theorem}{Theorem}[section]
\newtheorem{proposition}[theorem]{Proposition}

\newtheorem{lemma}[theorem]{Lemma}
\newtheorem{algorithm}[theorem]{Algorithm}
\newtheorem{definition}[theorem]{Definition}
\newtheorem{remark}[theorem]{Remark}

\newcommand\axis{\boldsymbol{a}}
\newcommand\ee{\boldsymbol{e}}
\newcommand\ff{\boldsymbol{f}}
\newcommand\hh{\boldsymbol{h}}
\newcommand\mm{\boldsymbol{m}}
\newcommand\nn{\boldsymbol{n}}
\newcommand\rr{\boldsymbol{r}}
\newcommand\svec{\boldsymbol{s}}
\newcommand\uu{\boldsymbol{u}}
\newcommand\vv{\boldsymbol{v}}

\newcommand\xx{\boldsymbol{x}}
\newcommand\yy{\boldsymbol{y}}
\newcommand\zz{\boldsymbol{z}}
\newcommand\Amat{\boldsymbol{A}}
\newcommand\matrixM{\mathbf{M}}
\newcommand\CC{\boldsymbol{C}}
\newcommand\HH{\boldsymbol{H}}
\newcommand\JJ{\boldsymbol{J}}
\newcommand\KK{\boldsymbol{K}}
\renewcommand\S{\mathcal{S}}
\newcommand\Id{\mathbf{Id}}
\newcommand\FF{\boldsymbol{F}}
\newcommand\LL{\boldsymbol{L}}
\newcommand\M{\mathcal{M}}
\newcommand\E{\mathcal{E}}
\newcommand\NN{\mathcal{N}}
\newcommand\OO{\mathcal{O}}
\newcommand\T{\mathcal{T}}
\newcommand\Vh{\boldsymbol{V}_{\!\!h}}
\newcommand\Mh{\boldsymbol{\M}_h}
\newcommand\Nh{\NN_h}
\newcommand\Ph{\mathbb{P}_h}
\newcommand\Th{\T_h}
\newcommand\vphi{\varphi}
\newcommand\pphi{\boldsymbol{\phi}}
\newcommand\vvphi{\boldsymbol{\vphi}}
\newcommand\ppsi{\boldsymbol{\psi}}
\newcommand\eps{\varepsilon}
\newcommand\xxi{\boldsymbol{\xi}}
\newcommand\eeta{\boldsymbol{\eta}}
\newcommand\interp{\mathcal{I}_h}
\newcommand\Interp{\boldsymbol{\mathcal{I}}_h}
\newcommand\mmhk{\mm_{hk}}
\newcommand\mmhkbar{\overline{\mm}_{hk}}
\newcommand\pphih{\pphi_h}
\newcommand\ppi{\boldsymbol{\pi}}
\newcommand\PPi{\boldsymbol{\Pi}}
\newcommand\0{\boldsymbol{0}}
\newcommand\sphere{\mathbb{S}^2}
\newcommand\curl{\nabla\times}

\newcommand\grad{\nabla}
\newcommand\Grad{\boldsymbol{\nabla}}
\newcommand\Lapl{\boldsymbol{\Delta}}
\newcommand\N{\mathbb{N}}
\newcommand\R{\mathbb{R}}
\newcommand{\abs}[1]{\lvert #1 \rvert}
\newcommand{\dual}[3][]{\langle #2,#3 \rangle_{#1}}
\newcommand{\inner}[3][]{\langle #2,#3 \rangle_{#1}}
\newcommand{\norm}[2][]{\lVert #2 \rVert_{#1}}
\newcommand{\weakstarto}{\overset{\star}{\rightharpoonup}}
\newcommand{\weakto}{\rightharpoonup}
\DeclareMathOperator{\diam}{diam}
\DeclareMathOperator{\vol}{vol}
\newcommand\ddt{\frac{\mathrm{d}}{\mathrm{d}t}}
\newcommand\de{\partial}
\newcommand\dt{\mathrm{d}t}
\newcommand\dx{\mathrm{d}\xx}
\newcommand\mmt{\de_t \mm}
\newcommand\Heff{\HH_{\mathrm{eff}}}
\newcommand\Hext{\HH_{\mathrm{ext}}}
\newcommand\Hs{\HH_{\mathrm{s}}}
\newcommand\ldm{\ell_{\mathrm{dm}}}
\newcommand\lex{\ell_{\mathrm{ex}}}
\newcommand\Ms{M_{\mathrm{s}}} 
\newcommand\heff{\hh_{\mathrm{eff}}}
\newcommand\Cinv{C_{\mathrm{inv}}}
\newcommand\aloc{a^{\mathrm{loc}}}
\newcommand\heffloc{\hh_{\mathrm{eff}}^{\mathrm{loc}}}
\newcommand\Eloc{\E^{\mathrm{loc}}}
\newcommand\Cpi{C_{\ppi}}
\newcommand\ddelta{\boldsymbol{\delta}}
\newcommand\newton{\star}
\newcommand\MMag{\boldsymbol{M}}
\newcommand\II{\mathbf{I}}

\begin{document}

\title{The mass-lumped midpoint scheme for computational micromagnetics:
Newton linearization and application to magnetic skyrmion dynamics}
\author{Giovanni~Di~Fratta}
\author{Carl-Martin~Pfeiler}
\author{Dirk~Praetorius}
\author{Michele~Ruggeri}
\address{Dipartimento di Matematica e Applicazioni ``R.\ Caccioppoli'', Universit\`{a} degli Studi di Napoli ``Federico II'', Via Cintia, Complesso Monte S.\ Angelo, 80126 Napoli, Italy}
\email{giovanni.difratta@unina.it}
\address{Institute of Analysis and Scientific Computing, TU Wien, Wiedner Hauptstrasse 8--10, 1040, Vienna, Austria}
\email{carl-martin.pfeiler@asc.tuwien.ac.at}
\email{dirk.praetorius@asc.tuwien.ac.at}
\address{Department of Mathematics and Statistics, University of Strathclyde, 26 Richmond Street, Glasgow G1 1XH, Scotland, UK}
\email{michele.ruggeri@strath.ac.uk}

\date{\today}
\thanks{\emph{Acknowledgements.}
This research has been supported by the Austrian Science Fund (FWF) through the doctoral school \emph{Dissipation and dispersion in nonlinear PDEs} (grant W1245) and the special research program \emph{Taming complexity in partial differential systems} (grant F65).
Giovanni Di~Fratta acknowledges the support of the Austrian Science Fund (FWF) through the project \emph{Analysis and Modeling of Magnetic Skyrmions} (grant P-34609). Giovanni Di~Fratta also thanks TU Wien and MedUni Wien for their support and hospitality.}

\begin{abstract}
\input{sec_abstract.tex}
\end{abstract}

\maketitle
\thispagestyle{fancy}

\input{sec_introduction.tex}
\input{sec_problem_formulation.tex}
\input{sec_algorithms.tex}
\input{sec_numerics.tex}
\input{sec_proof_ideal.tex}
\input{sec_proof_practical_fp.tex}
\input{sec_proof_practical_newton.tex}

\bibliographystyle{alpha} 
\bibliography{ref}
\end{document}

%% file: sec_abstract.tex
We discuss a mass-lumped midpoint scheme for the numerical approximation of the Landau--Lifshitz--Gilbert equation, which models the dynamics of the magnetization in ferromagnetic materials.
In addition to the classical micromagnetic field contributions, our setting covers the non-standard Dzyaloshinskii--Moriya interaction, which is the essential ingredient for the enucleation and stabilization of magnetic skyrmions.
Our analysis also includes the inexact solution of the arising nonlinear systems, for which we discuss both a constraint preserving fixed-point solver from the literature and a novel approach based on the Newton method.
We numerically compare the two linearization techniques and show that the Newton solver leads to a considerably lower number of nonlinear iterations.
Moreover, in a numerical study on magnetic skyrmions, we demonstrate that, for magnetization dynamics that are very sensitive to energy perturbations, the midpoint scheme, due to its conservation properties, is superior to the dissipative tangent plane schemes from the literature.

%% file: sec_introduction.tex
\section{Introduction}

\subsection{Energetics of a ferromagnet}

In the continuum theory of micromagnetism,
whose origin dates back to the seminal work of Landau--Lifshitz~\cite{ll1935} on small ferromagnetic particles,
the amount of magnetic moment (per unit volume) of a rigid ferromagnetic body
occupying a bounded region $\Omega \subset \R^3$ is represented by a classical vector field,
the \emph{magnetization} $\MMag \colon \Omega \to \R^3$.
Its module, $\Ms := \abs{\MMag}$, describes the so-called \emph{saturation magnetization}.
In single-crystal ferromagnets~\cite{afm2006,ad2015},
$\Ms$ depends only on the temperature and is assumed to be constant when the specimen is well below the so-called Curie temperature of the material.
In this case, the magnetization can be represented in the form $\MMag := \Ms \mm$,
where $\mm \colon \Omega \to \sphere$ is a vector field with values in the unit sphere of $\R^3$,
and the observable magnetization states minimize the micromagnetic energy functional~\cite{brown1963,hs1998}
\begin{align} \label{mpslabel:eq:GLunorm}
&\E(\mm)
: =
\E_{\Omega} (\mm)
+ \mathcal{K}_{\Omega} (\mm)
+ \mathcal{W}_{\Omega} (\mm)
+ \mathcal{A}_{\Omega}(\mm)
+ \mathcal{Z}_{\Omega}(\mm)
\\
\notag&\; : =
\int_{\Omega} A \abs{\Grad \mm}^2
+  D (\curl \mm)\cdot \mm
- \frac{\mu_0}{2}  \Ms \Hs (\mm) \cdot \mm
+ \varphi_{\mathrm{an}} (\mm)
- \mu_0 \Ms \Hext \cdot \mm\; \dx\,,
\end{align}
defined for every $\mm \in H^1 ( \Omega ; \sphere)$.

The \emph{exchange energy}, $\E_{\Omega}(\mm)$, penalizes spatial variations
of the direction of the magnetization, with $A>0$ representing a
material-dependent constant that summarizes the stiffness of short-range (symmetric) exchange interactions.
The second term, $\mathcal{K}_{\Omega}(\mm)$, represents the
\emph{bulk Dzyaloshinskii--Moriya interaction (DMI)}~\cite{dzyaloshinskii1958,moriya1960},
and accounts for antisymmetric exchange interactions caused by possible lacks of inversion
symmetry in the crystal structure of the ferromagnet.
The sign of the constant $D \in \R$ affects the chirality of the ferromagnetic
system~{\cite{trjcf2012,scrtf2013}}.
The third term, $\mathcal{W}_{\Omega}(\mm)$, is the \emph{magnetostatic self-energy},
i.e., the energy due to the stray field $\Hs (\mm)$ induced by $\Ms\mm$.
From the mathematical point of view, $\Hs (\mm)$ can be
characterized as the projection of $(-\Ms\mm) \in L^2(\R^3; \R^3)$
on the closed subspace of gradient vector fields $\grad H^1(\R^3 ; \R) := \{ \grad u \colon u \in H^1 ( \R^3 ; \R) \}$
(see, e.g., {\cite{praetorius2004,dmrs2020}})\footnote{Here, with a slight abuse of notation, we
identify $\mm$ with its extension by zero to the whole $\R^3$.}.
Here, $\mu_0$ denotes the vacuum permeability.
Additionally, the micromagnetic energy includes two additional energy contributions:
the \emph{magnetocrystalline anisotropy energy} $\mathcal{A}_{\Omega}(\mm)$
and the \emph{Zeeman energy} $\mathcal{Z}_{\Omega}(\mm)$.
The energy density $\varphi_{\mathrm{an}} \colon \sphere \to \R_{\ge 0}$ models
the existence of easy directions of the magnetization due to the crystallographic structure of the ferromagnet, while $\mathcal{Z}_{\Omega}(\mm)$ models the tendency of a specimen to have the magnetization
aligned with the external applied field $\Hext \in L^2 ( \Omega ; \R^3 )$,
assumed to be unaffected by variations of $\mm$.
The competition among the energy contributions in {\eqref{mpslabel:eq:GLunorm}} explains {\emph{most}} of the striking
spin textures observable in ferromagnetic materials~{\cite{hs1998}},
in particular, the emergence of magnetic skyrmions~{\cite{fcs2013,frc2017}}.

\subsection{A more general energy functional} \label{mpslabel:sec:energy}

When a ferromagnetic system consists of several magnetic materials,
the material-dependent quantities $A$, $D$, and $\Ms$
are no longer constant in the region $\Omega$ occupied by the ferromagnet,
and one has to model spin interactions among different magnetic materials at their touching interface~\cite{afm2006}.
The easiest way is to assume a \emph{strong coupling condition}~\cite{ad2015,abmn2021,dd2020}:
Although $\Ms$ can be discontinuous across an interface, the direction of the
magnetization never jumps through it.
Under this constitutive assumption, the analysis of the composite can be carried out under the classical conditions
$M_s \in L^{\infty} ( \Omega ; \R_{> 0} )$ and $\mm \in H^1 (\Omega ; \sphere)$.
In this setting, the observable states of a rigid ferromagnetic body can be characterized as the local minimizers
of the micromagnetic energy functional still defined by~\eqref{mpslabel:eq:GLunorm},
but with the quantities $A = A(\xx)$, $D = D(\xx)$, and $\Ms = \Ms(\xx)$ to be understood as functions defined on $\Omega$.

In this paper, we are interested in a more general energy functional which,
other than incorporating the previous one as a special case,
also accounts for the presence of anisotropies in the lattice structures of the constituents.
To introduce the model, we first observe that the bulk DMI energy density
can be equivalently rewritten as
\begin{equation*}
D (\curl \mm) \cdot \mm = D \sum_{d = 1}^3 (\ee_d \times \partial_d \mm) \cdot \mm,
\end{equation*}
where $\{ \ee_d \}_{d = 1,2,3}$ denotes the standard basis of $\R^3$.
It is therefore a special case of the energy density
\begin{equation*}
g_{\mathrm{asym}} (\xx, \svec, \xxi) = \sum_{d = 1}^3 \KK_{\!d} (\xx) \xxi_d \cdot \svec
\quad
\text{for all }
\xx \in \Omega, \, \svec \in \R^3, \text{ and } \xxi = (\xxi_1,\xxi_2,\xxi_3) \in \R^{3 \times 3},
\end{equation*}
with $\{ \KK_{\!d} \}_{d=1,2,3}$ being $3$-by-$3$ antisymmetric matrices,
i.e., $\KK_{\!d} = - \KK_{\!d}^T$.
Similarly, the symmetric exchange energy density can be generalized to the density
\begin{equation*}
g_{\mathrm{sym}} (\xx, \xxi) = \frac{1}{2} \sum_{d = 1}^3 \Amat_d (\xx) \xxi_d \cdot \xxi_d
\quad
\text{for all }
\xx \in \Omega \text{ and } \xxi = (\xxi_1,\xxi_2,\xxi_3) \in \R^{3 \times 3},
\end{equation*}
with $\{ \Amat_d \}_{d=1,2,3}$ being $3$-by-$3$ invertible symmetric matrices,
i.e., $\Amat_d = \Amat_d^T$.
Hence, for $g := g_{\mathrm{sym}} + g_{\mathrm{asym}}$,
it holds that
\begin{equation} \label{mpslabel:eq:intro:energy}
\begin{split}
g (\xx, \svec, \xxi)
& =
\frac{1}{2} \sum_{d=1}^3
\big( \Amat_d (\xx) \xxi_d \cdot \xxi_d - 2 \KK_{\!d} (\xx) \svec \cdot \xxi_d \big) \\
& =
\frac{1}{2} \sum_{d=1}^3
\Amat_d (\xx) \left( \xxi_d - \Amat_d^{- 1} (\xx) \KK_{\!d} (\xx) \svec \right)
\cdot \left( \xxi_d - \Amat_d^{- 1} (\xx) \KK_{\!d} (\xx) \svec \right) \\
& \quad + \frac{1}{2} \sum_{d=1}^3 \KK_{\!d} (\xx) \Amat_d^{- 1} (\xx) \KK_{\!d} (\xx) \svec \cdot \svec.
\end{split}
\end{equation}
Note that $\KK_{\!d} (\xx) \Amat_d^{- 1} (\xx) \KK_{\!d} (\xx)$ is a symmetric matrix.
This discussion suggests the opportunity to investigate an energy functional covering the above generalized form;
see~\eqref{mpslabel:eq:llg:energy} below.
It is worth to notice that the structure of this energy functional does not only allow
for the description of a mixture of ferromagnetic materials, but also
covers typical homogeneous models arising as $\Gamma$-limit of composite
ferromagnetic materials with highly oscillating heterogeneities~\cite{abmn2021,dd2020}.

\subsection{Landau--Lifshitz--Gilbert equation and its numerical integration}

When the magnetization $\mm$ does not minimize the micromagnetic energy functional,
the ferromagnetic system is in a non-equilibrium state.
A well-accepted model for its time evolution is the Landau--Lifshitz--Gilbert equation (LLG)~\cite{ll1935,gilbert1955},
which in the so-called Gilbert form reads
\begin{equation} \label{mpslabel:eq:LLG}
\mmt = - \gamma_0 \, \mm \times \Heff(\mm) + \alpha \, \mm \times \mmt.
\end{equation}
This phenomenological equation describes the magnetization dynamics as a
dissipative precession driven by the effective field $\Heff(\mm) := - \mu_0^{-1} \Ms^{-1} \frac{\delta\E(\mm)}{\delta \mm}$,
and modulated by the gyromagnetic ratio of the electron $\gamma_0>0$
and the Gilbert damping parameter $\alpha>0$.
The numerical approximation of LLG is not a trivial task.
Nonlinearities, the numerical realization of the unit-length constraint,
the possible coupling with other (nonlinear) partial differential equations,
and the need of unconditionally stable numerical schemes
make the problem very challenging.
For this reason, in the last twenty years, the problem has been the subject of several mathematical studies;
see, e.g., \cite{prohl2001, aj2006,kp2006,bp2006,gc2007,alouges2008a, bkp2008,  cimrak2008b, cimrak2009,akt2012,akst2014, bffgpprs2014, ahpprs2014, ft2017, kw2018, hpprss2019,dpprs2017,  afkl2021}.

In this work, we consider the mass-lumped midpoint scheme proposed in~\cite{bp2006}.
The method is based on a mass-lumped first-order finite element method for the spatial discretization
and the second-order midpoint rule for the time discretization,
and involves the solution of one nonlinear system per time-step.
Besides introducing the method, the work~\cite{bp2006} proves
unconditional convergence of the finite element approximation towards
a weak solution of LLG in the sense of~\cite{as1992} and proposes a fixed-point iteration
to linearize the nonlinear problem arising from the scheme.
The scheme has also been the subject of further research:
On the one hand, the works~\cite{bartels2006,cimrak2009} incorporate the inexact solution of the nonlinear system into the convergence result.
On the other hand, the work~\cite{prs2018} focuses on the design and the analysis
of effective approaches to treat the nonlocal field contributions.

\subsection{Contributions}

In this work, as a novel contribution, we extend the midpoint scheme
and its analysis to more general energy contributions; see the discussion in Section~\ref{mpslabel:sec:energy}.
In particular, the present analysis covers DMI,
which is not covered by the analysis in~\cite{bp2006,bartels2006,cimrak2009,prs2018}.
We note that DMI is
the essential ingredient for the enucleation and the stabilization of magnetic skyrmions.
At this point, it is worth pointing out that DMI contributions
represent a challenging testing ground for numerical schemes for LLG.
Indeed, besides requiring accurate adaptations in the numerical analysis,
they determine magnetization configurations --- magnetic skyrmions --- that
turn out to be very sensitive to small perturbations of the micromagnetic energy.
In addition, we also discuss the linearization of the nonlinear scheme:
We extend the fixed-point iteration proposed in~\cite{bartels2006} to the present setting and propose an approach based on the Newton method, for which we provide a first full analysis (well-posedness, stability, convergence).
Finally, in a collection of numerical experiments,
we accurately test the energy conservation properties of the mass-lumped midpoint scheme
and extensively compare it
with the tangent plane schemes from~\cite{alouges2008a,akst2014,hpprss2019,dpprs2017}.

\subsection{Outline}

The remainder of the work is organized as follows:
We conclude this section by collecting the notation used throughout the paper.
In Section~\ref{mpslabel:sec:problem},
we describe the mathematical problem under consideration.
In Section~\ref{mpslabel:sec:algorithm},
we present the proposed algorithms and state their stability and convergence results.
Section~\ref{mpslabel:sec:numerics} is devoted to numerical experiments.
Finally, in Sections~\ref{mpslabel:sec:proofs1}--\ref{mpslabel:sec:proofs2}, we collect the proofs of the results stated in Section~\ref{mpslabel:sec:algorithm}.

\subsection{Notation}

Throughout the paper,
we use the standard notation for Lebesgue, Sobolev, and Bochner spaces and norms.
To highlight (spaces of) vector-valued or matrix-valued functions, we use bold letters,
e.g., we denote both $L^2(\Omega;\R^3)$ and $L^2(\Omega;\R^{3 \times 3})$ by $\LL^2(\Omega)$.
We denote by $\inner[\Omega]{\cdot}{\cdot}$ the scalar product in $\LL^2(\Omega)$
and by $\dual{\cdot}{\cdot}$ the duality pairing between $\HH^1(\Omega)$ and its dual.
By $C>0$ we always denote a generic constant, which is independent of the discretization parameters, but not necessarily the same at each occurrence.

%% file: sec_problem_formulation.tex
\section{Problem formulation} \label{mpslabel:sec:problem}

Let $\Omega\subset\R^3$ be a bounded Lipschitz domain.
The energy of $\mm \in \HH^1(\Omega ; \sphere)$ is given by
\begin{equation} \label{mpslabel:eq:llg:energy}
\E(\mm)
=
\frac{1}{2} \, a(\mm,\mm) - \inner[\Omega]{\ff}{\mm},
\end{equation}
where $\ff \in \LL^2(\Omega)$, while
the bilinear form $a\colon \HH^1(\Omega) \times \HH^1(\Omega) \to \R$ is defined,
for all $\ppsi,\pphi \in \HH^1(\Omega)$,
by
\begin{equation} \label{mpslabel:eq:bilinearform}
a(\ppsi,\pphi)
= \sum_{d=1}^3 \inner[\Omega]{\Amat_d(\de_d \ppsi - \JJ_{\!d} \ppsi)}{\de_d \pphi - \JJ_{\!d} \pphi}
- \inner[\Omega]{\ppi(\ppsi)}{\pphi}.
\end{equation}
Here,
$\ppi\colon \LL^2(\Omega) \to \LL^2(\Omega)$ is a linear, bounded, and self-adjoint operator,
while, for $d=1,2,3$, the 3-by-3 matrices $\Amat_d$ and $\JJ_{\!d}$ have coefficients in $L^{\infty}(\Omega)$,
with $\Amat_d$ being also symmetric and uniformly positive definite, i.e., it holds that
$\Amat_d^T = \Amat_d$ and
\begin{equation*}
\Amat_d(\xx)\uu \cdot \uu \geq A_0 \abs{\uu}^2
\quad
\text{for almost all } \xx \in \Omega
\text{ and all } \uu\in\R^3,
\end{equation*}
where $A_0>0$ is a fixed constant.
The energy~\eqref{mpslabel:eq:llg:energy} covers the extensions of the classical micromagnetic functional discussed in Section~\ref{mpslabel:sec:energy}; cf.\ the expression in~\eqref{mpslabel:eq:intro:energy}.

The existence of minimizers of~\eqref{mpslabel:eq:llg:energy} in $\HH^1(\Omega ; \sphere)$ follows
from the direct method of calculus of variations.
Moreover, any minimizer $\mm \in \HH^1(\Omega ; \sphere)$ satisfies the Euler--Lagrange equations
\begin{equation*}
\dual{\heff(\mm)}{\pphi} = 0
\quad
\text{for all } \pphi \in \HH^1(\Omega) \text{ such that } \mm\cdot\pphi = 0 \text{ a.e.\ in } \Omega.
\end{equation*}
Here, $\heff(\mm) := - \frac{\delta\E(\mm)}{\delta \mm}$ is the (negative) G\^{a}teaux derivative of the energy, i.e.,
\begin{equation} \label{mpslabel:eq:heff}
- \dual{\heff(\mm)}{\pphi}
= \Big\langle\frac{\delta\E(\mm)}{\delta \mm} , \pphi \Big\rangle
= \lim_{\delta \to 0} \frac{\E(\mm + \delta \pphi) - \E(\mm)}{\delta}
\stackrel{\eqref{mpslabel:eq:llg:energy}}{=} a(\mm,\pphi) - \inner[\Omega]{\ff}{\pphi}.
\end{equation}
Turning to the dynamical case, a non-equilibrium configuration $\mm(t) \in \HH^1(\Omega ; \sphere)$
evolves according to~\eqref{mpslabel:eq:LLG}, which, after a suitable rescaling, reads
\begin{equation} \label{mpslabel:eq:llg}
\mmt = - \mm \times \heff(\mm) + \alpha \, \mm \times \mmt,
\end{equation}
with $\alpha>0$ being the Gilbert damping parameter.
The dynamics is dissipative in the sense that
any sufficiently smooth solution of~\eqref{mpslabel:eq:llg} satisfies the energy law
\begin{equation} \label{mpslabel:eq:decay}
\ddt\E(\mm(t)) = - \alpha \norm[\LL^2(\Omega)]{\mmt(t)}^2 \leq 0
\quad
\text{for all } t > 0.
\end{equation}
We conclude this section by recalling the notion of a weak solution of~\eqref{mpslabel:eq:llg}; see~\cite{as1992}.

\begin{definition} \label{mpslabel:def:weak}
Let $\mm^0 \in \HH^1(\Omega ; \sphere)$.
A vector field $\mm\colon\Omega \times \R_{>0} \to \sphere$ is called
a global weak solution of~\eqref{mpslabel:eq:llg}
if $\mm \in L^{\infty}(\R_{>0};\HH^1(\Omega ; \sphere))$ and, for all $T>0$, with $\Omega_T := \Omega \times (0, T)$ the following properties are satisfied:
\begin{itemize}
\item[\textrm{(i)}] $\mm\in \HH^1(\Omega_T)$;
\item[\textrm{(ii)}] $\mm(0)=\mm^0$ in the sense of traces;
\item[\textrm{(iii)}] For all $\vvphi\in\HH^1(\Omega_T)$, it holds that
\begin{equation} \label{mpslabel:eq:weak:variational}
\begin{split}
& \int_0^T \inner[\Omega]{\mmt(t)}{\vvphi(t)} \, \dt \\
& \quad = - \int_0^T \dual{\heff(\mm(t))}{\vvphi(t) \times \mm(t)} \, \dt
+ \alpha \int_0^T \inner[\Omega]{\mm(t)\times\mmt(t)}{\vvphi(t)} \, \dt;
\end{split}
\end{equation}
\item[\textrm{(iv)}] It holds that
\begin{equation} \label{mpslabel:eq:weak:energy}
\E(\mm(T)) + \alpha \int_0^T \norm[\LL^2(\Omega)]{\mmt(t)}^2 \dt
\leq \E(\mm^0).
\end{equation}
\end{itemize}
\end{definition}

We note that~\eqref{mpslabel:eq:heff} implicitly includes natural boundary conditions on $\mm$, which are homogeneous Neumann boundary conditions $\partial_{\nn} = \0$ if $\Amat_d= \lex^2 \Id$ and $\JJ_{\!d} = \0$ for $d=1,2,3$.
For a more explicit presentation, we refer to \cite{hpprss2019}.
The variational formulation~\eqref{mpslabel:eq:weak:variational} comes from a weak formulation of~\eqref{mpslabel:eq:llg}
in the space-time cylinder.
The energy inequality~\eqref{mpslabel:eq:weak:energy} is a weak counterpart of
the dissipative energy law~\eqref{mpslabel:eq:decay}.

\begin{remark} \label{mpslabel:rem:setting}
{\textrm{(i)}}
For ease of presentation, we restrict ourselves to the case of a time-independent field $\ff \in \LL^2(\Omega)$.
For time-dependent fields, the strong form~\eqref{mpslabel:eq:decay} and the weak form~\eqref{mpslabel:eq:weak:energy}
of the energy law of LLG read
\begin{equation*}
\ddt \widetilde\E(\mm(t)) = - \alpha \norm[\LL^2(\Omega)]{\mmt(t)}^2 + \inner[\Omega]{\ff(t)}{\mmt(t)}
\end{equation*}
and
\begin{equation*}
\widetilde\E(\mm(T))
+ \alpha \int_0^T \norm[\LL^2(\Omega)]{\mmt(t)}^2 \dt
- \int_0^T \inner[\Omega]{\ff(t)}{\mmt(t)} \, \dt
\leq \widetilde\E(\mm^0),
\end{equation*}
respectively,
where $\widetilde\E(\mm) = \E(\mm) + \inner[\Omega]{\ff}{\mm} = a(\mm,\mm)/2$.
\\
{\textrm{(ii)}}
The present setting covers and generalizes the model problems considered in previous
mathematical works on the numerical integration of LLG.
\begin{itemize}
\item With the choices $\Amat_d= \lex^2 \Id$ and $\JJ_{\!d} = \0$ for $d=1,2,3$,
where $\lex>0$ is the so-called exchange length and $\Id$ is the $3$-by-$3$ identity matrix,
$\ppi \equiv \0$, and $\ff \equiv \0$, we obtain that
\begin{equation*}
\dual{\heff(\ppsi)}{\vvphi}
= - \lex^2 \inner[\Omega]{\Grad\ppsi}{\Grad\pphi}.
\end{equation*}
This is the so-called small particle limit of LLG,
for which finite element schemes have been proposed, e.g.,
in the seminal papers~\cite{bp2006,alouges2008a}.
\item With the choices $\Amat_d= \lex^2 \Id$ and $\JJ_{\!d} = \0$ for $d=1,2,3$,
we obtain that
\begin{equation*}
\dual{\heff(\ppsi)}{\pphi}
= - \lex^2 \inner[\Omega]{\Grad\ppsi}{\Grad\pphi}
+ \inner[\Omega]{\ppi(\ppsi)}{\pphi}
+ \inner[\Omega]{\ff}{\pphi}.
\end{equation*}
This setting covers the classical energy contributions considered in micromagnetics
(exchange, uniaxial anisotropy, magnetostatic, Zeeman) and numerical integrators for
this case have been analyzed, e.g., in~\cite{akt2012,bffgpprs2014,akst2014,prs2018,dpprs2017}.
\item With the choices $\Amat_d= \lex^2 \Id$ and $\JJ_{\!d} = \ldm \, [\ee_d]_\times / (2 \lex^2)$ for $d=1,2,3$,
where $\ldm \in \R$ is a characteristic length associated with DMI,
$\ppi(\mm) = \ldm \, \mm/(2 \lex^2)$,
and $\ff \equiv \0$,
we obtain that\footnote{Here, $[\ee_d]_\times$ denotes the 3-by-3 matrix such that $[\ee_d]_\times \uu = \ee_d \times \uu$ for all $\uu \in \R^3$.}
\begin{equation*}
\dual{\heff(\ppsi)}{\pphi}
= - \lex^2 \inner[\Omega]{\Grad\ppsi}{\Grad\pphi}
- \frac{\ldm}{2} \inner[\Omega]{\curl\ppsi}{\pphi}
- \frac{\ldm}{2} \inner[\Omega]{\ppsi}{\curl\pphi},
\end{equation*}
which is the setting analyzed in~\cite{hpprss2019} for the simulation of chiral magnetic skyrmions by the means of a family of tangent plane integrators.
\end{itemize}
\end{remark}

%% file: sec_algorithms.tex
\section{Numerical algorithms and main results} \label{mpslabel:sec:algorithm}

\subsection{Preliminaries}\label{mpslabel:sec:preliminaries}

Let $\kappa \geq 1$.
For the spatial discretization, assuming $\Omega$ to be a polyhedral domain,
we consider a $\kappa$-quasi-uniform family $\{ \Th \}_{h>0}$
of regular tetrahedral triangulations of $\Omega$ parametrized by the
mesh size $h = \max_{K \in \Th} \diam(K)$, i.e.,
$\kappa^{-1} h \le \vol(K)^{1/3} \le h$ for all $K \in \Th$.
We denote by $\Nh$ the set of vertices of $\Th$.
For any $K \in \Th$, we denote by $\mathcal{P}^1(K)$
the space of first-order polynomials on $K$.
We consider the space of $\Th$-piecewise affine and globally continuous finite elements
\begin{equation*}
\S^1(\Th)
:=
\left\{v_h \in C^0(\overline{\Omega})\colon v_h \vert_K \in \mathcal{P}^1(K) \text{ for all } K \in \Th \right\}.
\end{equation*}
The classical basis for this finite-dimensional subspace of $H^1(\Omega)$ is the set of nodal hat functions
$\left\{\vphi_{\zz}\right\}_{\zz\in\Nh}$,
which satisfy $\vphi_{\zz}(\zz')=\delta_{\zz,\zz'}$ for all $\zz,\zz'\in\Nh$.
The nodal interpolant $\interp\colon C^0(\overline{\Omega}) \to \S^1(\Th)$ is defined by $\interp[u] = \sum_{\zz \in \Nh} u(\zz) \vphi_{\zz}$ for all $u \in C^0(\overline{\Omega})$.

Let $\Vh:= \S^1(\Th)^3$.
For each time-step, approximate solutions of~\eqref{mpslabel:eq:llg} are sought in the
\emph{set of admissible approximate magnetizations}
\begin{equation*}
\Mh := \left\{\pphih \in \Vh\colon \abs{\pphih(\zz)}=1 \text{ for all } \zz \in \Nh \right\},
\end{equation*}
which consists of all elements of $\Vh$ satisfying the unit-length constraint
at the nodes of the triangulation.

Besides the standard scalar product $\inner[\Omega]{\cdot}{\cdot}$,
given a mesh $\Th$ and the associated nodal interpolant $\interp[\cdot]$,
we consider the mass-lumped product $\inner[h]{\cdot}{\cdot}$ defined by
\begin{equation*}
\inner[h]{\ppsi}{\pphi}
= \int_\Omega \interp[\ppsi \cdot \pphi] \, \dx
\quad \text{for all } \ppsi, \pphi \in \CC^0(\overline{\Omega}).
\end{equation*}
Using the definition of the nodal interpolant, we see that
\begin{equation} \label{mpslabel:eq:mass-lumping}
\inner[h]{\ppsi}{\pphi}
= \sum_{\zz \in \Nh} \beta_{\zz} \, \ppsi(\zz)\cdot\pphi(\zz),
\quad
\text{where } \beta_{\zz} := \int_\Omega \vphi_{\zz} \, \dx > 0.
\end{equation}
On $\Vh$, $\inner[h]{\cdot}{\cdot}$ is a scalar product and
the induced norm $\norm[h]{\cdot}$ is equivalent to the standard norm of $\LL^2(\Omega)$.
In particular, it holds that
\begin{equation}  \label{mpslabel:eq:normEquivalence}
\norm[\LL^2(\Omega)]{\pphi_h}
\leq \norm[h]{\pphi_h}
\leq \sqrt{5} \, \norm[\LL^2(\Omega)]{\pphi_h}
\quad \text{for all } \pphi_h \in \Vh;
\end{equation}
see~\cite[Lemma~3.9]{bartels2015}.
Finally, we define the mapping $\Ph \colon \HH^1(\Omega)^{\star} \to \Vh$ by
\begin{equation} \label{mpslabel:eq:pseudo-projection}
\inner[h]{\Ph \uu}{\pphi_h}
= \dual{\uu}{\pphi_h}
\quad \text{for all } \uu \in \HH^1(\Omega)^{\star} \text{ and } \pphi_h \in \Vh,
\end{equation}
i.e., $\Ph \uu \in \Vh$ is the Riesz representative of $\inner{\uu}{\cdot} \in \Vh^\star$
in the Hilbert space $(\Vh,\inner[h]{\cdot}{\cdot})$.

For the time discretization, we consider a partition of the positive real axis $\R_{>0}$
with constant time-step size $k>0$, i.e., $t_i := ik$ for all $i \in \N_0$.
Given a sequence $\{ \pphi_h^i \}_{i \in \N_0} \subset \Vh$, we define
\begin{equation} \label{mpslabel:eq:midpoint-derivative}
\pphi_h^{i+1/2} := \frac{\pphi_h^{i+1} + \pphi_h^i}{2}
\quad
\text{and}
\quad
d_t\pphi_h^{i+1} := \frac{\pphi_h^{i+1} - \pphi_h^{i}}{k}
\quad
\text{for all } i \in \N_0,
\end{equation}
as well as the piecewise linear time reconstruction
\begin{equation} \label{mpslabel:eq:timeApprox}
\pphi_{hk}(t) := \frac{t-t_i}{k}\pphi_h^{i+1} + \frac{t_{i+1} - t}{k}\pphi_h^i
\quad
\text{for all } i \in \N_0
\text{ and } t \in [t_i,t_{i+1}],
\end{equation}
which satisfies $\pphi_{hk} \in \HH^1(\Omega \times (0,T))$ for any $T > 0$.
\subsection{Ideal midpoint scheme}

In the following algorithm, we adapt the scheme initially proposed in~\cite{bp2006} to the present setting.
The fundamental ingredients are the midpoint rule for the time discretization,
the finite element space $\Vh$ endowed with
the mass-lumped scalar product $\inner[h]{\cdot}{\cdot}$ for the spatial integration,
and the mapping~\eqref{mpslabel:eq:pseudo-projection} for the discrete realization of the effective field.
We refer to the method as \emph{ideal midpoint scheme} in the sense that, as we will see in the next section,
practical implementations require suitable modifications.

\begin{algorithm}[ideal midpoint scheme] \label{mpslabel:alg:mps}
Input:
$\mm_h^0 \in \Mh$. \\
Loop:
For all $i \in \N_0$, compute $\mm_h^{i + 1} \in \Vh$ such that
\begin{equation} \label{mpslabel:eq:mps}
\inner[h]{d_t \mm_h^{i+1}}{\pphih}
= - \inner[h]{ \mm_h^{i + 1/2} \times \Ph \heff(\mm_h^{i + 1/2})}{\pphih}
+ \alpha \inner[h]{\mm_h^{i+1/2} \times d_t \mm_h^{i+1}}{\pphih}
\end{equation}
for all $\pphih \in \Vh$.\\
Output:
Sequence of approximations $\left\{\mm_h^i\right\}_{i \in \N_0}$.
\qed
\end{algorithm}

With the sequence of approximations $\{\mm_h^i\}_{i \in \N_0}$ delivered
by Algorithm~\ref{mpslabel:alg:mps}, we define the piecewise linear time reconstruction $\mmhk$ via \eqref{mpslabel:eq:timeApprox}.
In the following theorem,
we establish the stability and convergence of the approximations obtained with Algorithm~\ref{mpslabel:alg:mps}.
The proof is postponed to Section~\ref{mpslabel:sec:proofs1}.

\begin{theorem} \label{mpslabel:thm:main}
{\textrm{(i)}}
Suppose that $\mm_h^0 \in \Mh$.
Then, for all $i \in \N_0$, \eqref{mpslabel:eq:mps} admits a solution $\mm_h^{i + 1} \in \Mh$.
In particular, the scheme preserves the unit-length constraint at any time-step at the nodes of the triangulation.\\
{\textrm{(ii)}}
The scheme is unconditionally stable in the sense that, for all $J \in \N$, it holds that
\begin{equation} \label{mpslabel:eq:stability}
\E(\mm_h^J) + \alpha k \sum_{i=0}^{J-1} \norm[h]{d_t \mm_h^i}^2
= \E(\mm_h^0).
\end{equation}
{\textrm{(iii)}}
Suppose that $\mm_h^0 \to \mm^0$ in $\HH^1(\Omega)$ as $h \to 0$.
Then, there exist a global weak solution $\mm \colon \Omega \times \R_{>0} \to \sphere$ of~\eqref{mpslabel:eq:llg}
in the sense of Definition~\ref{mpslabel:def:weak}
and a subsequence of $\{ \mmhk \}$ (not relabeled)
which unconditionally converges towards $\mm$.
Specifically,
as $h,k \to 0$,
$\mmhk \weakstarto \mm$ in $L^{\infty}(\R_{>0};\HH^1(\Omega ; \sphere))$
and $\mmhk\vert_{\Omega_T} \weakto \mm\vert_{\Omega_T}$ in $\HH^1(\Omega_T)$ for all $T>0$.
\end{theorem}

\begin{remark}
Note that, differently from the corresponding estimates
for tangent plane schemes~\cite{alouges2008a,akst2014,hpprss2019,dpprs2017},
the stability result for Algorithm~\ref{mpslabel:alg:mps} (Theorem~\ref{mpslabel:thm:main}{\textrm{(ii)}}) does not require
any geometric assumption on the mesh.
Moreover, \eqref{mpslabel:eq:stability} holds with equality and without any artificial dissipative term on the left-hand side.
\end{remark}

Theorem~\ref{mpslabel:thm:main}{\textrm{(i)}} establishes unconditional existence of a solution of~\eqref{mpslabel:eq:mps},
but does not provide information about its uniqueness.
If $k=o(h^2)$,
one can show that a suitable fixed-point iteration is a contraction
provided that the discretization parameters are sufficiently small.
With the Banach fixed-point theorem, this implies that each time-step of
Algorithm~\ref{mpslabel:alg:mps} is well-posed.

\begin{proposition} \label{mpslabel:prop:ideal:wellposedness}
Suppose that $k=o(h^2)$ as $h,k \to 0$.
Then, there exist thresholds $h_0>0$ and $k_0 > 0$ such that,
for all $h < h_0$ and $k < k_0$,
the variational problem~\eqref{mpslabel:eq:mps} admits a unique solution $\mm_h^{i+1} \in \Mh$ for all $i \in \N_0$.
\end{proposition}

The proof of Proposition~\ref{mpslabel:prop:ideal:wellposedness} can be obtained
simplifying the argument of
the proofs of Proposition~\ref{mpslabel:prop:fp} and Theorem~\ref{mpslabel:thm:main:fp} below,
therefore we omit it.

\subsection{Practical midpoint schemes}\label{mpslabel:sec:practical}

Each time-step of Algorithm~\ref{mpslabel:alg:mps} requires the solution of a \emph{nonlinear} system
and the computation of \emph{nonlocal} field contributions.

Nonlinearity is a consequence of the first term on the right-hand side of~\eqref{mpslabel:eq:mps}.
The second term on the right-hand side, at first glance also nonlinear in $\mm_h^{i+1}$, turns out to be actually linear.
Indeed, it holds that
\begin{equation*}
\mm_h^{i+1/2} \times d_t \mm_h^{i+1}
\stackrel{\eqref{mpslabel:eq:midpoint-derivative}}{=} \frac{\mm_h^{i+1} + \mm_h^i}{2}
\times \frac{\mm_h^{i+1} - \mm_h^i}{k}
= - \frac{1}{k} \, \mm_h^{i+1} \times \mm_h^i.
\end{equation*}
However, using an arbitrary off-the-shelf nonlinear solver for~\eqref{mpslabel:eq:mps},
the conservation and stability properties of Algorithm~\ref{mpslabel:alg:mps}
established in Theorem~\ref{mpslabel:thm:main}(i)--(ii) are in general lost.
Moreover, $\ppi$ can be nonlocal and non-exactly computable (e.g., for the stray field),
so that the field contribution $\ppi(\mm_h^{i+1/2})$ must be numerically approximated.
Hence, a direct implementation of Algorithm~\ref{mpslabel:alg:mps} should be
based on an inner iteration performing the solution of the nonlinear system~\eqref{mpslabel:eq:mps}
and the approximate computation of $\ppi(\mm_h^{i+1/2})$.

In the remainder of this section, we discuss and analyze
an effective treatment of the nonlocal contribution, which we combine with
two approaches for the linearization of~\eqref{mpslabel:eq:mps},
from which we will obtain two \emph{practical midpoint schemes}.

To start with, we define
the bilinear form $\aloc\colon \HH^1(\Omega) \times \HH^1(\Omega) \to \R$
by
\begin{equation} \label{mpslabel:eq:bilinearform_loc}
\aloc(\ppsi,\pphi)
= \sum_{d=1}^3 \inner[\Omega]{\Amat_d(\de_d \ppsi - \JJ_{\!d} \ppsi)}{\de_d \pphi - \JJ_{\!d} \pphi}
\quad
\text{for all } \ppsi,\pphi \in \HH^1(\Omega).
\end{equation}
We consider the local parts of the energy and the effective field given by
\begin{equation*}
\Eloc(\mm)
: =
\frac{1}{2} \, \aloc(\mm,\mm) - \inner[\Omega]{\ff}{\mm}
\stackrel{\eqref{mpslabel:eq:llg:energy}}{=}
\E(\mm) + \frac{1}{2} \inner[\Omega]{\ppi(\mm)}{\mm}
\end{equation*}
and
\begin{equation*}
\begin{split}
- \dual{\heffloc(\mm)}{\pphi}
:= \Big\langle\frac{\delta\Eloc(\mm)}{\delta \mm} , \pphi \Big\rangle
& \stackrel{\eqref{mpslabel:eq:bilinearform_loc}}{=} \aloc(\mm,\pphi) - \inner[\Omega]{\ff}{\pphi} \\
& \, \stackrel{\eqref{mpslabel:eq:bilinearform}}{=} a(\mm,\pphi) + \inner[\Omega]{\ppi(\mm)}{\pphi} - \inner[\Omega]{\ff}{\pphi} \\
& \, \stackrel{\eqref{mpslabel:eq:heff}}{=} - \dual{\heff(\mm)}{\pphi} + \inner[\Omega]{\ppi(\mm)}{\pphi},
\end{split}
\end{equation*}
respectively.
Then, for $i \in \N_0$,
we rewrite~\eqref{mpslabel:eq:mps} in terms of the new unknown $\eeta_h^i := \mm_h^{i+1/2} \in \Vh$.
Since $d_t \mm_h^{i+1} = 2(\eeta_h^i - \mm_h^i)/k$,
it is easy to see that~\eqref{mpslabel:eq:mps} is equivalent to the following problem:
First, compute $\eeta_h^i \in \Vh$ such that, for all $\pphi_h \in \Vh$, it holds that
\begin{equation} \label{mpslabel:eq:mps_eta1}
\begin{split}
& \inner[h]{\eeta_h^i}{\pphih}
+ \frac{k}{2} \inner[h]{ \eeta_h^i \times \Ph \heffloc(\eeta_h^i)}{\pphih}
+ \frac{k}{2} \inner[h]{ \eeta_h^i \times \Ph \ppi(\eeta_h^i)}{\pphih}
+ \alpha \inner[h]{\eeta_h^i \times \mm_h^i}{\pphih} \\
& \quad = \inner[h]{\mm_h^i}{\pphih}.
\end{split}
\end{equation}
Then, define 
\begin{equation} \label{mpslabel:eq:mps_eta2}
\mm_h^{i+1} := 2 \, \eeta_h^i - \mm_h^i.
\end{equation}
To treat the nonlocal contribution $\ppi(\eeta_h^i)$,
we adopt the implicit-explicit (IMEX) approach introduced in~\cite{prs2018}.
Let $\ppi_h \colon \Vh \to \Vh$ be an operator approximating $\ppi$,
assumed to be linear and uniformly bounded in $\LL^2(\Omega)$
in the sense that $\norm[L(\LL^2(\Omega);\LL^2(\Omega))]{\ppi_h} \leq \Cpi$
for some $\Cpi>0$ independent of $h$.
Moreover, we say that $\ppi_h$ is \emph{consistent} with $\ppi$, if for all $\pphi \in \LL^2(\Omega)$ and all $(\pphi_h)_{h>0} \subset \Vh$ with $\pphi_h \to \pphi$ in $\LL^2(\Omega)$ as $h \to 0$, it holds that
\begin{align}\label{mpslabel:eq:consistency}
\ppi_h(\pphi_h) \to \ppi(\pphi) \quad \text{in} \quad \LL^2(\Omega) \quad \text{as } h \to 0.
\end{align}
We define $\mm_h^{-1} := \mm_h^0$ and
\begin{equation} \label{mpslabel:eq:pi:ab}
\PPi_h(\mm_h^i, \mm_h^{i-1})
:=
\frac{3}{2}\ppi_h(\mm_h^i) - \frac{1}{2}\ppi_h(\mm_h^{i-1}).
\end{equation}
Then, in~\eqref{mpslabel:eq:mps_eta1}, we replace $\ppi(\eeta_h^i)$ with its approximation $\PPi_h(\mm_h^i, \mm_h^{i-1})$ to obtain
\begin{equation} \label{mpslabel:eq:mps_eta1_imex}
\begin{split}
& \inner[h]{\eeta_h^i}{\pphih}
+ \frac{k}{2} \inner[h]{ \eeta_h^i \times \Ph \heffloc(\eeta_h^i)}{\pphih}
+ \frac{k}{2} \inner[h]{ \eeta_h^i \times \Ph \PPi_h(\mm_h^i,\mm_h^{i-1})}{\pphih} \\
& \qquad + \alpha \inner[h]{\eeta_h^i \times \mm_h^i}{\pphih}
= \inner[h]{\mm_h^i}{\pphih}.
\end{split}
\end{equation}
In particular, the nonlocal contribution, treated explicitly, becomes independent of the unknown $\eeta_h^i$.
We now discuss two strategies to linearize~\eqref{mpslabel:eq:mps_eta1_imex} in order to arrive at two practical midpoint schemes.
To emphasize the inexact solution of~\eqref{mpslabel:eq:mps_eta1_imex} up to some accuracy $\eps > 0$, we write $\mm_{h\eps}^i$ rather than $\mm_{h}^i$ for the iterates of the practical (linearized) midpoint schemes.

\subsubsection{Constraint-preserving fixed-point iteration} \label{mpslabel:sec:fp}
We solve~\eqref{mpslabel:eq:mps_eta1_imex} with the following fixed-point iteration:
Let $\eps>0$ denote some prescribed tolerance.
Set $\eeta_h^{i,0} := \mm_{h\eps}^i$.
For $\ell \in \N_0$, given $\eeta_h^{i,\ell} \in \Vh$,
compute $\eeta_h^{i,\ell+1} \in \Vh$ such that, for all $\pphi_h \in \Vh$, it holds that
\begin{equation} \label{mpslabel:eq:mps_eta_fp}
\begin{split}
& \inner[h]{\eeta_h^{i,\ell+1}}{\pphih}
+ \frac{k}{2} \inner[h]{ \eeta_h^{i,\ell+1} \times \Ph \heffloc(\eeta_h^{i,\ell})}{\pphih}
+ \frac{k}{2} \inner[h]{ \eeta_h^{i,\ell+1} \times \Ph \PPi_h(\mm_{h\eps}^i,\mm_{h\eps}^{i-1})}{\pphih} \\
& \qquad + \alpha \inner[h]{\eeta_h^{i,\ell+1} \times \mm_{h\eps}^i}{\pphih}
= \inner[h]{\mm_{h\eps}^i}{\pphih},
\end{split}
\end{equation}
until
\begin{equation} \label{mpslabel:eq:stopping_fp}
\norm[h]{\Interp[\eeta_h^{i,\ell+1} \times \Ph (\heffloc(\eeta_h^{i,\ell+1}) - \heffloc(\eeta_h^{i,\ell}))]} \le \eps,
\end{equation}
where $\Interp[\cdot]$ denotes the vector-valued nodal interpolant.
If $\ell^* \in \N_0$ is the smallest integer for which the stopping criterion~\eqref{mpslabel:eq:stopping_fp} is satisfied,
in view of~\eqref{mpslabel:eq:mps_eta2},
the approximate magnetization at the new time-step is defined as
$\mm_{h\eps}^{i+1} := 2 \, \eeta_h^{i,\ell^*+1} - \mm_{h\eps}^i$.

In the following proposition, we collect the properties of the proposed fixed-point iteration.
The proof is postponed to Section~\ref{mpslabel:sec:proofs2}.

\begin{proposition} \label{mpslabel:prop:fp}
Let $i \in \N_0$.\\
{\textrm{(i)}}
For all $\ell \in \N_0$, the variational problem~\eqref{mpslabel:eq:mps_eta_fp} admits a unique solution $\eeta_h^{i,\ell+1} \in \Vh$.
Moreover, it holds that $\norm[\LL^{\infty}(\Omega)]{\eeta_h^{i,\ell+1}} \le 1$.\\
{\textrm{(ii)}}
If $k=o(h^2)$ as $h,k \to 0$, there exist
a contraction constant $0 < q < 1$
and thresholds $h_0, k_0 > 0$ such that,
for all $h < h_0$ and $k < k_0$, it holds that
\begin{equation} \label{mpslabel:eq:contraction}
\norm[h]{\eeta_h^{i,\ell+2} - \eeta_h^{i,\ell+1}}
\le
q \,
\norm[h]{\eeta_h^{i,\ell+1} - \eeta_h^{i,\ell}}
\quad \text{for all } \ell \in \N_0.
\end{equation}
The constants $q,h_0,k_0$ depend only on the mesh parameter $\kappa$
and the problem data.\\
{\textrm{(iii)}}
Under the assumptions of part~${\textrm{(ii)}}$,
the stopping criterion~\eqref{mpslabel:eq:stopping_fp} is met in a finite number of iterations.
If $\ell^* \in \N_0$ denotes the smallest integer for which~\eqref{mpslabel:eq:stopping_fp} is satisfied,
the new approximation $\mm_{h\eps}^{i+1} := 2 \, \eeta_h^{i,\ell^*+1} - \mm_{h\eps}^i \in \Vh$ belongs to $\Mh$.
\end{proposition}

For all $i \in \N_0$,
let $\rr_{h\eps}^i := \Ph (\heffloc(\eeta_h^{i,\ell^*+1}) - \heffloc(\eeta_h^{i,\ell^*})) \in \Vh$.
Because of the stopping criterion~\eqref{mpslabel:eq:stopping_fp},
it holds that $\norm[h]{\Interp[\mm_{h\eps}^{i+1/2} \times \rr_{h\eps}^i]} \leq \eps$.
With this definition, the proposed linearization of Algorithm~\ref{mpslabel:alg:mps}
is covered by the following algorithm.

\begin{algorithm}[practical midpoint scheme, constraint preserving fixed-point iteration] \label{mpslabel:alg:mps_fp}
Input:
$\mm_{h\eps}^0 := \mm_h^0 \in \Mh$. \\
Loop:
For all $i \in \N_0$, use the constraint preserving fixed-point iteration~\eqref{mpslabel:eq:mps_eta_fp}--\eqref{mpslabel:eq:stopping_fp} to compute $\mm_{h\eps}^{i + 1} \in \Mh$ and $\rr_{h\eps}^i \in \Vh$ with $\norm[h]{\Interp[\mm_{h\eps}^{i+1/2} \times \rr_{h\eps}^i]} \leq \eps$
such that
\begin{equation} \label{mpslabel:eq:mps:fp}
\begin{split}
\inner[h]{d_t \mm_{h\eps}^{i+1}}{\pphih}
& = - \inner[h]{ \mm_{h\eps}^{i + 1/2} \times \Ph \heff(\mm_{h\eps}^{i + 1/2})}{\pphih}
+ \alpha \inner[h]{\mm_{h\eps}^{i+1/2} \times d_t \mm_{h\eps}^{i+1}}{\pphih} \\
& \qquad + \inner[h]{\mm_{h\eps}^{i+1/2} \times [\rr_{h\eps}^i + \Ph (\ppi(\mm_{h\eps}^{i + 1/2}) - \PPi_h(\mm_{h\eps}^i,\mm_{h\eps}^{i-1}) )]}{\pphih}
\end{split}
\end{equation}
for all $\pphih \in \Vh$. \\
Output:
Sequence of approximations $\left\{\mm_{h\eps}^i\right\}_{i \in \N_0}$.
\qed
\end{algorithm}

In the following theorem,
we establish the stability and convergence of the approximations obtained with Algorithm~\ref{mpslabel:alg:mps_fp}.
The proof is postponed to Section~\ref{mpslabel:sec:proofs2}.

\begin{theorem} \label{mpslabel:thm:main:fp}
{\textrm{(i)}}
Suppose that $\mm_h^0 \in \Mh$.
If $k=o(h^2)$ as $h,k \to 0$, there exist thresholds $h_0>0$ and $k_0 > 0$ such that,
for all $h < h_0$ and $k < k_0$, \eqref{mpslabel:eq:mps:fp} admits solutions $\mm_{h\eps}^{i + 1} \in \Mh$
and $\rr_{h\eps}^i \in \Vh$ with $\norm[h]{\Interp[\mm_{h\eps}^{i+1/2} \times \rr_{h\eps}^i]} \leq \eps$ for all $i \in \N_0$.
In particular, the scheme preserves the unit-length constraint at the nodes of the triangulation
for all time-steps.
The thresholds $h_0,k_0$ depend only on the mesh parameter $\kappa$ and the problem data.\\
{\textrm{(ii)}}
Under the assumptions of part~{\textrm{(i)}}, for all $h < h_0$, $k < k_0$, and $J \in \N$,
the scheme satisfies the discrete energy identity
\begin{align} \label{mpslabel:eq:energylaw:fp}
& \E(\mm_{h\eps}^J) + \alpha k \sum_{i=0}^{J-1} \norm[h]{d_t \mm_{h\eps}^{i+1}}^2 = \E(\mm_h^0) \\
\notag & \
- k \sum_{i=0}^{J-1} \inner[h]{\mm_{h\eps}^{i+1/2} \times [\rr_{h\eps}^i + \Ph(\ppi(\mm_{h\eps}^{i + 1/2}) - \PPi_h(\mm_{h\eps}^i,\mm_{h\eps}^{i-1}))]}{\Ph \heff(\mm_{h\eps}^{i + 1/2}) - \alpha \, d_t \mm_{h\eps}^{i+1}}.
\end{align}
{\textrm{(iii)}}
Let $T>0$.
Let $J \in \N$ be the smallest integer such that $T \leq kJ$.
Under the assumptions of part~{\textrm{(i)}}, if $\eps = \OO(h)$ and $\{\mm_h^0\}_{h > 0}$ is bounded in $\HH^1(\Omega)$ as $h, \eps \to 0$,
there exist thresholds $0<h_0^* \leq h_0$, $0<k_0^* \leq k_0$, and $\eps_0^* > 0$ such that,
for all $h < h_0^*$, $k < k_0^*$, $\eps < \eps_0^*$, and $1 \leq j \leq J$,
we have the stability estimate
\begin{equation} \label{mpslabel:eq:stability:fp}
\norm[\HH^1(\Omega)]{\mm_{h\eps}^j}^2
+ k \sum_{i=0}^{j-1} \norm[\LL^2(\Omega)]{d_t \mm_{h\eps}^{i+1}}^2
\leq C.
\end{equation}
The constant $C>0$ and the thresholds $h_0^*,k_0^*,\eps_0^*$
depend only on the mesh parameter $\kappa$, the final time $T$, and the problem data.\\
{\textrm{(iv)}}
Additionally to the assumptions of part~{\textrm{(iii)}}, assume $\mm_h^0 \to \mm^0$ in $\HH^1(\Omega)$ as $h \to 0$, and suppose that $\ppi_h$ is consistent~\eqref{mpslabel:eq:consistency} with $\ppi$.
Then, there exist $\mm \in \HH^1(\Omega_T) \cap L^{\infty}(0,T;\HH^1(\Omega ; \sphere))$
and a subsequence of $\{ \mm_{h\eps k} \}$ (not relabeled)
which converges towards $\mm$ as $h,k,\eps \to 0$.
Specifically, $\mm_{h\eps k}\vert_{\Omega_T} \weakstarto \mm$ in $L^{\infty}(0,T;\HH^1(\Omega ; \sphere))$
and $\mm_{h\eps k}\vert_{\Omega_T} \weakto \mm$ in $\HH^1(\Omega_T)$
as $h,k,\eps \to 0$.
The limit function $\mm$ satisfies the conditions {\textrm{(i)}}--{\textrm{(iv)}} of Definition~\ref{mpslabel:def:weak}.
\end{theorem}

\subsubsection{Newton iteration}\label{mpslabel:sec:newton}
Based on the Newton scheme, in~\cite[Section~1.4.1]{bbnp2014} the authors employ a linearization of the nonlinear system~\eqref{mpslabel:eq:mps} in the ideal midpoint scheme with simplified effective field, i.e., without nonlocal contributions and without DMI.
Their 2D numerical experiments give hope for a less restrictive CFL condition than for the fixed-point iteration from Section~\ref{mpslabel:sec:fp}.

For three dimensional micromagnetics and considering the full effective field~\eqref{mpslabel:eq:heff}, in Section~\ref{mpslabel:sec:apply_newton_to_mps} we apply Newton's method to the nonlinear system of equations~\eqref{mpslabel:eq:mps_eta1_imex} resulting in the following iteration:
Let $\eps>0$ denote some tolerance.
Set $\eeta_h^{i,0} := \mm_{h\eps}^i$.
For $\ell \in \N_0$, given $\eeta_h^{i,\ell} \in \Vh$,
compute $\uu_h^{i,\ell} \in \Vh$ such that, for all $\pphi_h \in \Vh$, it holds that
\begin{align} \label{mpslabel:eq:mps_eta_newton}
\notag & \inner[h]{\uu_h^{i,\ell}}{\pphih}
+ \frac{k}{2} \inner[h]{ \uu_h^{i,\ell} \times \Ph (\heffloc(\eeta_h^{i,\ell}) + \PPi_h(\mm_{h\eps}^i,\mm_{h\eps}^{i-1}))}{\pphih}
+ \alpha \inner[h]{\uu_h^{i,\ell} \times \mm_{h\eps}^i}{\pphih} \\
&\quad + \frac{k}{2} \inner[h]{\eeta_h^{i,\ell} \times \Ph (\heffloc(\uu_h^{i,\ell}) - \ff)}{\pphi_h}
 = \inner[h]{\mm_{h\eps}^i - \eeta_h^{i,\ell}}{\pphih} \\
\notag &\qquad\quad - \frac{k}{2} \inner[h]{ \eeta_h^{i,\ell} \times \Ph (\heffloc(\eeta_h^{i,\ell}) + \PPi_h(\mm_{h\eps}^i,\mm_{h\eps}^{i-1}))}{\pphih}
- \alpha \inner[h]{\eeta_h^{i,\ell} \times \mm_{h\eps}^i}{\pphih},
\end{align}
and define $\eeta_h^{i,\ell+1} := \eeta_h^{i,\ell} + \uu_h^{i,\ell}$
until
\begin{equation} \label{mpslabel:eq:stopping_newton}
\norm[h]{\Interp\big[ \uu_h^{i,\ell} \times \Ph (\heffloc(\uu_h^{i,\ell}) - \ff) \big]}
\le \eps.
\end{equation}

If $\ell^* \in \N_0$ is the smallest integer for which the stopping criterion~\eqref{mpslabel:eq:stopping_newton} is satisfied,
the approximate magnetization at the new time-step is defined as
$\mm_{h\eps}^{i+1} := 2 \, \eeta_h^{i,\ell^*+1} - \mm_{h\eps}^i$.

For all $i \in \N_0$,
let $\rr_{h\eps}^i = \Interp\big[ \uu_h^{i,\ell^*} \times \Ph (\heffloc(\uu_h^{i,\ell^*}) - \ff) \big] \in \Vh$.
In view of the stopping criterion~\eqref{mpslabel:eq:stopping_newton}, it holds that $\norm[h]{\rr_{h\eps}^i} \leq \eps$.
With this definition, the proposed linearization of Algorithm~\ref{mpslabel:alg:mps} based
on the Newton method is covered by the following algorithm.

\begin{algorithm}[practical midpoint scheme, Newton iteration] \label{mpslabel:alg:mps_newton}
Input:
$\mm_{h\eps}^0 := \mm_h^0 \in \Vh$. \\
Loop:
For all $i \in \N_0$, use Newton's method~\eqref{mpslabel:eq:mps_eta_newton}--\eqref{mpslabel:eq:stopping_newton} with initial guess $\eeta_h^{i, 0} := \mm_{h\eps}^i$ to compute $\mm_{h\eps}^{i + 1} \in \Vh$ and $\rr_{h\eps}^i \in \Vh$ with $\norm[h]{\rr_{h\eps}^i} \leq \eps$ such that
\begin{equation} \label{mpslabel:eq:mps:newton}
\begin{split}
\inner[h]{d_t \mm_{h\eps}^{i+1}}{\pphih}
& = - \inner[h]{ \mm_{h\eps}^{i + 1/2} \times \Ph\heff(\mm_{h\eps}^{i + 1/2})}{\pphih}
+ \alpha \inner[h]{\mm_{h\eps}^{i+1/2} \times d_t \mm_{h\eps}^{i+1}}{\pphih}
 \\
& \quad + \inner[h]{\rr_{h\eps}^i}{\pphih} + \inner[h]{\mm_{h\eps}^{i+1/2} \times \Ph (\ppi(\mm_{h\eps}^{i + 1/2}) - \PPi_h(\mm_{h\eps}^i,\mm_{h\eps}^{i-1}) )}{\pphih}
\end{split}
\end{equation}
for all $\pphih \in \Vh$.\\
Output: Sequence of approximations $\left\{\mm_{h\eps}^i\right\}_{i \in \N_0}$.
\qed
\end{algorithm}

The results on Algorithm~\ref{mpslabel:alg:mps_newton} are stated in Lemma~\ref{mpslabel:le:newton:Linfty} ($\LL^\infty$-bound), Theorem~\ref{mpslabel:thm:main:newton} (stability), and Theorem~\ref{mpslabel:thm:main:newton:convergence} (well-posedness) below.
Compared to Algorithm~\ref{mpslabel:alg:mps_fp}, our analysis is more involved:
Precisely, for $i < J$ the proof of Theorem~\ref{mpslabel:thm:main:newton:convergence} requires $i$-independent bounds on $\norm[\LL^\infty(\Omega)]{\mm_{h\eps}^i}$ and $\E(\mm_{h\eps}^i)$ in order to guarantee well-posedness of Algorithm~\ref{mpslabel:alg:mps_newton}, while Lemma~\ref{mpslabel:le:newton:Linfty} and Theorem~\ref{mpslabel:thm:main:newton} require termination of~\eqref{mpslabel:eq:mps_eta_newton}--\eqref{mpslabel:eq:stopping_newton} so that $\mm_{h\eps}^{i+1}$ is well-defined.\\

In contrast to the fixed-point iteration~\eqref{mpslabel:eq:mps_eta_fp}--\eqref{mpslabel:eq:stopping_fp}, the Newton iteration~\eqref{mpslabel:eq:mps_eta_newton}--\eqref{mpslabel:eq:stopping_newton} does not inherently preserve discrete unit-length, i.e., in general $|\mm_{h\eps}^{i+1}(\zz)| \not= |\mm_{h\eps}^{i}(\zz)|$ for $\zz\in\NN_h$.
However, assuming well-posedness of Algorithm~\ref{mpslabel:alg:mps_newton}, in the following lemma we establish uniform $\LL^\infty(\Omega)$-boundedness of the approximations obtained with Algorithm~\ref{mpslabel:alg:mps_newton}.
\begin{lemma}\label{mpslabel:le:newton:Linfty}
Suppose $h, k, \eps > 0$, $\mm_{h}^0 \in \Mh$ and let $J \in \N$ be the smallest integer such that $T \leq kJ$.
Let $0 \le i < J$ and suppose that the Newton iteration~\eqref{mpslabel:eq:mps_eta_newton}--\eqref{mpslabel:eq:stopping_newton} in Algorithm~\ref{mpslabel:alg:mps_newton} terminates for all $0 \le n \le i$, i.e., the sequences $\{\mm_{h\eps}^n\}_{n=0}^{i + 1}, \{\rr_{h\eps}^n\}_{n=0}^{i} \subset \Vh$ are the output of Algorithm~\ref{mpslabel:alg:mps_newton} and satisfy~\eqref{mpslabel:eq:mps:newton} with $\norm[h]{\rr_{h\eps}^n} \le \eps$ for all $0 \le n \le i$. \\
{\textrm{(i)}} If $\eps = \mathcal{O}(h^{3/2})$ as $h,k,\eps \to 0$, then there exists a constant $C_\infty > 0$ and thresholds $h_0>0$, $k_0 > 0$, and $\eps_0 > 0$ such that,
for all $h < h_0$, $k < k_0$, and $\eps < \eps_0$, it holds that $\norm[\LL^\infty(\Omega)]{\mm_{h\eps}^{n+1}} \le C_\infty$ uniformly for all $0 \le n \le i$.
The thresholds $h_0,k_0, \eps_0$ depend only on the mesh parameter $\kappa$ and the problem data, while the bound $C_\infty > 0$ depends only on $\kappa$, $\eps h^{-3/2} \lesssim 1$, and the final time $T > 0$, but not on the integer $i < J$.\\
{\textrm{(ii)}} If $\eps = o(h^{3/2})$, then there holds $\norm[{L^\infty([0, t_{i+1}]\times\Omega)}]{\mathcal{I}_h(|\mm_{h\eps k}|^2) - 1} \to 0$ as $h, k, \eps \to 0$.
\end{lemma}
Assuming well-posedness of Algorithm~\ref{mpslabel:alg:mps_newton}, in the following theorem we establish the stability and convergence of the approximations obtained with Algorithm~\ref{mpslabel:alg:mps_newton}.
The proof is postponed to Section~\ref{mpslabel:sec:newton:proof}.
\begin{theorem} \label{mpslabel:thm:main:newton}
Let $T>0$ and suppose that $\{\Th\}_{h>0}$ is a $\kappa$-quasi-uniform family of triangulations.
Suppose $h, k, \eps > 0$, $\mm_{h}^0 \in \Mh$ and let $J \in \N$ be the smallest integer such that $T \leq kJ$.
Let $0 \le i < J$ and suppose that the Newton iteration~\eqref{mpslabel:eq:mps_eta_newton}--\eqref{mpslabel:eq:stopping_newton} in Algorithm~\ref{mpslabel:alg:mps_newton} terminates for all $0 \le n \le i$, i.e., the sequences $\{\mm_{h\eps}^n\}_{n=0}^{i + 1}, \{\rr_{h\eps}^n\}_{n=0}^{i} \subset \Vh$ are the output of Algorithm~\ref{mpslabel:alg:mps_newton} and satisfy~\eqref{mpslabel:eq:mps:newton} with $\norm[h]{\rr_{h\eps}^n} \le \eps$ for all $0 \le n \le i$. \\
{\textrm{(i)}} Under these assumptions, the scheme satisfies the discrete energy identity
\begin{align} \label{mpslabel:eq:energylaw:newton}
& \E(\mm_{h\eps}^{i+1}) 
+ \alpha k \sum_{n=0}^{i} \norm[h]{d_t \mm_{h\eps}^{n+1}}^2 
= \E(\mm_h^0) \\
\notag&\; - k \sum_{n=0}^{i} \inner[h]{\rr_{h\eps}^n + \mm_{h\eps}^{n+1/2} \times \Ph(\ppi(\mm_{h\eps}^{n + 1/2}) - \PPi_h(\mm_{h\eps}^n,\mm_{h\eps}^{n-1}))}{\Ph \heff(\mm_{h\eps}^{n + 1/2}) - \alpha \, d_t \mm_{h\eps}^{n+1}}.
\end{align}
{\textrm{(ii)}}
If $k = o(h^2)$, $\eps = \OO(h^{3/2})$ and $\{\mm_h^0\}_{h > 0}$ is bounded in $\HH^1(\Omega)$ as $k, h, \eps \to 0$, there exist thresholds $h_0 > 0$, $k_0 > 0$, and $0 < \eps_0 \leq \alpha$ such that, for all $h < h_0$, $k < k_0$, $\eps < \eps_0$ we have the stability estimate
\begin{equation} \label{mpslabel:eq:mps:newton:stability}
\norm[\HH^1(\Omega)]{\mm_{h\eps}^{i+1}}^2
+ k \sum_{n=0}^{i} \norm[\LL^2(\Omega)]{d_t \mm_{h\eps}^{n+1}}^2
\leq C.
\end{equation}
The constant $C>0$ and the thresholds $h_0, k_0, \eps_0$
depend only on the mesh parameter $\kappa$, the final time $T$, and the problem data.\\
{\textrm{(iii)}}
Additionally to the assumptions of part~{\textrm{(ii)}}, suppose that $\mm_h^0 \to \mm^0$ in $\HH^1(\Omega)$ as $h \to 0$, and that $\ppi_h$ is consistent~\eqref{mpslabel:eq:consistency} with $\ppi$.
Then, there exist $\mm \in \HH^1(\Omega_T) \cap L^{\infty}(0,T;\HH^1(\Omega ; \sphere))$
and a subsequence of $\{ \mm_{h\eps k} \}$ (not relabeled)
which converges towards $\mm$ as $h,k,\eps \to 0$.
Specifically, $\mm_{h\eps k} \weakstarto \mm$ in $L^{\infty}(0,T;\HH^1(\Omega ; \sphere))$
and $\mm_{h\eps k} \weakto \mm$ in $\HH^1(\Omega_T)$
as $h,k,\eps \to 0$.
The limit function $\mm$ satisfies the conditions {\textrm{(i)}}--{\textrm{(iv)}} of Definition~\ref{mpslabel:def:weak}.
\end{theorem}

The following theorem guarantees that under appropriate CFL conditions Algorithm~\ref{mpslabel:alg:mps_newton} is well-posed, which is required by Lemma~\ref{mpslabel:le:newton:Linfty} and Theorem~\ref{mpslabel:thm:main:newton}.
\begin{theorem}\label{mpslabel:thm:main:newton:convergence}
Let $T>0$.
Suppose $h, k, \eps > 0$, $\mm_h^0 \in \Mh$ and let $J \in \N$ be the smallest integer such that $T \leq kJ$.\\
{\textrm{(i)}} If $k=o(h^{7/3})$ and $\eps = \mathcal{O}(h^{3/2})$ as $h,k,\eps \to 0$, then there exist thresholds $h_0>0$, $k_0 > 0$, and $\eps_0 > 0$ such that,
for all $h < h_0$, $k < k_0$, and $\eps < \eps_0$, Algorithm~\ref{mpslabel:alg:mps_newton} is well defined, i.e., for all $i = 0, \dots, J-1$ it provides after finitely many iterations of Newton's method~\eqref{mpslabel:eq:mps_eta_newton}--\eqref{mpslabel:eq:stopping_newton} solutions $\mm_{h\eps}^{i + 1}, \rr_{h\eps}^i \in \Vh$ to~\eqref{mpslabel:eq:mps:newton} with $\norm[h]{\rr_{h\eps}^i} \leq \eps$.\\
{\textrm{(ii)}} In particular, there exists a constant $C_\newton > 0$ such that the number of Newton iterations~\eqref{mpslabel:eq:mps_eta_newton}--\eqref{mpslabel:eq:stopping_newton} required to solve~\eqref{mpslabel:eq:mps:newton} is bounded by $\log_2\log_2(C_\newton kh^{-7/2}\eps^{-1})$.
The thresholds $h_0,k_0, \eps_0$ and the constant $C_\newton$ depend only on the mesh parameter $\kappa$ and the problem data.
\end{theorem}

\subsubsection{Coupling conditions on practical midpoint schemes}
While the ideal midpoint scheme (Algorithm~\ref{mpslabel:alg:mps}) is unconditionally convergent towards a weak solution of LLG, the analysis of the practical midpoint schemes (Algorithm~\ref{mpslabel:alg:mps_fp} and Algorithm~\ref{mpslabel:alg:mps_newton}) crucially relies on CFL conditions imposed on the discretization parameters $h, k, \eps > 0$.
We conclude this section by Table~\ref{mpslabel:table:cfl_overview}, giving an overview on the imposed coupling conditions sufficient to establish a rigorous analysis of the practical midpoint schemes.
\begin{table}[h]
{\small
\begin{tabular}{|c|c|c|}
\hline
 & Fixed-point linearization & Newton linearization\\
 & Algorithm~\ref{mpslabel:alg:mps_fp} & Algorithm~\ref{mpslabel:alg:mps_newton}\\
\hline
well-posedness & $k=o(h^2)$ & $k = o(h^{7/3})$, $\eps = \mathcal{O}(h^{3/2})$\\
$\LL^\infty(\Omega)$-bound & none & $\eps = \mathcal{O}(h^{3/2})$\\
stability & $\eps = \mathcal{O}(h)$ & $\eps = \mathcal{O}(h^{3/2})$\\
convergence & $\eps = \mathcal{O}(h)$ & $\eps = \mathcal{O}(h^{3/2})$\\
\hline
total & $k = o(h^2)$, $\eps = \mathcal{O}(h)$ & $k = o(h^{7/3})$, $\eps = \mathcal{O}(h^{3/2})$\\
\hline
\end{tabular}
}
\caption{Sufficient CFL conditions for the analysis of the practical midpoint schemes of Section~\ref{mpslabel:sec:practical}.}
\label{mpslabel:table:cfl_overview}
\end{table}

%% file: sec_numerics.tex
\section{Numerical experiments} \label{mpslabel:sec:numerics}
The goal of this section is threefold:
First, in Section~\ref{mpslabel:sec:fert} we verify the extension of the midpoint scheme to the DMI contribution and its correct implementation by simulating an experiment on skyrmion dynamics from~\cite{scrtf2013}.
The simulation results with the midpoint scheme are compared to theirs and to simulations with the tangent plane scheme from~\cite{hpprss2019}.
In Section~\ref{mpslabel:sec:critical_D}, we introduce a variation of the experiment from~\cite{scrtf2013} in order to compare reliability of the midpoint scheme to the generally cheaper tangent plane schemes in simulating sensitive skyrmion dynamics susceptible to slight (artificial) disturbances.
By doing this, we emphasize the advantages of discrete energy conservation realized by the midpoint scheme.
Finally, in an academic setting the CFL conditions arising from our analysis sufficient to prove well-posedness of the practical midpoint schemes are experimentally verified in Section~\ref{mpslabel:sec:numerics_cfl}.
In particular, the numerical CFL study hints that the CFL condition $k = o(h^{7/3})$ derived for the practical midpoint scheme based on the Newton iteration is likely pessimistic and might be weakened to $k = o(h^2)$ with a sharper analysis.
Moreover, we compare the number of iterations in the nonlinear solvers of the two practical midpoint schemes, as well as the impact of the solver accuracy $\eps > 0$ on the deviation from discrete unit-length.
All experiments in this section were performed with Commics~\cite{commics,prsehhsmp2020}.

\subsection{Stability of isolated skyrmions in nanodisks} \label{mpslabel:sec:fert}
To validate the extension of the midpoint scheme incorporating the DMI contribution, we reproduce a numerical experiment from~\cite{scrtf2013} for both the practical midpoint scheme based on the constraint preserving fixed-point iteration (Algorithm~\ref{mpslabel:alg:mps_fp}) and the practical midpoint scheme based on Newton's method (Algorithm~\ref{mpslabel:alg:mps_newton}).
There, the relaxed states of a thin nanodisk of diameter \SI{80}{\nano\meter} (aligned with $x_1 x_2$-plane) and thickness \SI{0.4}{\nano\meter} ($x_3$-direction) centered at $(0,0,0)$ for different values of the DMI constant are investigated.
The effective field consists of exchange interaction, perpendicular uniaxial anisotropy, interfacial DMI, and stray field, i.e.,
\begin{equation*}
\Heff(\mm)
= \frac{2A}{\mu_0 \Ms} \Lapl\mm
+ \frac{2 K}{\mu_0 \Ms} (\axis\cdot\mm)\axis
- \frac{2D}{\mu_0 \Ms}
\begin{pmatrix}
- \de_1 m_3 \\
- \de_2 m_3 \\
\de_1 m_1 + \de_2 m_2
\end{pmatrix}
+ \Hs(\mm)\,.
\end{equation*}
The involved material parameters mimic those of cobalt:
$\Ms=$ \SI{5.8e5}{\ampere\per\meter}, $\alpha=$ \num{0.3}, $A=$ \SI{1.5e-11}{\joule\per\meter}, $K=$ \SI{8e5}{\joule\per\meter\cubed}, and $\axis=(0,0,1)$.
For the DMI constant, the range $D=$ \num{0}, \num{1}, \dots, \num{8} \si{\milli\joule\per\square\meter} is considered.
The initial condition is a skyrmion-like state, i.e., given $r = \sqrt{x_1^2 + x_2^2}$, we define $\mm^0(\xx)=(0,0,-1)$ if $r \in [0,15]$ \si{\nano\meter} and $\mm^0(\xx)=(0,0,1)$ if $r \in (15,40]$ \si{\nano\meter}.
For all simulations we choose $T=$ \SI{1}{\nano\second}, which experimentally turns out to be a sufficiently large time to relax the system.
The computational domain is discretized by a regular partition generated by Netgen~\cite{ngsolve} consisting of \num{34596} tetrahedra and \num{11797} vertices, which corresponds to a prescribed mesh size of \SI{1}{\nano\meter}.
For the time discretization, we consider a uniform partition of the time interval $(0,T)$ with a time-step size of \SI{2.5}{\femto\second}.
We note that the time-step size has to be chosen considerably smaller than, e.g., for (different variants of) the tangent plane scheme; see our previous work~\cite[Section~4.3]{hpprss2019}.
This is due to the more restrictive CFL conditions required for convergence of the nonlinear solvers in the practical midpoint schemes; see Theorem~\ref{mpslabel:thm:main:fp} and Theorem~\ref{mpslabel:thm:main:newton:convergence}.
The accuracy for the nonlinear solver is chosen as $\eps = 10^{-8}$.
\par
The stable state is a quasi-uniform ferromagnetic state for the values $D=$ \num{0}, \num{1}, \num{2} \si{\milli\joule\per\square\meter}, a skyrmion for the values $D=$ \num{3}, \dots, \num{6} \si{\milli\joule\per\square\meter}, and a multidomain state (target skyrmion) for the values $D=$ \num{7}, \num{8} \si{\milli\joule\per\square\meter}; see Figure~\ref{mpslabel:fig:relaxed:skyrmion-like}.
The skyrmion size, i.e., the diameter of the circle $\{ m_3 = 0 \}$ in the $x_1 x_2$-plane, increases from the minimum value of circa \SI{14}{\nano\meter} for $D=$ \SI{3}{\milli\joule\per\square\meter} to the maximum value of circa \SI{48}{\nano\meter} for $D=$ \SI{6}{\milli\joule\per\square\meter}.
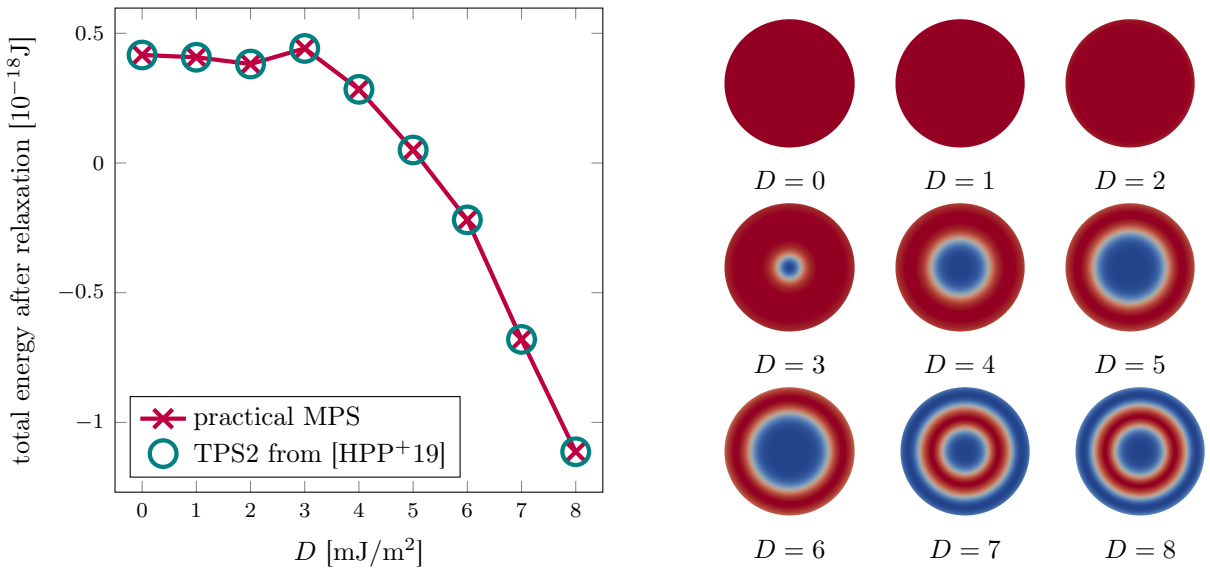
\begin{figure}[h]
\captionsetup[subfigure]{labelformat=empty}
\centering
\begin{minipage}{0.575\textwidth}
\hfill
\begin{subfigure}{\textwidth}
\begin{tikzpicture}
\pgfplotstableread{plots/fert/energy_at_1ns.dat}{\data}
\pgfplotstableread{plots/fert/tps2_data.dat}{\tpsdata}
\begin{axis}[
width = 80mm,
height = 80mm,
xlabel={\footnotesize $D$ [\si{\milli\joule\per\square\meter}]},
ylabel={\footnotesize total energy after relaxation [$10^{-18}\si{\joule}$]},
xmin=-0.5,
xmax=8.5,
xtick={0, 1, 2, 3, 4, 5, 6, 7, 8},
legend style={legend pos=south west, legend cell align = left},
]
\addplot[purple, ultra thick, mark=x, mark size=5] table[x=D, y=energy]{\data};
\addplot[teal, ultra thick, only marks, mark=o, mark size=5] table[x=D, y=energy]{\tpsdata};
\legend{\footnotesize practical MPS,\footnotesize TPS2 from \cite{hpprss2019}}
\end{axis}
\end{tikzpicture}
\end{subfigure}
\end{minipage}%
\hfill
\begin{minipage}{0.4\textwidth}
\begin{subfigure}[b]{0.275\textwidth}
\includegraphics[width=\textwidth]{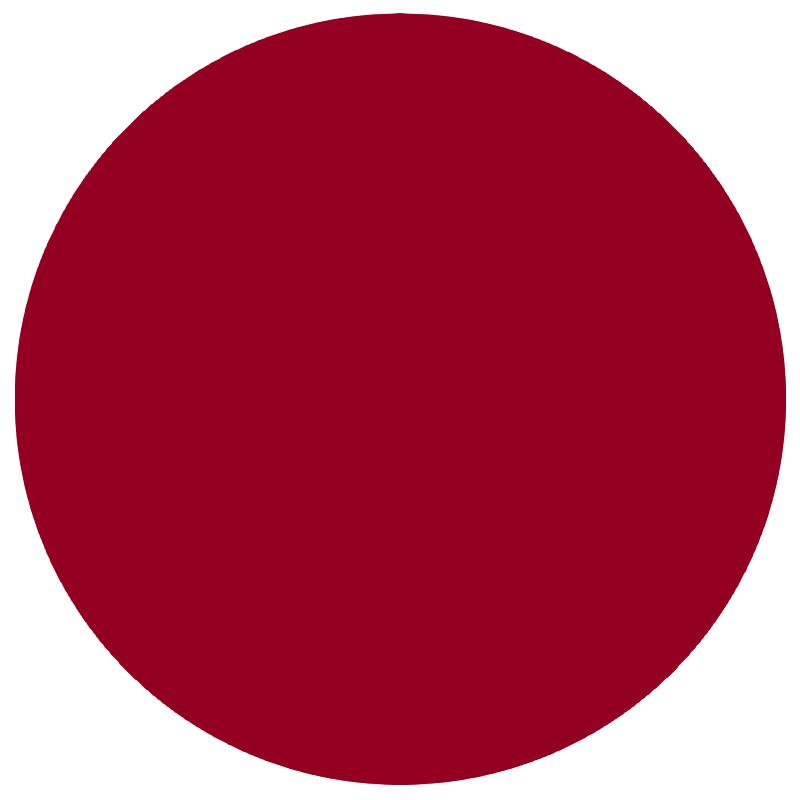}
\caption{\footnotesize $D=0$}
\end{subfigure}
\hfill
\begin{subfigure}[b]{0.275\textwidth}
\includegraphics[width=\textwidth]{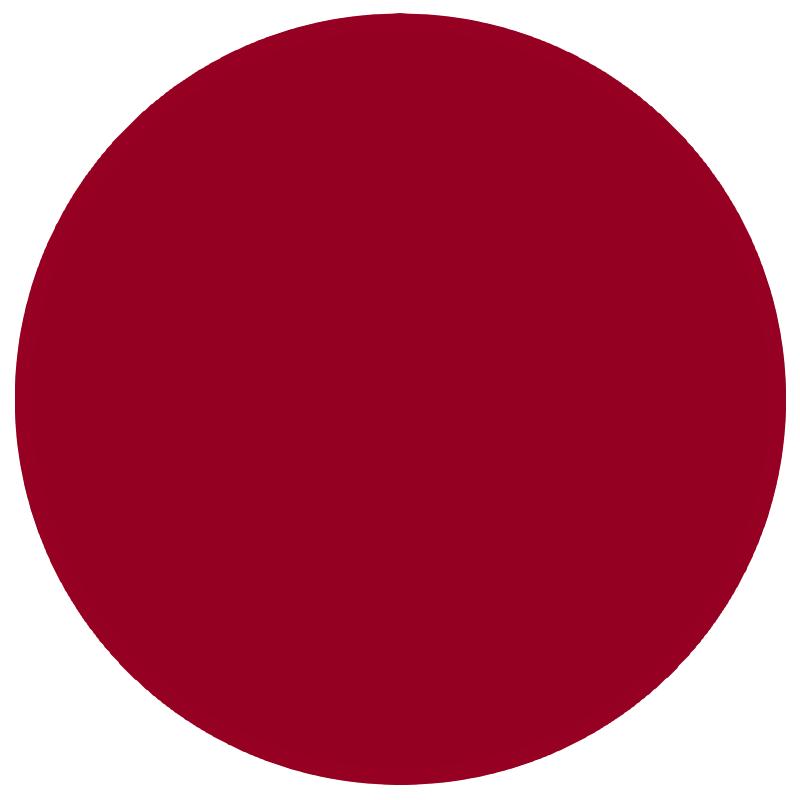}
\caption{\footnotesize $D=1$}
\end{subfigure}
\hfill
\begin{subfigure}[b]{0.275\textwidth}
\includegraphics[width=\textwidth]{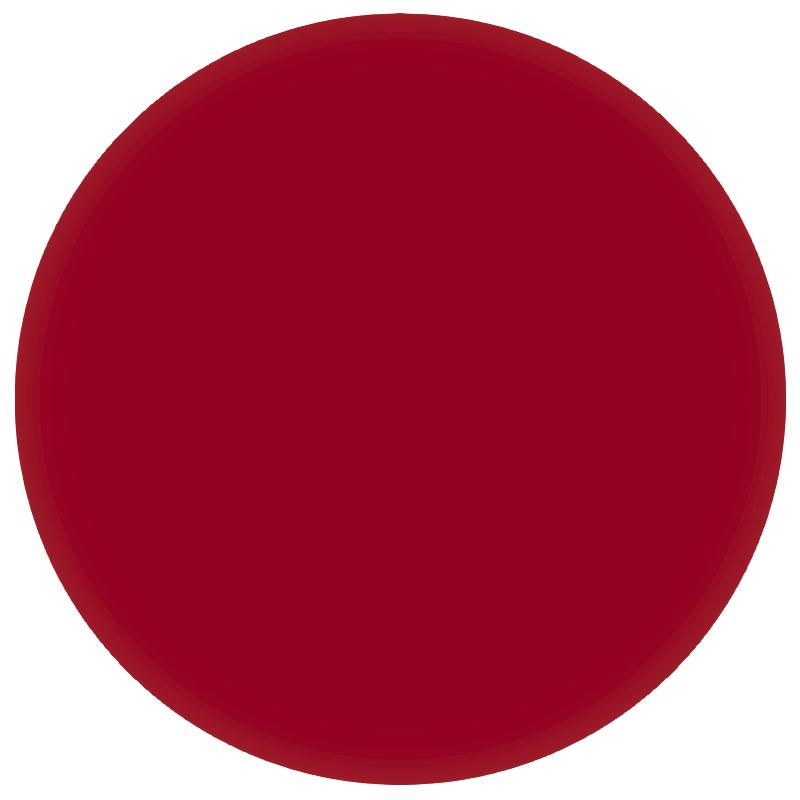}
\caption{\footnotesize $D=2$}
\end{subfigure}
\hfill
\newline
\hfill
\begin{subfigure}[b]{0.275\textwidth}
\includegraphics[width=\textwidth]{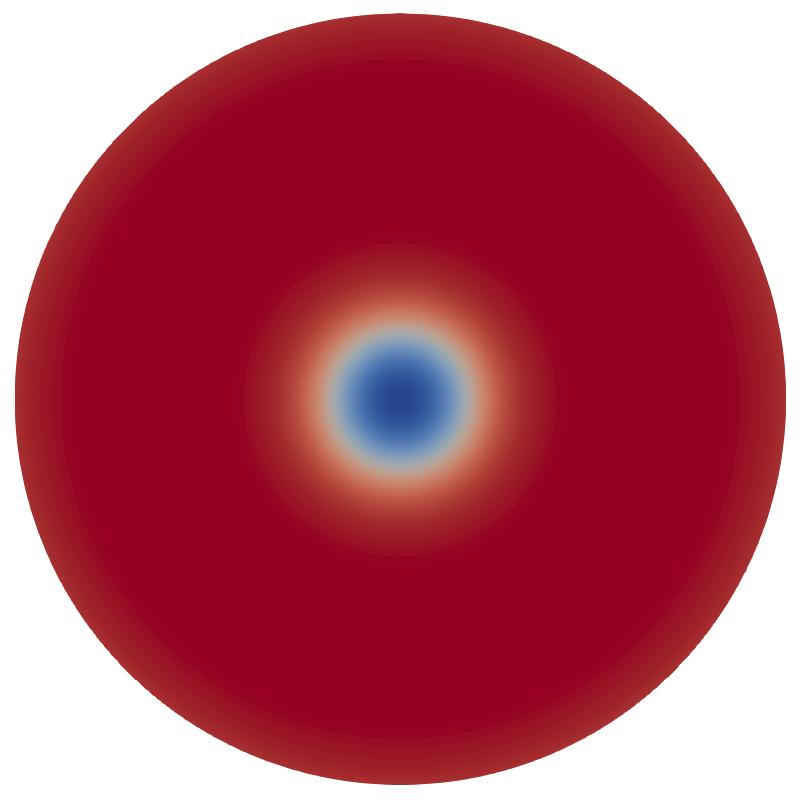}
\caption{\footnotesize $D=3$}
\end{subfigure}
\hfill
\begin{subfigure}[b]{0.275\textwidth}
\includegraphics[width=\textwidth]{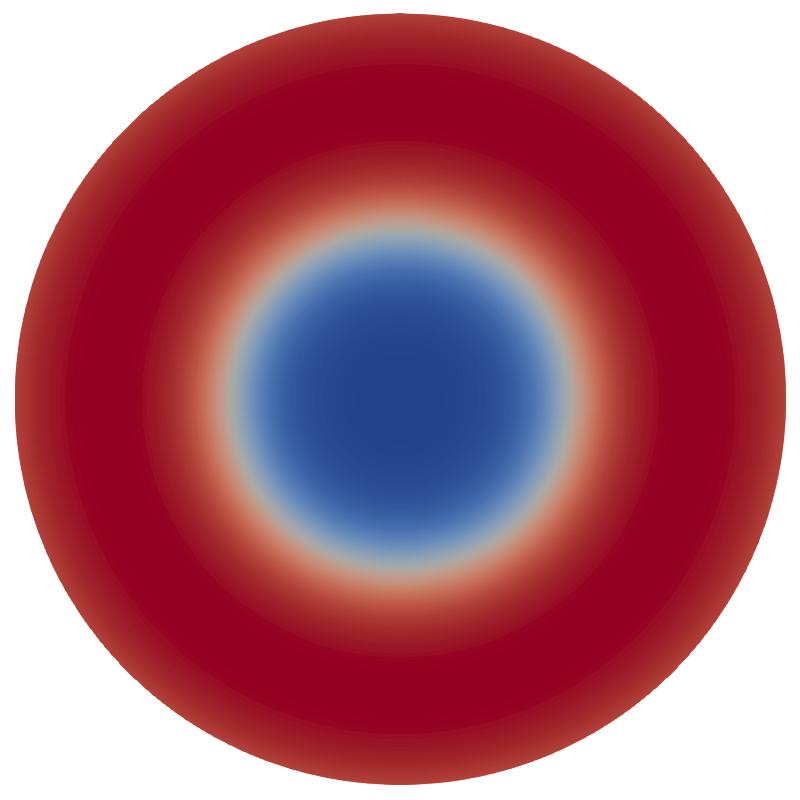}
\caption{\footnotesize $D=4$}
\end{subfigure}
\hfill
\begin{subfigure}[b]{0.275\textwidth}
\includegraphics[width=\textwidth]{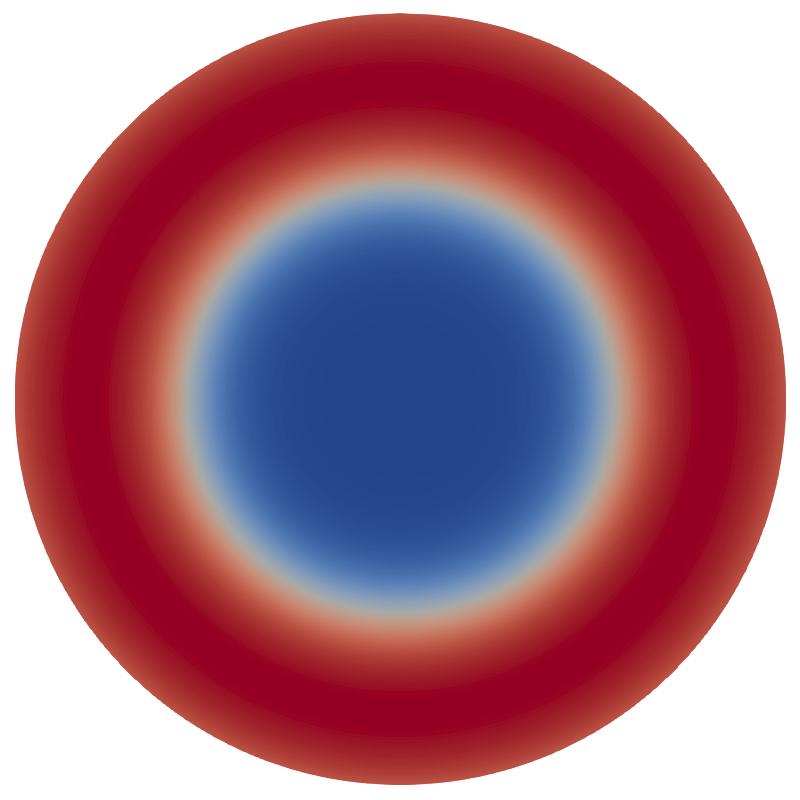}
\caption{\footnotesize $D=5$}
\end{subfigure}
\hfill
\newline
\hfill
\begin{subfigure}[b]{0.275\textwidth}
\includegraphics[width=\textwidth]{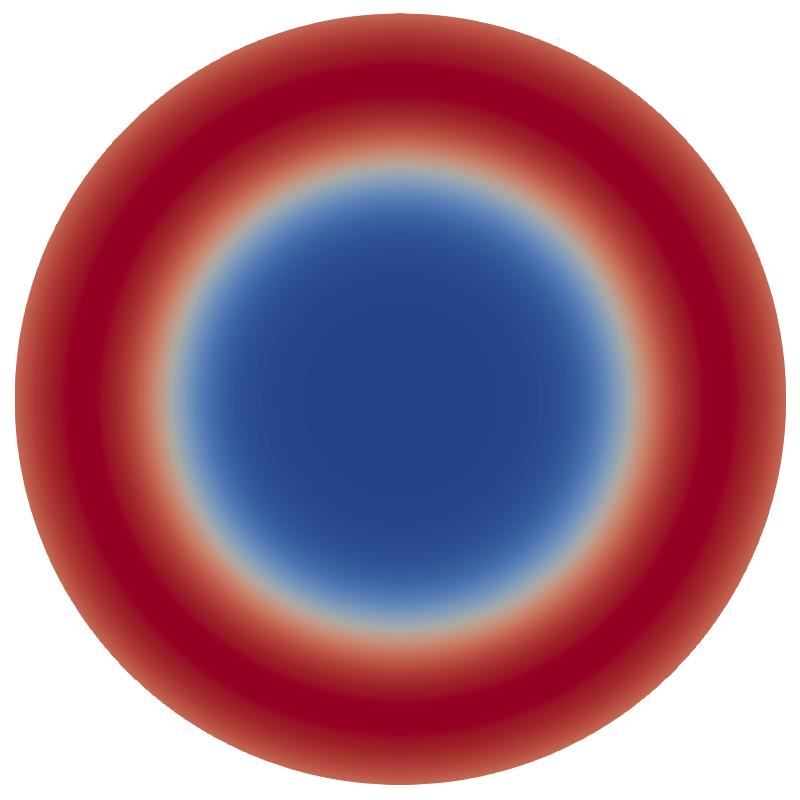}
\caption{\footnotesize $D=6$}
\end{subfigure}
\hfill
\begin{subfigure}[b]{0.275\textwidth}
\includegraphics[width=\textwidth]{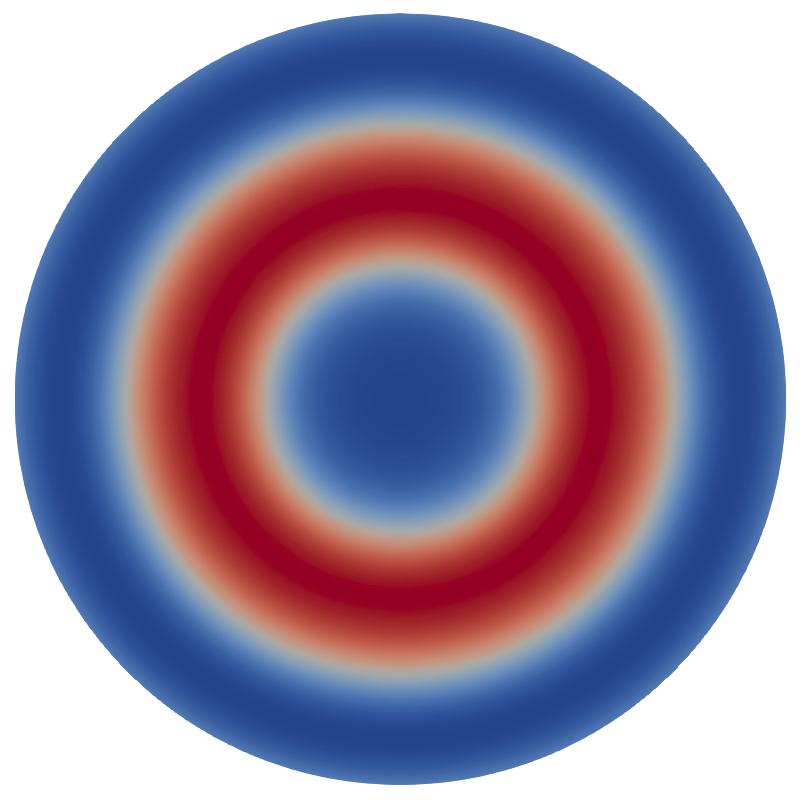}
\caption{\footnotesize $D=7$}
\end{subfigure}
\hfill
\begin{subfigure}[b]{0.275\textwidth}
\includegraphics[width=\textwidth]{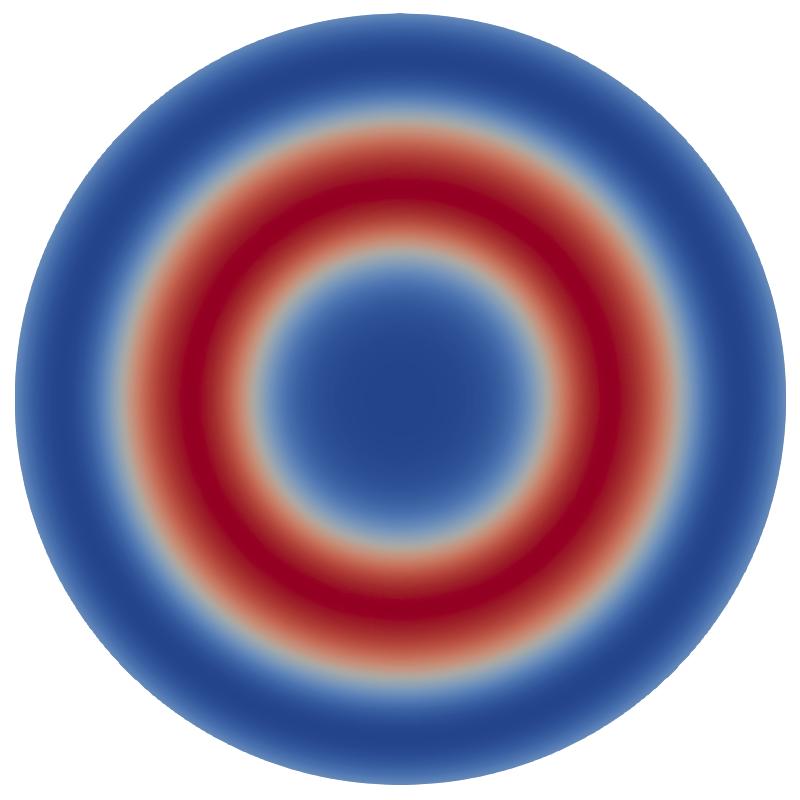}
\caption{\footnotesize $D=8$}
\end{subfigure}
\end{minipage}%
\hfill
\caption[Experiment of Section~\ref{mpslabel:sec:fert}:
Final energies and 2D views of the relaxed state for different DMI values]{Experiment of Section~\ref{mpslabel:sec:fert}.
Left: Final energy for different values of the DMI constant.
Right: Magnetization component $\mm_3$ ranging from $-1$ (blue) to $+1$ (red) of the relaxed state for different values of the DMI constant (in \si{\milli\joule\per\square\meter}).
The results computed with either of the practical midpoint schemes, Algorithm~\ref{mpslabel:alg:mps_fp} or Algorithm~\ref{mpslabel:alg:mps_newton}, coincide.}
\label{mpslabel:fig:relaxed:skyrmion-like}
\end{figure}
\par
In Figure~\ref{mpslabel:fig:relaxed:skyrmion-like}, the relaxed states computed with the practical midpoint scheme for different values of the DMI constant are given.
The energy values and the magnetization profiles are in perfect quantitative and qualitative agreement with those reported in~\cite[Figure~1]{scrtf2013} and~\cite[Section~4.2]{hpprss2019}.
This validates both the extension of the midpoint scheme to DMI energy contributions and its implementation in Commics~\cite{commics,prsehhsmp2020}.

\subsection{Reliable schemes for energy sensitive dynamics}\label{mpslabel:sec:critical_D}
We recall the discrete energy equality~\eqref{mpslabel:eq:stability} achieved by the ideal midpoint scheme
\begin{equation*} 
\E(\mm_h^J) + \alpha k \sum_{i=0}^{J-1} \norm[h]{d_t \mm_h^i}^2
= \E(\mm_h^0)\,.
\end{equation*}
Differently, for the first-order tangent plane scheme~\cite{alouges2008a} we recite from \cite[Proposition~2]{hpprss2019} the discrete energy inequality
\begin{equation} \label{mpslabel:eq:tps1:energy}
\E(\mm_h^J) + (\alpha - Ck/h) k \sum_{i=0}^{J-1} \norm[h]{\vv_h^i}^2 + \lex^2(\theta-1/2)k^2\sum_{i=0}^{J-1} \norm[\LL^2(\Omega)]{\Grad \vv_h^i}^2
\le \E(\mm_h^0)\,,
\end{equation}
where $\vv_h^i$ denotes the discrete time derivative computed in the $i$-th time-step of the tangent plane scheme to define the update $\mm_h^{i+1}(\zz) = (\mm_h^i(\zz) + k\vv_h^i(\zz)) / |\mm_h^i(\zz) + k\vv_h^i(\zz)| \in \sphere$ for all $\zz \in \Nh$.
We note that the generic constant $C > 0$ in~\eqref{mpslabel:eq:tps1:energy} stems from an inverse estimate used in the analysis of~\cite{hpprss2019} to control the discrete energy in presence of a DMI energy contribution.
The third term on the left-hand side in~\eqref{mpslabel:eq:tps1:energy} corresponds to artificial damping introduced by implicit treatment in time of the Laplacian for $1/2 < \theta \le 1$, while the inequality (instead of equality) is a result of the nodal projection in each time-step. 
As a third integrator we consider the (almost) second-order tangent plane scheme from~\cite{akst2014}, which provides a discrete energy inequality, which, although not identical to~\eqref{mpslabel:eq:tps1:energy}, introduces similar artificial energy dissipation due to implicit treatment of the Laplacian and the nodal projection update.
For the second-order schemes, i.e., for the midpoint scheme and the second-order tangent plane scheme, an IMEX treatment of the lower-order terms is employed, which results in a perturbation of order $\mathcal{O}(k^2)$ of the respective discrete energy identity~\cite{dpprs2017}.
While the discrete energy identity for the midpoint scheme mimics the continuous law
\begin{align*}
\E(\mm(\tau)) + \alpha\int_0^\tau \norm[\LL^2(\Omega)]{\partial_t \mm(t)}^2 \dt 
= \E(\mm^0) \,,
\end{align*}
due to the severe CFL condition $k = o(h^2)$ the practical midpoint schemes are very restrictive on the time-step size.
In contrast, the tangent plane integrators allow for considerably larger time-step sizes, but introduce artificial damping to the system.
Hence, we expect decreased reliability of the tangent plane integrators for accurately simulating processes, which are particularly sensitive to slight inaccuracies in the discrete energy evolution.

To quantify the effects of this artificial damping introduced by the tangent plane integrators, we extend the experiment of Section~\ref{mpslabel:sec:fert}:
Considering the different relaxed states in Figure~\ref{mpslabel:fig:relaxed:skyrmion-like}(right), one infers that between $D = \SI{2}{\milli\joule\per\square\metre}$ and $D = \SI{3}{\milli\joule\per\square\metre}$ there is a (qualitative) discontinuity, corresponding to a jump in Figure~\ref{mpslabel:fig:relaxed:skyrmion-like}(left) if the resolution on the $D$-axis was increased.
Analogously, this applies to the interval from $D = \SI{6}{\milli\joule\per\square\metre}$ to $D = \SI{7}{\milli\joule\per\square\metre}$.
The goal of this experiment is the determination of the points of transition $D_{crit}^{2-3}$ and $D_{crit}^{6-7}$ from the quasi-uniform relaxed state to the skyrmion state between $D = \SI{2}{\milli\joule\per\square\metre}$ and $D = \SI{3}{\milli\joule\per\square\metre}$, as well as from the skyrmion state to the target skyrmion state between $D = \SI{6}{\milli\joule\per\square\metre}$ and $D = \SI{7}{\milli\joule\per\square\metre}$, respectively.
We will evaluate and compare the reliability of the midpoint scheme (MPS), the first-, and the second-order tangent plane scheme (TPS1 and TPS2) in determining $D_{crit}^{2-3}$ and $D_{crit}^{6-7}$.
For all three schemes, dynamics are simulated with identical parameters:

We consider the fixed mesh from Section~\ref{mpslabel:sec:fert}.
Although this mesh does not satisfy the so-called \emph{angle condition} ensuring validity of~\eqref{mpslabel:eq:tps1:energy}, stability of the tangent plane integrators is still recovered for the smaller time-step sizes meeting $k=o(h^2)$ in this experiment; see~\cite[(15) and Remark~3(iv)]{hpprss2019}.
To narrow down the critical values $D_{crit}^{2-3}$ and $D_{crit}^{6-7}$, we simulate the relaxation dynamics for different values of the DMI constant $D$ corresponding to a resolution of $\SI{0.0025}{\milli\joule\per\square\metre}$ as seen in Figure~\ref{mpslabel:fig:critical_D}.
For each of the integrators and all considered DMI constants $D$, we relax the initial state using time-step sizes $k = \SI{1/100}{\pico\second}, \SI{1/200}{\pico\second}, \SI{1/400}{\pico\second}, \SI{1/800}{\pico\second}, \SI{1/1600}{\pico\second}$, where the two largest time-step sizes are omitted for the midpoint scheme because experimentally they do not fulfill the CFL constraint $k = o(h^2)$, i.e., neither of the nonlinear solvers converges for $k = \SI{1/100}{\pico\second}, \SI{1/200}{\pico\second}$.
We expect the simulations to be more and more accurate as the time-step size $k > 0$ decreases.
The accuracy for the nonlinear solver is chosen as $\eps = 10^{-8}$.

The results of this experiment displayed in Figure~\ref{mpslabel:fig:critical_D} show a sharp transition $D_{crit}^{2-3}$ between the uniform state and the skyrmion state.
There is no sharp transition from the skyrmion state to the target skyrmion state, as the experiment reveals a small interval of DMI parameters $D$ for which relaxation leads to states we call \emph{broken (symmetry) states} --- neither a skyrmion nor a target skyrmion; see Figure~\ref{mpslabel:fig:energy_Dcrit} for a compilation of simulation details on this interval of broken states.
While for the tangent plane integrators the determined transition value $D_{crit}^{2-3}$ and the transition interval of broken states clearly show a dependence on the used time discretization $k > 0$, the results for the midpoint scheme are robust and, in particular, are identical for all investigated time-step sizes.
We draw the conclusion that the varying transition thresholds obtained for decreasing time-step size $k > 0$ by simulations with either of the tangent plane integrators are a consequence of the artificial energy dissipation quantified in~\eqref{mpslabel:eq:tps1:energy}.

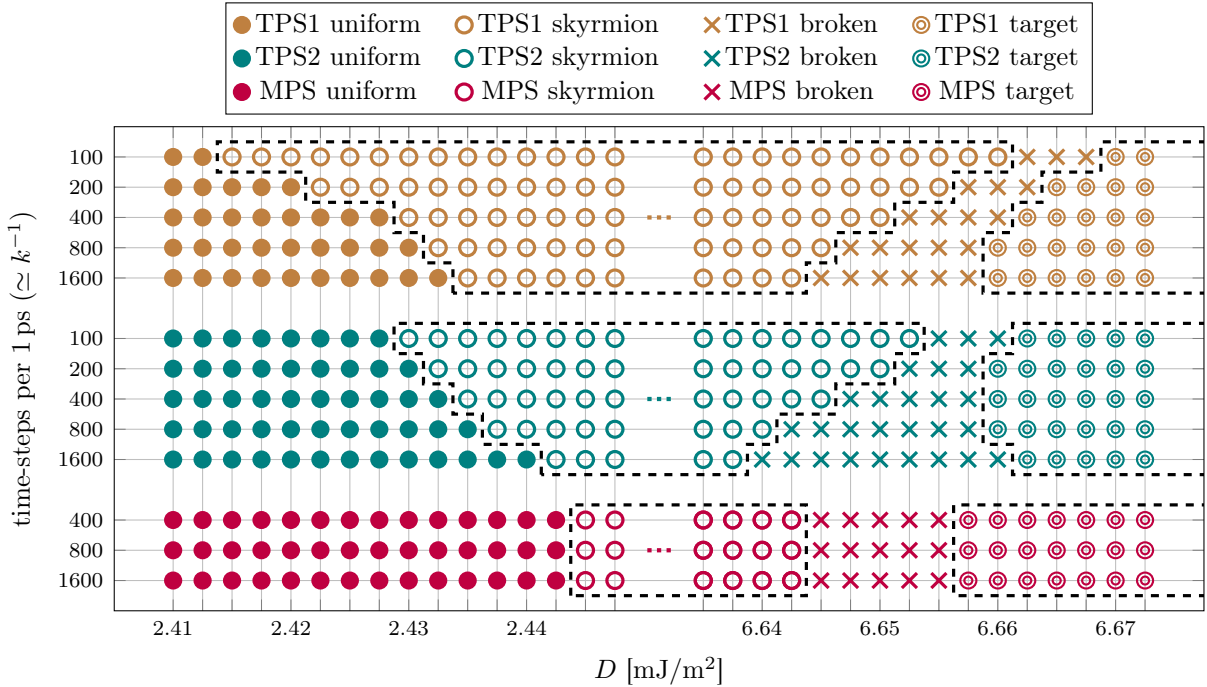
\begin{figure}[h]
\input{plots/merged_interval_plot.tex}
\caption[Experiment of Section~\ref{mpslabel:sec:critical_D}: Reliability of schemes in critical transition parameter detection]{Experiment of Section~\ref{mpslabel:sec:critical_D}.
Each marker corresponds to one simulation carried out with one of the three integrators, for a DMI parameter $D$ with one particular time-step size $k$.
The shape of a marker characterizes the qualitative result after relaxing the skyrmion-like initial state until the equilibrium state is reached, i.e., whether a quasi-uniform, a skyrmion-like, a broken unsymmetrical, or a target skyrmion state is obtained.
The results computed with either of the practical midpoint schemes, Algorithm~\ref{mpslabel:alg:mps_fp} or Algorithm~\ref{mpslabel:alg:mps_newton}, coincide.}
\label{mpslabel:fig:critical_D}
\end{figure}
\begin{figure}[h]
\centering
\begin{subfigure}{0.48\textwidth}
\begin{tikzpicture}
\pgfplotstableread{plots/fert/energy_at_1ns.dat}{\slike}
\pgfplotstableread{plots/energy_zoom_tps2_k1by100.dat}{\zoom}
\begin{axis}[
width = 75mm,
height = 60mm,
xlabel={\footnotesize$D$ [\si{\milli\joule\per\square\meter}]},
ylabel={\footnotesize final energy [$10^{-18}\si{\joule}$]},
xmin=5.8,
xmax=7.2,
xtick={6.0, 6.1, 6.2, 6.3, 6.4, 6.5, 6.6, 6.7, 6.8, 6.9, 7.0},
xticklabels={6, , 6.2, , 6.4, , 6.6, , 6.8, , 7},
legend style={legend pos=south west, legend cell align = left}
]
\addplot[teal, ultra thick, mark=*, mark size=2] table[x=D, y=energy]{\zoom};
\addplot[gray, ultra thick, dashed, mark=o, mark size=4, mark options={solid}] table[x=D, y=energy]{\slike};
\legend{
\footnotesize higher resolution,
\footnotesize interpolated
}
\end{axis}
\end{tikzpicture}
\end{subfigure}
\begin{subfigure}{0.48\textwidth}
\begin{tikzpicture}
\pgfplotstableread{plots/plot_energy.dat}{\data}
\begin{axis}[
width = 75mm,
height = 60mm,
xlabel={\footnotesize Time [\si{\nano\second}]},
xmin=-0.05,
xmax=1.1,
ymin=-5.6e-1,
ymax=-2.8e-1,
]
\addplot[red, dashed, thick] table[x=t, y=D_below_skyrmion]{\data};
\addplot[purple, thick] table[x=t, y=D_skyrmion]{\data};
\addplot[teal, thick] table[x=t, y=D_broken1]{\data};
\addplot[black, thick] table[x=t, y=D_broken2]{\data};
\addplot[blue, thick] table[x=t, y=D_broken3]{\data};
\addplot[green, thick] table[x=t, y=D_target]{\data};
\addplot[brown, dashed, thick] table[x=t, y=D_above_target]{\data};
\legend{\footnotesize$D=6.6500$, \footnotesize$D=6.6525$, \footnotesize$D=6.6550$, \footnotesize$D=6.6575$, \footnotesize$D=6.6600$, \footnotesize$D=6.6625$, \footnotesize$D=6.6650$}
\end{axis}
\end{tikzpicture}
\end{subfigure}
\medskip
\medskip
\captionsetup[subfigure]{labelformat=empty}
\centering
\hfill
\begin{subfigure}[b]{0.15\textwidth}
\includegraphics[width=\textwidth]{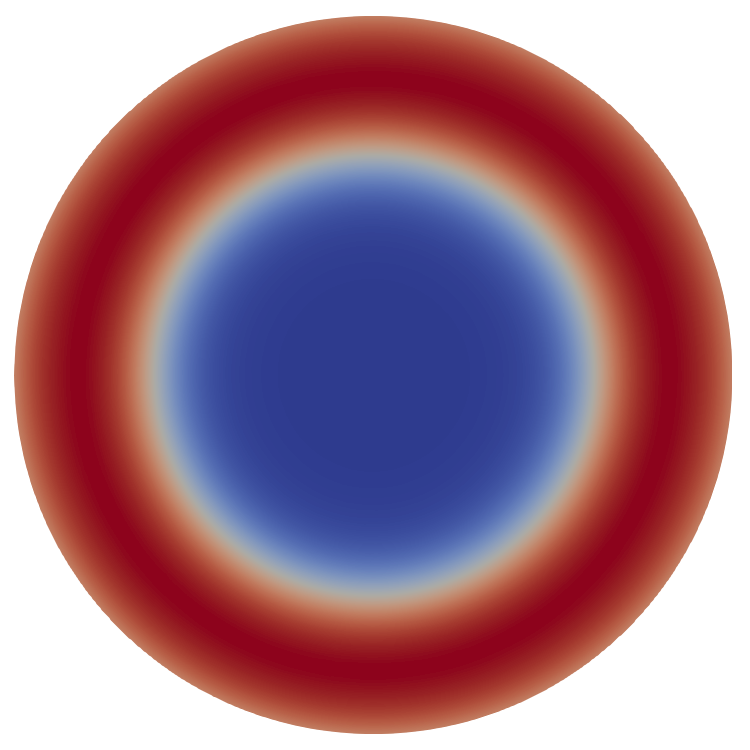}
\caption{\footnotesize$D = 6.6525$}
\end{subfigure}
\hfill
\begin{subfigure}[b]{0.15\textwidth}
\includegraphics[width=\textwidth]{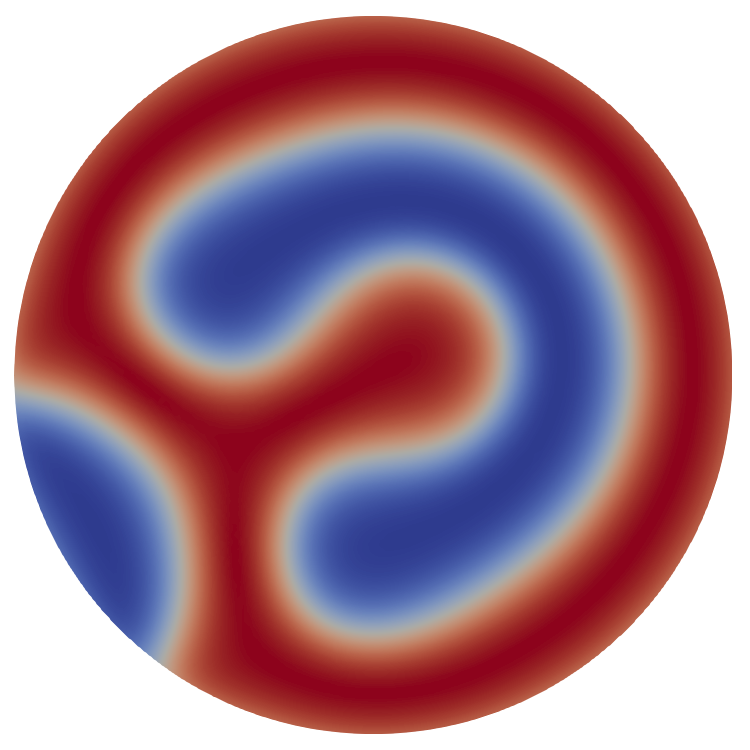}
\caption{\footnotesize$D = 6.655$}
\end{subfigure}
\hfill
\begin{subfigure}[b]{0.15\textwidth}
\includegraphics[width=\textwidth]{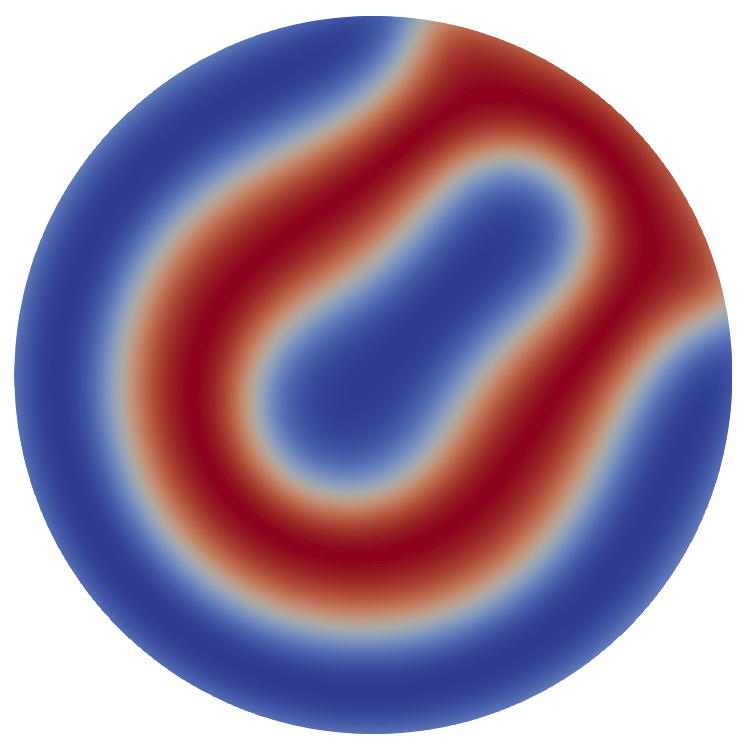}
\caption{\footnotesize$D = 6.6575$}
\end{subfigure}
\hfill
\begin{subfigure}[b]{0.15\textwidth}
\includegraphics[width=\textwidth]{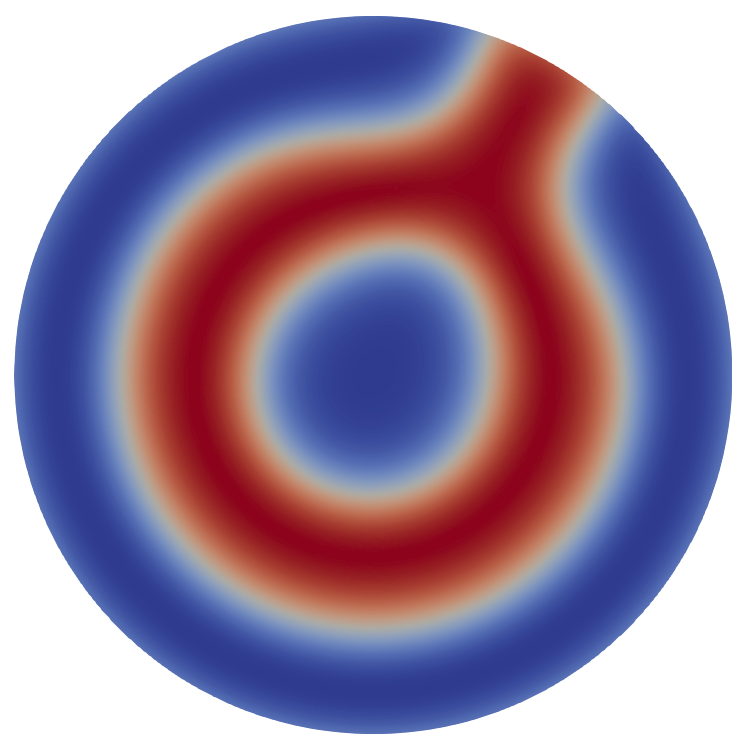}
\caption{\footnotesize$D = 6.66$}
\end{subfigure}
\hfill
\begin{subfigure}[b]{0.15\textwidth}
\includegraphics[width=\textwidth]{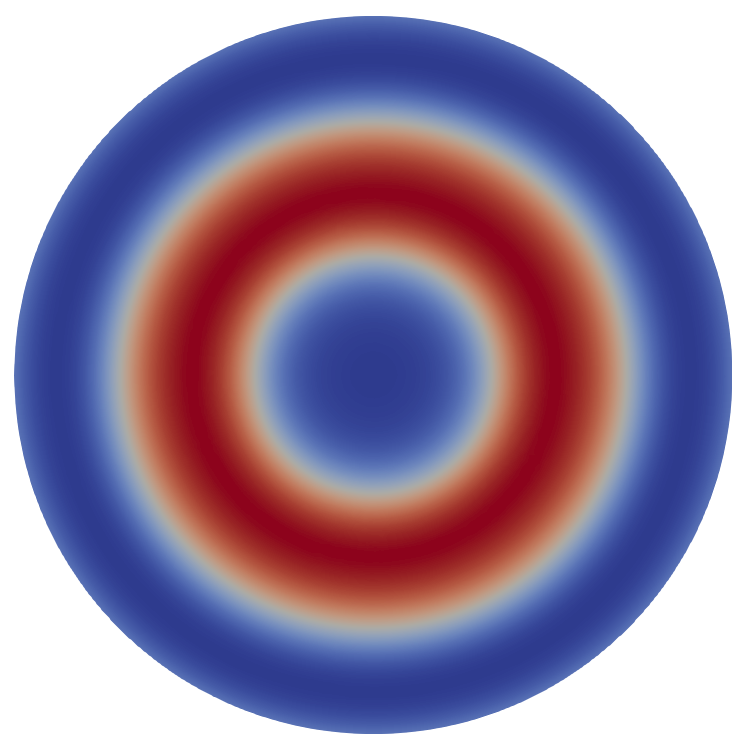}
\caption{\footnotesize$D = 6.6625$}
\end{subfigure}
\hfill
\caption[Experiment of Section~\ref{mpslabel:sec:critical_D}: Details on final energies, energy evolution, and 2D view of broken symmetry states]{Experiment of Section~\ref{mpslabel:sec:critical_D}.
Details on the simulations for critical transition values between $D=\SI{6}{\milli\joule\per\square\metre}$ and $D=\SI{7}{\milli\joule\per\square\metre}$ with the second-order tangent plane scheme and time-step size $k = \SI{1/100}{\pico\second}$.
Left: Critical area of Figure~\ref{mpslabel:fig:relaxed:skyrmion-like}(left) recomputed with higher resolution reveals the predicted jump.
Right: Evolution of the total energy for different DMI constants $D$ emphasizing the transition dynamics.
Bottom: Relaxed states colored by $\mm_3$ ranging from $-1$ (blue) to $+1$ (red).
We observed that any relaxed state with broken symmetry obtained in the experiment of this section (marked with a cross $\times$ in Figure~\ref{mpslabel:fig:critical_D}), qualitatively coincides with one of the three broken-symmetry states obtained by TPS2 and $k=\SI{1/100}{\pico\second}$ displayed here.}
\label{mpslabel:fig:energy_Dcrit}
\end{figure}

We conclude that the tangent plane schemes are preferable for uncritical simulations as in Section~\ref{mpslabel:sec:fert} or~\cite[Section~4.2]{hpprss2019}, where small deteriorations of the energy are acceptable, as they lead to already accurate results for much coarser time discretizations.
However, when it comes to the simulation of dynamics, which are very sensitive to small inaccuracies and crucially depend on an accurate energy evolution, the midpoint scheme yields the most reliable results.

\subsection{Numerical study on the CFL conditions}\label{mpslabel:sec:numerics_cfl}
Our results from Section~\ref{mpslabel:sec:fp} and Section~\ref{mpslabel:sec:newton}, respectively, provide sufficient CFL conditions guaranteeing well-posedness and stability of the practical midpoint schemes in Theorem~\ref{mpslabel:thm:main:fp}{\textrm{(i)}} and Theorem~\ref{mpslabel:thm:main:newton:convergence}.
In this section we investigate whether the CFL conditions arising from theory are also necessary in practice, or if they are technical artifacts possibly caused by unsharp estimates.

We consider the unit cube $\Omega \subset \R^3$ centered at the origin.
Steered by the exchange-only  effective field $\heff(\mm) = \lex^2 \Lapl\mm$, the so-called initial \emph{hedgehog} state $\mm^0 \in \HH^1(\Omega; \sphere)$ with $\mm^0(\xx) := \xx / |\xx| \in \sphere$ is relaxed towards equilibrium.
The exchange length $\lex > 0$ and the Gilbert damping parameter $\alpha > 0$ are fixed at $1$.
The other discretization parameters --- namely the mesh-size $h > 0$, the time-step size $k > 0$, and the nonlinear solver accuracy $\eps > 0$ --- are subject to the numerical studies and are specified separately for each experiment.
Linear systems are solved with GMRES and accuracy $10^{-14}$.
For given $N \in \N$, the geometry is discretized by a structured mesh consisting of $(N+1)^3$ vertices and $6 N^3$ elements as described in \cite[Section~5.2]{prsehhsmp2020}, leading to a uniform mesh $\T_h$ of congruent tetrahedra, each of diameter $h_{\operatorname{max}} = \sqrt{3} / N$ and with shortest edge length $h_{\operatorname{min}} = 1 / N$.
To break symmetry, the discontinuity of the hedgehog state at the origin is discretized via $\mm_h^0(\0) := \ee_3 \in \sphere$, while $\mm_h^0(\zz) := \zz / |\zz| \in \sphere$ for all other $\0 \not= \zz \in \NN_h$.

\subsubsection{Feasible discretization parameters for nonlinear solvers}\label{mpslabel:sec:cfl_feasibility}
In the next section we carry out a numerical study on the CFL coupling of the time-step size to the mesh size arising from our analysis.
Since the constants hidden in CFL conditions are usually not readily available, we need to propose an appropriate criterion for the classification of given discretization parameters as feasible or non-feasible.
Hence, the goal is to derive such a criterion from the numerical experiment in this section.

For fixed mesh size $h_{\operatorname{min}} = 1/8$, nonlinear solver tolerance $\eps = 10^{-8}$, and starting from a rather fine time discretization $k = 0.00016$, we iteratively increase the time-step size by $25\%$ multiple times and track the number of nonlinear iterations required to meet the stopping criterion~\eqref{mpslabel:eq:stopping_fp} or \eqref{mpslabel:eq:stopping_newton}, respectively, in the first time-step of Algorithm~\ref{mpslabel:alg:mps_fp} or Algorithm~\ref{mpslabel:alg:mps_newton}.
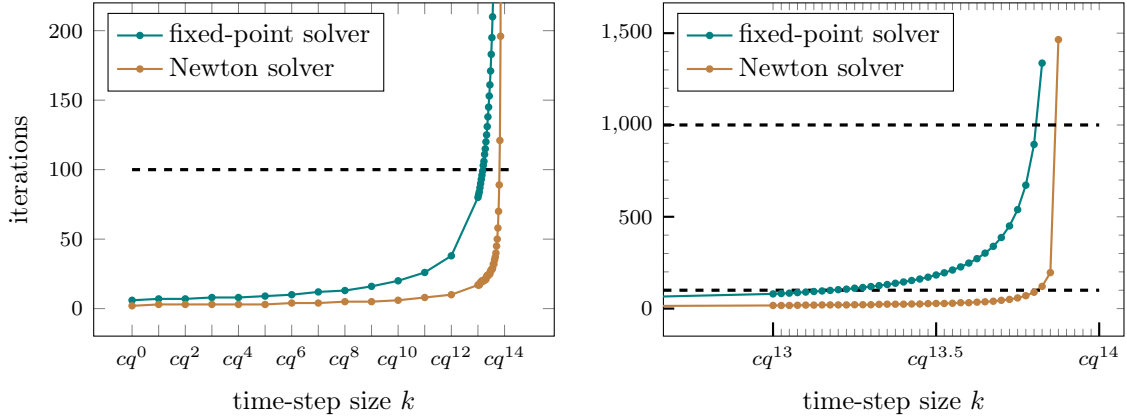
\begin{figure}[h]
\centering
\begin{subfigure}{0.48\textwidth}
\begin{tikzpicture}
\pgfplotstableread{plots/CFL_iter_vs_k/fp_h1By8_iter_vs_k.dat}{\fpdata}
\pgfplotstableread{plots/CFL_iter_vs_k/newton_h1By8_iter_vs_k.dat}{\newtondata}
\begin{semilogxaxis}[
xlabel={\footnotesize time-step size $k$},
ylabel={\footnotesize iterations},
xtick={1.600000e-01, 2.000000e-01, 2.500000e-01, 3.125000e-01, 3.906250e-01, 4.882812e-01, 6.103516e-01, 7.629395e-01, 9.536743e-01, 1.192093, 1.490116, 1.862645, 2.328306, 2.910383, 3.637979},
xticklabels={
$cq^0$, 
, 
$cq^2$, 
, 
$cq^4$, 
, 
$cq^6$, 
, 
$cq^8$, 
, 
$cq^{10}$,
,
$cq^{12}$,
,
$cq^{14}$
},
height = 60mm,
width=\textwidth,
legend style={legend pos=north west, legend cell align = left},
ymax=220,
ytick distance=50,
]
\addplot[teal, mark=*, mark size=1, thick] table[x=k, y=iter] {\fpdata};
\addplot[brown, mark=*, mark size=1, thick] table[x=k, y=iter] {\newtondata};
\addplot[black, dashed, very thick] coordinates {(0.16, 100) (4, 100)};
\legend{\footnotesize fixed-point solver,\footnotesize  Newton solver}
\end{semilogxaxis}
\end{tikzpicture}
\end{subfigure}
\begin{subfigure}{0.48\textwidth}
\begin{tikzpicture}
\pgfplotstableread{plots/CFL_iter_vs_k/fp_h1By8_iter_vs_k.dat}{\fpdata}
\pgfplotstableread{plots/CFL_iter_vs_k/newton_h1By8_iter_vs_k.dat}{\newtondata}
\begin{semilogxaxis}[
xlabel={\footnotesize time-step size $k$},
xtick={2.910383, 2.926664, 2.943036, 2.959500, 2.976056, 2.992704, 3.009446, 3.026282, 3.043211, 3.060235, 3.077355, 3.094570, 3.111882, 3.129290, 3.146796, 3.164400, 3.182102, 3.199903, 3.217804, 3.235805, 3.253907, 3.272110, 3.290414, 3.308822, 3.327332, 3.345945, 3.364663, 3.383486, 3.402414, 3.421447, 3.440587, 3.459835, 3.479190, 3.498653, 3.518225, 3.537907, 3.557698, 3.577601, 3.597614, 3.617740, 3.637978},
xticklabels={$cq^{13}$,,,,,,,,,,,,,,,,,,,,$cq^{13.5}$,,,,,,,,,,,,,,,,,,,,$cq^{14}$},
height = 60mm,
width=\textwidth,
legend style={legend pos=north west, legend cell align = left},
xmin=2.7,
xmax=3.7,
ymin=-150,
minor y tick num=4,
ytick={0, 500, 1000, 1500},
every minor tick/.append style={minor tick length=2pt},
every major tick/.append style={major tick length=2.75pt},
legend pos=north west,
]
\addplot[teal, mark=*, mark size=1, thick] table[x=k, y=iter] {\fpdata};
\addplot[brown, mark=*, mark size=1, thick] table[x=k, y=iter] {\newtondata};
\addplot[black, dashed, very thick] coordinates {(2.7, 1000) (3.637978, 1000)};
\addplot[black, dashed, very thick] coordinates {(2.7, 100) (3.637978, 100)};
\addplot[black, thick] coordinates {(2.910383, -150) (2.910383, -75)};
\addplot[black, thick] coordinates {(3.253907, -150) (3.253907, -75)};
\addplot[black, thick] coordinates {(3.637978, -150) (3.637978, -75)};
\addplot[black, thick] coordinates {(2.7, 0) (2.72, 0)};
\addplot[black, thick] coordinates {(2.7, 500) (2.72, 500)};
\addplot[black, thick] coordinates {(2.7, 1000) (2.72, 1000)};
\addplot[black, thick] coordinates {(2.7, 1500) (2.72, 1500)};
\legend{\footnotesize fixed-point solver,\footnotesize  Newton solver}
\end{semilogxaxis}
\end{tikzpicture}
\end{subfigure}
\caption[Experiment of Section~\ref{mpslabel:sec:cfl_feasibility}: Linearization strategies in practical midpoint schemes: CFL threshold detection and comparison of iteration numbers]{Experiment of Section~\ref{mpslabel:sec:cfl_feasibility}.
Left: The number of nonlinear iterations rapidly grows as the time-step size $k = cq^j > 0$ with $c=0.00016$ and $q = 5/4$ approaches the threshold value $k \to k_{\operatorname{thresh}}(h)$.
Right: Zoom into the critical area between $cq^{13}$ and $cq^{14}$ where the blow-up occurs.}
\label{mpslabel:fig:cfl_iter_vs_k}
\end{figure}
The results depicted in Figure~\ref{mpslabel:fig:cfl_iter_vs_k} show that for both practical midpoint schemes the number of nonlinear iterations stays well-bounded until a certain threshold value $k_{\operatorname{thresh}}(h) > 0$ is approached.
Close to the threshold value, however, an increase of the time-step size by $25\%$ impacts the number of nonlinear iterations by numbers of magnitude, if the solver converges at all.
Hence, it is reasonable to classify time-step sizes $k > 0$ with $k < k_{\operatorname{thresh}}(h)$ as feasible, and those with $k > k_{\operatorname{thresh}}(h)$ as non-feasible.
Surprisingly, despite the different theoretical CFL conditions $k = o(h^2)$ and $k=o(h^{7/3})$ imposed in Proposition~\ref{mpslabel:prop:fp} and Theorem~\ref{mpslabel:thm:main:newton:convergence}, respectively, the threshold value $k_{\operatorname{thresh}}(h)$ seems to coincide for Algorithm~\ref{mpslabel:alg:mps_fp} and Algorithm~\ref{mpslabel:alg:mps_newton}.
This observation is investigated further in Section~\ref{mpslabel:sec:cfl_k_vs_h}.
Finally, we note that in view of the quadratic convergence of Newton's method, it is not surprising that the Newton solver clearly outperforms the fixed point iteration in terms of nonlinear iteration numbers.

Motivated by the results of this experiment, in Section~\ref{mpslabel:sec:cfl_k_vs_h} we will use the following criterion to classify feasibility of discretization parameters:
If for any given $(h, k, \eps)$ the respective stopping criterion~\eqref{mpslabel:eq:stopping_fp} or \eqref{mpslabel:eq:stopping_newton}, is not met after at most $\num{100}$ iterations of the nonlinear solver in Algorithm~\ref{mpslabel:alg:mps_fp} or Algorithm~\ref{mpslabel:alg:mps_newton}, we consider the practical midpoint scheme as non-feasible for this combination of discretization parameters $h, k, \eps > 0$.
Given $h > 0$ this classification of feasibility is an estimate for the threshold value $k_{\operatorname{thresh}}(h) > 0$ such that the nonlinear solver converges for $0 < k < k_{\operatorname{thresh}}(h)$ and diverges for $k > k_{\operatorname{thresh}}(h)$.
Although only an approximation, Figure~\ref{mpslabel:fig:cfl_iter_vs_k} shows that in view of practical applicability this estimation of $k_{\operatorname{thresh}}(h)$ seems quite appropriate as nonlinear iteration numbers increase drastically as $k$ approaches $k_{\operatorname{thresh}}(h)$.

\subsubsection{Coupling of time-step size to mesh size}\label{mpslabel:sec:cfl_k_vs_h}
We consider the CFL conditions $k = o(h^2)$ and $k = o(h^{7/3})$ from Theorem~\ref{mpslabel:thm:main:fp}{\textrm{(i)}} and Theorem~\ref{mpslabel:thm:main:newton:convergence}, respectively, sufficient to guarantee convergence of the fixed point iteration and the Newton solver.
For different mesh sizes $h_{\operatorname{min}} \in \{2^{-j} \colon j = 1, \dots, 5\}$, time-step sizes $k \in \{0.00016 \cdot \big(\frac{5}{4}\big)^j \colon j = 0, \dots, 27\}$, and nonlinear solver accuracy $\eps > 0$ fixed at $10^{-8}$, we investigate convergence of the nonlinear solver for one time-step of relaxing the initial hedgehog state.
As argued in Section~\ref{mpslabel:sec:cfl_feasibility}, the threshold value of $100$ nonlinear iterations is used to classify feasibility of the discretization parameters.

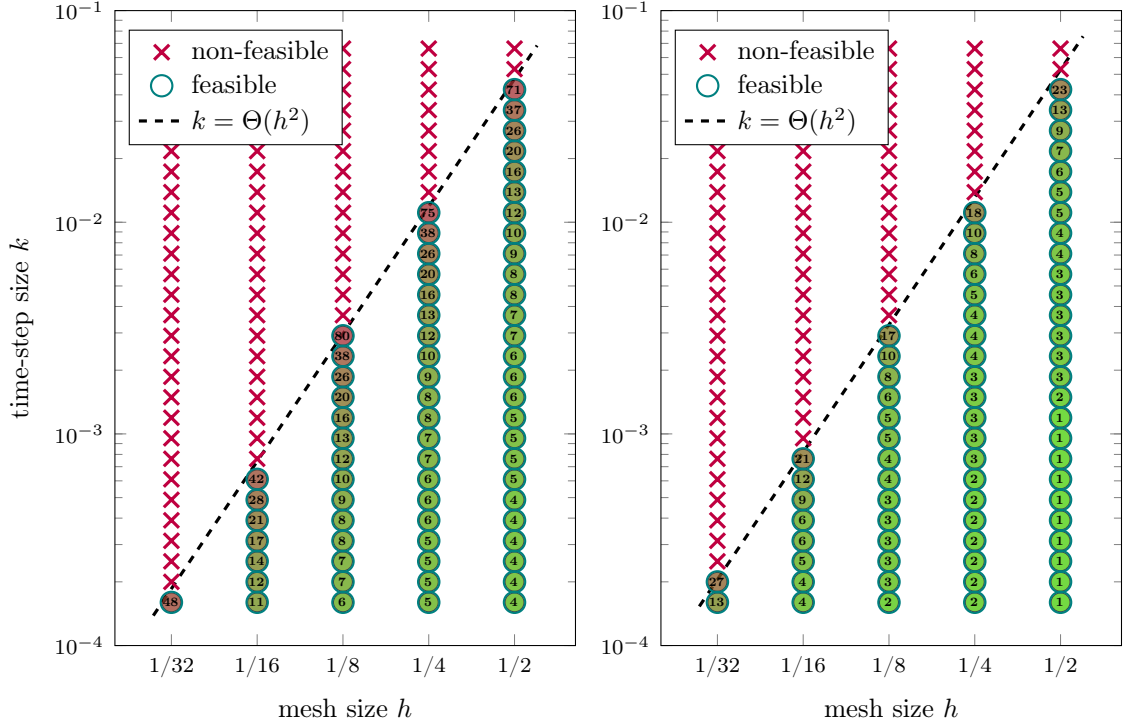
\begin{figure}[h]
\centering
\begin{subfigure}{0.48\textwidth}
\begin{tikzpicture}
\pgfplotstableread{plots/CFL_k_vs_h/fp_data_bad.dat}{\baddata}
\pgfplotstableread{plots/CFL_k_vs_h/fp_data_good.dat}{\gooddata}
\pgfplotstableread{plots/CFL_k_vs_h/data_h.dat}{\hdata}
\begin{loglogaxis}[
xlabel={\footnotesize mesh size $h$},
ylabel={\footnotesize time-step size $k$},
xtick={0.03125, 0.0625, 0.125, 0.25, 0.5},
xticklabels={1/32, 1/16, 1/8, 1/4, 1/2},
height = 100mm,
width=\textwidth,
legend style={legend pos=north west, legend cell align = left},
ymax=1e-1,
ymin=1e-4,
]
\addplot[only marks, purple, mark=x, mark size=4, very thick] table[x=h, y=k] {\baddata};
\addplot[only marks, teal, mark=o, mark size=4, thick] table[x=h, y=k] {\gooddata};
\addplot[black, dashed, very thick] table[x=h, y expr={0.19*(\thisrow{h}*\thisrow{h})}]{\hdata};

\pgfplotstablegetrowsof{plots/CFL_k_vs_h/fp_data_good.dat}
\pgfmathsetmacro{\rows}{\pgfplotsretval-1}
\foreach \i in {0,...,\rows}{%
  \pgfplotstablegetelem{\i}{[index] 0}\of{\gooddata}
  \let\hval\pgfplotsretval
  \pgfplotstablegetelem{\i}{[index] 1}\of{\gooddata}
  \let\kval\pgfplotsretval
  \pgfplotstablegetelem{\i}{[index] 2}\of{\gooddata}
  \let\fpval\pgfplotsretval
  \edef\temp{\noexpand\node[fill={rgb:purple,\the\numexpr2*\fpval\relax;green!80!black,30;yellow,10}, opacity=0.75, draw=black,circle,minimum size=0.5cm,inner sep=0pt, anchor=center, scale=0.5] at (axis cs:\hval, \kval) {};
  \noexpand\node at (axis cs:\hval, \kval) [scale=0.6] {\noexpand\tiny$\,\noexpand\mathbf{\fpval}\;$};
}
\temp
}
\legend{\footnotesize non-feasible, \footnotesize feasible, \footnotesize$k = \Theta(h^2)$}
\end{loglogaxis}
\end{tikzpicture}
\end{subfigure}
\begin{subfigure}{0.48\textwidth}
\begin{tikzpicture}
\pgfplotstableread{plots/CFL_k_vs_h/newton_data_bad.dat}{\baddata}
\pgfplotstableread{plots/CFL_k_vs_h/newton_data_good.dat}{\gooddata}
\pgfplotstableread{plots/CFL_k_vs_h/data_h.dat}{\hdata}
\begin{loglogaxis}[
xlabel={\footnotesize mesh size $h$},
xtick={0.03125, 0.0625, 0.125, 0.25, 0.5},
xticklabels={1/32, 1/16, 1/8, 1/4, 1/2},
height = 100mm,
width=\textwidth,
legend style={legend pos=north west, legend cell align = left},
ymax=1e-1,
ymin=1e-4,
]
\addplot[only marks, purple, mark=x, mark size=4, very thick] table[x=h, y=k] {\baddata};
\addplot[only marks, teal, mark=o, mark size=4, thick] table[x=h, y=k] {\gooddata};
\addplot[black, dashed, very thick] table[x=h, y expr={0.210*(\thisrow{h}*\thisrow{h})}]{\hdata};

\pgfplotstablegetrowsof{plots/CFL_k_vs_h/newton_data_good.dat}
\pgfmathsetmacro{\rows}{\pgfplotsretval-1}
\foreach \i in {0,...,\rows}{%
  \pgfplotstablegetelem{\i}{[index] 0}\of{\gooddata}
  \let\hval\pgfplotsretval
  \pgfplotstablegetelem{\i}{[index] 1}\of{\gooddata}
  \let\kval\pgfplotsretval
  \pgfplotstablegetelem{\i}{[index] 2}\of{\gooddata}
  \let\fpval\pgfplotsretval
  \edef\temp{\noexpand\node[fill={rgb:purple,\the\numexpr2*\fpval\relax;green!80!black,30;yellow,10}, opacity=0.75, draw=black,circle,minimum size=0.5cm,inner sep=0pt, anchor=center, scale=0.5] at (axis cs:\hval, \kval) {};
  \noexpand\node at (axis cs:\hval, \kval) [scale=0.6] {\noexpand\tiny$\,\noexpand\mathbf{\fpval}\;$};
}
\temp
}
\legend{\footnotesize non-feasible, \footnotesize feasible, \footnotesize$k = \Theta(h^2)$}
\end{loglogaxis}
\end{tikzpicture}
\end{subfigure}
\caption[Experiment of Section~\ref{mpslabel:sec:cfl_k_vs_h}:
Convergence of the nonlinear solvers in the practical midpoint schemes]{Experiment of Section~\ref{mpslabel:sec:cfl_k_vs_h}.
Convergence of the nonlinear solvers in the practical midpoint schemes is investigated.
For feasible parameters, the number of nonlinear iterations is given inside the circle.
Left: Practical midpoint scheme based on the fixed-point iteration (Algorithm~\ref{mpslabel:alg:mps_fp}).
Right: Practical midpoint scheme based on the Newton iteration (Algorithm~\ref{mpslabel:alg:mps_newton}).
The data points show feasibility if $k = \Theta(h^\beta)$ with possible slopes $1.93 \le \beta_{\operatorname{fixed-point}} \le 2.09$ and $1.85 \le \beta_{\operatorname{newton}} \le 2.01$.
}
\label{mpslabel:fig:cfl_k_vs_h}
\end{figure}

The results of this experiment shown in Figure~\ref{mpslabel:fig:cfl_k_vs_h} give insight to the applicability of the practical midpoint schemes:
First, for Algorithm~\ref{mpslabel:alg:mps_fp} the theoretically sufficient CFL condition $k = o(h^2)$ is shown to be sharp in practice.
Further, since the experiment reveals the same CFL condition $k = o(h^2)$ to be sufficient for convergence of the Newton solver, we expect that the well-posedness analysis of the Newton iteration can be improved weakening the CFL condition in Theorem~\ref{mpslabel:thm:main:newton:convergence} from $k = o(h^{7/3})$ to $k = o(h^2)$.
We note that also in the simulation of skyrmion dynamics in Section~\ref{mpslabel:sec:critical_D} both practical midpoint schemes were equivalently restrictive on the time discretization.
Lastly, this experiment shows that, in terms of iteration numbers, the Newton solver outperforms the fixed-point solver as expected from theory (quadratic vs.\ linear convergence).

\subsubsection{Constraint violation induced by nonlinear solver accuracy}
\label{mpslabel:sec:eps_vs_unitLength}
In contrast to the fixed-point iteration from Section~\ref{mpslabel:sec:fp}, the Newton iteration from Section~\ref{mpslabel:sec:newton} does not inherently preserve discrete unit-length, i.e., $\mm_{h\eps}^i \not\in \Mh$ for the Newton linearization.
To quantify the impact of the Newton solver on the discrete magnetization length, the initial hedgehog state is relaxed to equilibrium ($T = 5$) using different nonlinear solver accuracies $\eps \in \{10^{-j/2} \colon j = 0, \dots, 24\}$.
We simulate the dynamics for $h_{\operatorname{min}} = 1/4$ and $h_{\operatorname{min}} = 1/8$ with time-step sizes chosen roughly half the value of $k_{\operatorname{thresh}}(h)$ from Section~\ref{mpslabel:sec:cfl_feasibility}.
In Figure~\ref{mpslabel:fig:eps_vs_constraint} we plot the deviations
\begin{align}\label{mpslabel:deviation:maxmin}
\max_{\zz\in\Nh}|\mm_{h\eps k}(T, \zz)| - 1
\qquad\text{and}\qquad
1 - \min_{\zz\in\Nh}|\mm_{h\eps k}(T, \zz)|
\end{align}
over the nonlinear solver accuracy $\eps > 0$.
In this experiment for the practical midpoint scheme based on the Newton iteration deviation from unit-length decreases with rate between $\Theta(\eps^{9/10})$ and $\Theta(\eps^{8/10})$ as $\eps \to 0$.
In contrast to that, for the practical midpoint scheme based on the fixed-point iteration the deviation from unit-length is unaffected by the choice of $\eps > 0$ as expected from theory.
\begin{figure}[htb!]
\centering
\begin{tikzpicture}
\pgfplotstableread{plots/CFL_eps_vs_constraint/compact_deviation_newton_h1By4.dat}{\newtondataByFour}
\pgfplotstableread{plots/CFL_eps_vs_constraint/compact_deviation_fp_h1By4.dat}{\fpdataByFour}
\pgfplotstableread{plots/CFL_eps_vs_constraint/compact_deviation_newton_h1By8.dat}{\newtondataByEight}
\pgfplotstableread{plots/CFL_eps_vs_constraint/compact_deviation_fp_h1By8.dat}{\fpdataByEight}
\pgfplotstableread{plots/CFL_eps_vs_constraint/ref_0.9.dat}{\refdataDotNine}
\pgfplotstableread{plots/CFL_eps_vs_constraint/ref_0.8.dat}{\refdataDotEight}
\begin{loglogaxis}[
xlabel={\footnotesize nonlinear solver accuracy $\eps$},
ylabel={\footnotesize deviation},
xtick={1e-0, 1e-2, 1e-4, 1e-6, 1e-8, 1e-10, 1e-12},
ytick={1e-4, 1e-6, 1e-8, 1e-10, 1e-12, 1e-14},
height = 60.5mm,
width=80mm,
legend style={legend pos=outer north east, legend cell align = left},
]
\addplot[brown, only marks, mark=o, mark size=2, thick] table[x=eps, y=errInfAbove] {\newtondataByFour};
\addplot[brown, only marks, mark=x, mark size=2, thick] table[x=eps, y=errInfBelow] {\newtondataByFour};
\addplot[teal, only marks, mark=o, mark size=2, thick] table[x=eps, y=errInfAbove] {\newtondataByEight};
\addplot[teal, only marks, mark=x, mark size=2, thick] table[x=eps, y=errInfBelow] {\newtondataByEight};
\addplot[purple, only marks, mark=o, mark size=2, thick] table[x=eps, y=errInfAbove] {\fpdataByFour};
\addplot[purple, only marks, mark=x, mark size=2, thick] table[x=eps, y=errInfBelow] {\fpdataByFour};
\addplot[gray, only marks, mark=o, mark size=2, thick] table[x=eps, y=errInfAbove] {\fpdataByEight};
\addplot[gray, only marks, mark=x, mark size=2, thick] table[x=eps, y=errInfBelow] {\fpdataByEight};
\addplot[black, dashed, very thick] table[x=eps, y expr={0.03*(\thisrow{epsPow})}]{\refdataDotNine};
\addplot[black, dashed, very thick] table[x=eps, y expr={0.0003*(\thisrow{epsPow})}]{\refdataDotEight};
\legend{
\footnotesize Newton $h=1/4$: $\max - 1$,
\footnotesize Newton $h=1/4$: $1 -\min$,
\footnotesize Newton $h=1/8$: $\max - 1$,
\footnotesize Newton $h=1/8$: $1 -\min$,
\footnotesize fixed-point $h=1/4$: $\max - 1$,
\footnotesize fixed-point $h=1/4$: $1 -\min$,
\footnotesize fixed-point $h=1/8$: $\max - 1$,
\footnotesize fixed-point $h=1/8$: $1 -\min$,,
\footnotesize $\Theta(\eps^{9/10})$ and $\Theta(\eps^{8/10})$,
}
\end{loglogaxis}
\end{tikzpicture}
\caption[Experiment of Section~\ref{mpslabel:sec:eps_vs_unitLength}:
Dependence of the constraint violation on the nonlinear solver accuracy $\eps > 0$]{Experiment of Section~\ref{mpslabel:sec:eps_vs_unitLength}.
Dependence of the constraint violation~\eqref{mpslabel:deviation:maxmin} on the nonlinear solver accuracy $\eps > 0$ is investigated.
For the Newton solver the deviation from unit-length decreases as $\eps \to 0$.
No obvious correlation is observed for the fixed-point iteration, which is expected, since it is designed to be constraint preserving; see Proposition~\ref{mpslabel:prop:fp}{\rm(iii)}.}
\label{mpslabel:fig:eps_vs_constraint}
\end{figure}
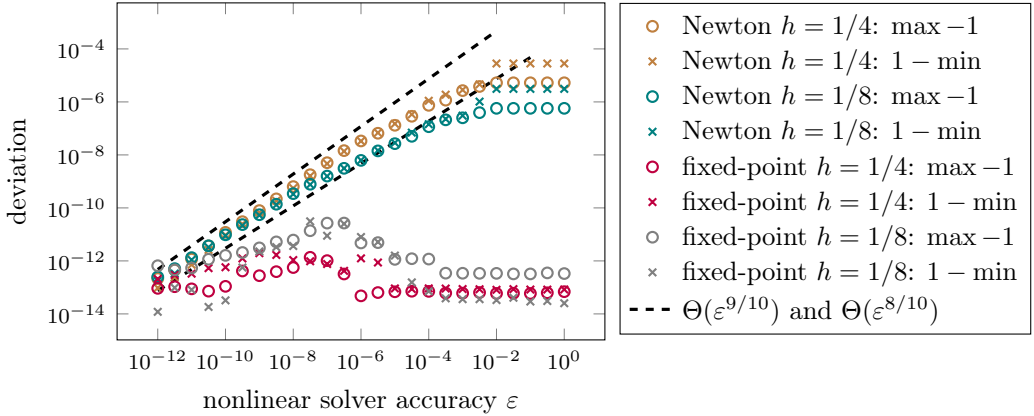

As in this experiment both practical midpoint schemes were stable (i.e., non energy-increasing) even for nonlinear solver accuracies as large as $\eps = 1$, an experimental setup for the investigation of the coupling to the mesh size $\eps = \mathcal{O}(h)$ and $\eps = \mathcal{O}(h^{3/2})$ from Theorem~\ref{mpslabel:thm:main:fp}{\textrm{(i)}} and Theorem~\ref{mpslabel:thm:main:newton:convergence}, respectively, is yet to be proposed in a future numerical study.

%% file: plots/merged_interval_plot.tex
\pgfdeclareplotmark{Oo}{%
  \pgfpathcircle{\pgfpointorigin}{\pgfplotmarksize}%
  \pgfpathcircle{\pgfpointorigin}{0.5\pgfplotmarksize}%
  \pgfusepathqstroke
}
\begin{tikzpicture}
\pgfplotstableread{plots/interval_tps1_all_uniform.dat}{\tpsUniform}
\pgfplotstableread{plots/first_interval_tps1_all_skyrmion.dat}{\tpsSkyrmion}
\pgfplotstableread{plots/interval_tps1_all_skyrmion.dat}{\tpsLargeSkyrmion}
\pgfplotstableread{plots/interval_tps1_all_broken.dat}{\tpsBroken}
\pgfplotstableread{plots/interval_tps1_all_target.dat}{\tpsTarget}
\pgfplotstableread{plots/interval_tps2_all_uniform.dat}{\ttpsUniform}
\pgfplotstableread{plots/first_interval_tps2_all_skyrmion.dat}{\ttpsSkyrmion}
\pgfplotstableread{plots/interval_tps2_all_skyrmion.dat}{\ttpsLargeSkyrmion}
\pgfplotstableread{plots/interval_tps2_all_broken.dat}{\ttpsBroken}
\pgfplotstableread{plots/interval_tps2_all_target.dat}{\ttpsTarget}
\pgfplotstableread{plots/interval_mps_all_uniform.dat}{\mpsUniform}
\pgfplotstableread{plots/first_interval_mps_all_skyrmion.dat}{\mpsSkyrmion}
\pgfplotstableread{plots/interval_mps_all_skyrmion.dat}{\mpsLargeSkyrmion}
\pgfplotstableread{plots/interval_mps_all_broken.dat}{\mpsBroken}
\pgfplotstableread{plots/interval_mps_all_target.dat}{\mpsTarget}
\begin{axis}[
width = \textwidth, 
height = 80mm,
xlabel={\footnotesize $D$ [\si{\milli\joule\per\square\meter}]},
ylabel={\footnotesize time-steps per \SI{1}{\pico\second} ($\simeq k^{-1}$)},
xmin=2.405,
xmax=2.4975,
ymin=0,
ymax=16,
ytick={1, 2, 3, 5, 6, 7, 8, 9, 11, 12, 13, 14, 15},
yticklabels={
$1600$,
$800$,
$400$,
$1600$,
$800$,
$400$,
$200$,
$100$,
$1600$,
$800$,
$400$,
$200$,
$100$,
},
grid=major,
xtick={
2.4100, 2.4125, 2.4150, 2.4175,
2.4200, 2.4225, 2.4250, 2.4275,
2.4300, 2.4325, 2.4350, 2.4375,
2.4400, 2.4425, 2.4450, 2.4475,
2.4550, 2.4575,
2.4600, 2.4625, 2.4650, 2.4675,
2.4700, 2.4725, 2.4750, 2.4775,
2.4800, 2.4825, 2.4850, 2.4875,
2.4900, 2.4925
}, 
xticklabels={2.41,,,, 2.42,,,, 2.43,,,, 2.44,,,, ,,6.64,,,, 6.65,,,, 6.66,,,, 6.67,},
legend columns=4,
legend style={/tikz/column 4/.style={column sep=15pt}},
legend style={at={(0.5,1.025)},anchor=south}
]
\addplot[mark options={solid}, brown, only marks, mark=*, mark size=3, thick] table[x=D, y expr=\thisrowno{1}+10]{\tpsUniform};
\addplot[mark options={solid}, brown, only marks, mark=o, mark size=3, very thick] table[x=D, y expr=\thisrowno{1}+10]{\tpsSkyrmion};
\addplot[mark options={solid}, brown, only marks, mark=x, mark size=4, very thick] table[x expr=\thisrowno{0}-4.18, y expr=\thisrowno{1}+10]{\tpsBroken};
\addplot[mark options={solid}, brown, only marks, mark=Oo, mark size=3, thick] table[x expr=\thisrowno{0}-4.18, y expr=\thisrowno{1}+10]{\tpsTarget};
\addplot[mark options={solid}, teal, only marks, mark=*, mark size=3, thick] table[x=D, y expr=\thisrowno{1}+4]{\ttpsUniform};
\addplot[mark options={solid}, teal, only marks, mark=o, mark size=3, very thick] table[x=D, y expr=\thisrowno{1}+4]{\ttpsSkyrmion};
\addplot[mark options={solid}, teal, only marks, mark=x, mark size=4, very thick] table[x expr=\thisrowno{0}-4.18, y expr=\thisrowno{1}+4]{\ttpsBroken};
\addplot[mark options={solid}, teal, only marks, mark=Oo, mark size=3, thick] table[x expr=\thisrowno{0}-4.18, y expr=\thisrowno{1}+4]{\ttpsTarget};
\addplot[mark options={solid}, purple, only marks, mark=*, mark size=3, thick] table[x=D, y=nr]{\mpsUniform};
\addplot[mark options={solid}, purple, only marks, mark=o, mark size=3, very thick] table[x=D, y=nr]{\mpsSkyrmion};
\addplot[mark options={solid}, purple, only marks, mark=x, mark size=4, very thick] table[x expr=\thisrowno{0}-4.18, y expr=\thisrowno{1}]{\mpsBroken};
\addplot[mark options={solid}, purple, only marks, mark=Oo, mark size=3, thick] table[x expr=\thisrowno{0}-4.18, y expr=\thisrowno{1}]{\mpsTarget};
\addplot[mark options={solid}, brown, only marks, mark=o, mark size=3, very thick] table[x expr=\thisrowno{0}-4.18, y expr=\thisrowno{1}+10]{\tpsLargeSkyrmion};
\addplot[mark options={solid}, teal, only marks, mark=o, mark size=3, very thick] table[x expr=\thisrowno{0}-4.18, y expr=\thisrowno{1}+4]{\ttpsLargeSkyrmion};
\addplot[mark options={solid}, purple, only marks, mark=o, mark size=3, very thick] table[x expr=\thisrowno{0}-4.18, y expr=\thisrowno{1}]{\mpsLargeSkyrmion};
\addplot[mark options={solid}, purple, only marks, mark=o, mark size=3, very thick] table[x expr=\thisrowno{0}-4.18, y expr=\thisrowno{1}]{\mpsLargeSkyrmion};
\addplot [mark=none, ultra thick, dotted, brown] coordinates {(2.45025, 13) (2.45225, 13)};
\addplot [mark=none, ultra thick, dotted, teal] coordinates {(2.45025, 7) (2.45225, 7)};
\addplot [mark=none, ultra thick, dotted, purple] coordinates {(2.45025, 2) (2.45225, 2)};

\addplot [mark=none, very thick, dashed, black] coordinates 
{(2.44375, 0.5) (2.44375, 3.5) (2.46375, 3.5) (2.46375, 0.5) (2.44375, 0.5)};
\addplot [mark=none, very thick, dashed, black] coordinates 
{(2.51, 0.5) (2.47625, 0.5) (2.47625, 3.5) (2.51, 3.5)};
\addplot [mark=none, very thick, dashed, black] coordinates 
{
(2.44125, 4.5) (2.44125, 5.5) 
(2.43625, 5.5) (2.43625, 6.5)
(2.43375, 6.5) (2.43375, 7.5)
(2.43125, 7.5) (2.43125, 8.5)
(2.42875, 8.5) (2.42875, 9.5)
(2.47375, 9.5)
(2.47375, 8.5) (2.47125, 8.5)
(2.47125, 7.5) (2.46625, 7.5)
(2.46625, 6.5) (2.46125, 6.5)
(2.46125, 5.5) (2.45875, 5.5)
(2.45875, 4.5) (2.44125, 4.5)
};
\addplot [mark=none, very thick, dashed, black] coordinates 
{
(2.51, 9.5) (2.48125, 9.5)
(2.48125, 8.5) (2.47875, 8.5)
(2.47875, 5.5) (2.48125, 5.5)
(2.48125, 4.5) (2.51, 4.5)
};
\addplot [mark=none, very thick, dashed, black] coordinates 
{
(2.43375, 10.5) (2.43375, 11.5)
(2.43125, 11.5) (2.43125, 12.5)
(2.42875, 12.5) (2.42875, 13.5)
(2.42125, 13.5) (2.42125, 14.5)
(2.41375, 14.5) (2.41375, 15.5)
(2.48125, 15.5)
(2.48125, 14.5) (2.47625, 14.5)
(2.47625, 13.5) (2.47125, 13.5)
(2.47125, 12.5) (2.46625, 12.5)
(2.46625, 11.5) (2.46375, 11.5)
(2.46375, 10.5) (2.43375, 10.5)
};
\addplot [mark=none, very thick, dashed, black] coordinates 
{
(2.51, 15.5)
(2.48875, 15.5) (2.48875, 14.5)
(2.48375, 14.5) (2.48375, 13.5)
(2.48125, 13.5) (2.48125, 12.5)
(2.47875, 12.5) (2.47875, 10.5)
(2.51, 10.5)
};
\legend{
\footnotesize TPS1 uniform\;\;\;\;,
\footnotesize TPS1 skyrmion,
\footnotesize TPS1 broken\;\;\;\;,
\footnotesize TPS1 target,
\footnotesize TPS2 uniform\;\;\;\;,
\footnotesize TPS2 skyrmion,
\footnotesize TPS2 broken\;\;\;\;,
\footnotesize TPS2 target,
\footnotesize MPS uniform\;\;\;\;,
\footnotesize MPS skyrmion,
\footnotesize MPS broken\;\;\;\;,
\footnotesize MPS target
}
\end{axis}
\end{tikzpicture}

%% file: sec_proof_ideal.tex
\section[Analysis of the ideal midpoint scheme]{Proof of Theorem~\ref{mpslabel:thm:main} for the ideal midpoint scheme} \label{mpslabel:sec:proofs1}

\subsection{Existence of solutions, unit-length constraint, and stability}

\begin{proof}[Proof of Theorem~\ref{mpslabel:thm:main}{\textrm{(i)}}]
Let $i \in \N_0$ be arbitrary.
Define $\FF \colon \Vh \to \Vh$ by
\begin{equation*}
\FF (\pphi_h) := \pphi_h - \mm_h^i
+ \frac{k}{2} \, \Interp \big[ \pphi_h \times \Ph \heff(\pphi_h) + \alpha \, \pphi_h \times \mm_h^i \big]
\quad
\text{for all } \pphi_h \in \Vh.
\end{equation*}
If $\eeta_h \in \Vh$ satisfies $\FF (\eeta_h) = \0$, then $\mm_h^{i+1} := 2 \eeta_h - \mm_h^i$
satisfies~\eqref{mpslabel:eq:mps}.
Since
\begin{equation*}
\inner[h]{\FF (\pphi_h)}{\pphi_h}
= \inner[h]{\pphi_h - \mm_h^i}{\pphi_h}
\geq 0
\quad
\text{for all } \pphi_h \in \Vh \text{ with } \norm[h]{\pphi_h} = \norm[h]{\mm_h^i} > 0,
\end{equation*}
an application of the Brouwer fixed-point theorem
(see, e.g., \cite[Chapter~IV, Corollary~1.1]{gr1986})
ensures the existence of $\eeta_h \in \Vh$ such that
$\norm[h]{\eeta_h} \leq \norm[h]{\mm_h^i}$
and
$\FF (\eeta_h) = \0$.
This proves that~\eqref{mpslabel:eq:mps} admits a solution $\mm_h^{i + 1} \in \Vh$.

Let $\zz \in \Nh$ be arbitrary.
We test~\eqref{mpslabel:eq:mps} with $\pphih = \mm_h^{i+1/2}(\zz) \vphi_{\zz} \in \Vh$ to obtain that
\begin{equation*}
\inner[h]{d_t \mm_h^{i+1}}{\vphi_{\zz} \mm_h^{i+1/2}(\zz)}
= \frac{\beta_{\zz}}{2k} \left( \abs{\mm_h^{i+1}(\zz)}^2 - \abs{\mm_h^i(\zz)}^2 \right)
= 0.
\end{equation*}
We conclude that $\abs{\mm_h^{i+1}(\zz)} = \abs{\mm_h^i(\zz)}$.
Since $\mm_h^0 \in \Mh$ by assumption, we conclude that $\mm_h^{i+1} \in \Mh$.
\end{proof}

\begin{proof}[Proof of Theorem~\ref{mpslabel:thm:main}{\textrm{(ii)}}]
Let $J \in \N$.
To show~\eqref{mpslabel:eq:stability}, we choose the test function
$\pphih = \alpha \, d_t \mm_h^{i+1} - \Ph \heff(\mm_h^{i + 1/2}) \in \Vh$
in~\eqref{mpslabel:eq:mps}.
We obtain the equality
\begin{equation*}
\inner[h]{\Ph \heff(\mm_h^{i + 1/2})}{d_t \mm_h^{i+1}}
= \alpha \norm[h]{d_t \mm_h^{i+1}}^2.
\end{equation*}
For the left-hand side, it holds that
\begin{equation} \label{mpslabel:eq:aux_eq}
\begin{split}
\inner[h]{\Ph \heff(\mm_h^{i + 1/2})}{d_t \mm_h^{i+1}}
& \! \stackrel{\eqref{mpslabel:eq:pseudo-projection}}{=} \! \dual{\heff(\mm_h^{i + 1/2})}{d_t \mm_h^{i+1}} \\
& \stackrel{\eqref{mpslabel:eq:heff}}{=} - a(\mm_h^{i + 1/2},d_t \mm_h^{i+1})
+ \inner[\Omega]{\ff}{d_t \mm_h^{i+1}} \\
& \stackrel{\eqref{mpslabel:eq:llg:energy}}{=} - \frac{1}{k} ( \E(\mm_h^{i + 1}) - \E(\mm_h^i) ).
\end{split}
\end{equation}
We conclude that
\begin{equation*}
\E(\mm_h^{i + 1}) - \E(\mm_h^i) = - \alpha k \norm[h]{d_t \mm_h^{i+1}}^2.
\end{equation*}
Summation over $i=0,\dots,J-1$ yields~\eqref{mpslabel:eq:stability}.
\end{proof}

\subsection{Weak convergence result}

To start with, we note that the bilinear forms $a(\cdot,\cdot)$ and $\aloc(\cdot,\cdot)$ are continuous, i.e.,
there exists $C_1>0$ such that
\begin{subequations} \label{mpslabel:eq:bilinear_inequalities}
\begin{alignat}{2}
\label{mpslabel:eq:a:continuity}
a(\ppsi,\vvphi) &\leq (C_1 + \norm[L(\LL^2(\Omega);\LL^2(\Omega))]{\ppi}) \norm[\HH^1(\Omega)]{\ppsi}\norm[\HH^1(\Omega)]{\vvphi}
&\quad&
\text{for all } \ppsi, \vvphi \in \HH^1(\Omega), \\
\label{mpslabel:eq:b:continuity}
\aloc(\ppsi,\vvphi) &\leq C_1 \norm[\HH^1(\Omega)]{\ppsi}\norm[\HH^1(\Omega)]{\vvphi}
&&
\text{for all } \ppsi, \vvphi \in \HH^1(\Omega),
\end{alignat}
and satisfy the G{\aa}rding inequality, i.e.,
there exist $C_2>0$ and $C_3 \in \R$ such that
\begin{equation}\label{mpslabel:eq:gaarding}
a(\ppsi,\ppsi) \ge \aloc(\ppsi,\ppsi) \ge C_2 \norm[\HH^1(\Omega)]{\ppsi}^2 - C_3 \norm[\LL^2(\Omega)]{\ppsi}^2
\quad
\text{for all } \ppsi \in \HH^1(\Omega).
\end{equation}
\end{subequations}
The constants $C_1,C_2,C_3$ in~\eqref{mpslabel:eq:bilinear_inequalities} depend on
$\norm[\LL^{\infty}(\Omega)]{\Amat_n}$ and $\norm[\LL^{\infty}(\Omega)]{\JJ_n}$ ($n=1,2,3$),
and $A_0$.
Finally, we consider, besides~\eqref{mpslabel:eq:timeApprox}, the piecewise constant time reconstruction
$\mmhkbar$ defined by $\mmhkbar(t) := \mm_h^{i+1/2}$ for all $i \in \N_0$ and $t \in [t_i,t_{i+1})$.

With these ingredients, we prove the convergence result for Algorithm~\ref{mpslabel:alg:mps}.

\begin{proof}[Proof of Theorem~\ref{mpslabel:thm:main}{\textrm{(iii)}}]
The proof follows the lines of~\cite{bp2006,prs2018}, therefore we only sketch it.
Let $J \in \N$.
Since $\mm_h^J \in \Mh$, $\norm[\LL^2(\Omega)]{\mm_h^J} \leq \abs{\Omega}^{1/2}$.
Hence, combining the inequalities~\eqref{mpslabel:eq:bilinear_inequalities} and the norm equivalence~\eqref{mpslabel:eq:normEquivalence}
with~\eqref{mpslabel:eq:stability},
we obtain the estimate
\begin{equation} \label{mpslabel:eq:stability:proof}
\norm[\HH^1(\Omega)]{\mm_h^J}^2
+ k \sum_{i=0}^{J-1} \norm[\LL^2(\Omega)]{d_t \mm_h^i}^2
\leq C,
\end{equation}
where $C>0$ depends only on the problem data.
We infer the uniform boundedness of the sequences of time reconstructions $\{\mmhk\}$ and $\{\mmhkbar\}$
in $L^{\infty}(\R_{>0};\HH^1(\Omega))$.
Let $T>0$ be arbitrary.
From~\eqref{mpslabel:eq:stability:proof}, it also follows the uniform boundedness of
$\{\mmhk\vert_{\Omega_T}\}$ (resp., $\{\mmhkbar\vert_{\Omega_T}\}$)
in $\HH^1(\Omega_T)$ and in $L^{\infty}(0,T;\HH^1(\Omega))$
(resp., only in $L^{\infty}(0,T;\HH^1(\Omega))$).
With successive extractions of convergent subsequences (not relabeled),
one can show that there exists a common limit $\mm\in L^{\infty}(\R_{>0};\HH^1(\Omega))$
with $\mm\vert_{\Omega_T} \in \HH^1(\Omega_T)$
for which we have the convergences
$\mmhk, \mmhkbar \weakstarto \mm$ in $L^{\infty}(\R_{>0};\HH^1(\Omega))$,
$\mmhk\vert_{\Omega_T}, \mmhkbar\vert_{\Omega_T} \weakstarto \mm\vert_{\Omega_T}$
in $L^{\infty}(0,T ; \HH^1(\Omega))$,
and $\mmhk\vert_{\Omega_T} \weakto \mm\vert_{\Omega_T}$ in $\HH^1(\Omega_T)$,
With the argument of~\cite[Sections~3.2--3.3]{prs2018},
one also gets that the limit function $\mm$ is $\sphere$-valued
and satisfies the initial condition $\mm(0)=\mm^0$ in the sense of traces.

To verify the variational formulation~\eqref{mpslabel:eq:weak:variational},
let $\vvphi \in \CC^{\infty}(\overline{\Omega_T})$.
Let $J \in \N$ the smallest integer such that $T \leq kJ$.
We define the semi-discrete function $\vvphi_h \in \CC^{\infty}([0,kJ];\Vh)$
by $\vvphi_h(t)=\Interp[\vvphi(t)]$ for all $t \in [0,kJ]$.
For $i=0,\dots,J-1$ and $t \in (t_i,t_{i+1})$, we test~\eqref{mpslabel:eq:mps} with $\pphih =  \vvphi_h(t) \in \Vh$.
Then, integrating in time over $(t_i,t_{i+1})$ and summing over $i=0,\dots,J-1$,
we obtain that
\begin{equation} \label{mpslabel:eq:varform_before_limit}
\begin{split}
\int_0^{kJ} \inner[h]{\de_t \mmhk(t)}{\vvphi_h(t)} \dt
& = - \int_0^{kJ}  \inner[h]{ \mmhkbar(t) \times \Ph \heff(\mmhkbar(t))}{\vvphi_h(t)} \dt \\
& \quad + \alpha \int_0^{kJ} \inner[h]{\mmhkbar(t) \times \de_t \mmhk(t)}{\vvphi_h(t)} \dt.
\end{split}
\end{equation}
The argument in~\cite[Section~3]{bp2006} shows that
\begin{gather*}
\int_0^{kJ} \inner[h]{\de_t \mmhk(t)}{\vvphi_h(t)} \dt
\to \int_0^T \inner[\Omega]{\mmt(t)}{\vvphi(t)} \dt
\quad \text{as } h,k \to 0
\quad \text{and} \\
\int_0^{kJ} \inner[h]{\mmhkbar(t) \times \de_t \mmhk(t)}{\vvphi_h(t)} \dt
\to \int_0^T \inner[\Omega]{\mm(t) \times \mmt(t)}{\vvphi(t)} \dt
\quad \text{as } h,k \to 0.
\end{gather*}
For the first term on the right-hand side of~\eqref{mpslabel:eq:varform_before_limit} simple algebraic manipulations together with \eqref{mpslabel:eq:pseudo-projection} show that
\begin{align} \label{mpslabel:eq:convergenceTerm}
\notag\!\int_0^{kJ}\!\!  \inner[h]{ \mmhkbar(t) \!\times\! \Ph \heff(\mmhkbar(t))}{\vvphi_h(t)} \dt
& = \int_0^{kJ}\!\!  \dual{\heff(\mmhkbar(t))}{(\Interp\! -\! 1) [ \vvphi_h(t) \times \mmhkbar(t)]} \dt \\
&\quad + \int_0^{kJ}\!\! \dual{\heff(\mmhkbar(t))}{ \vvphi_h(t) \times \mmhkbar(t) } \dt.
\end{align}
Since
\begin{equation*}
\begin{split}
& \left\lvert \int_0^{kJ}  \dual{\heff(\mmhkbar(t))}{(\Interp - 1) [ \vvphi_h(t) \times \mmhkbar(t)]} \dt \right\rvert \\
& \quad \leq C \int_0^{kJ} ( \norm[\HH^1(\Omega)]{\mmhkbar(t)} + \norm[\LL^2(\Omega)]{\ff})
\norm[\HH^1(\Omega)]{(\Interp - 1) [ \vvphi_h(t) \times \mmhkbar(t)]} \dt
\end{split}
\end{equation*}
and
$\norm[L^{\infty}(0,T;\HH^1(\Omega))]{(\Interp - 1) [ \vvphi_h \times \mmhkbar]} \leq C h$
(see~\cite[equations~(39)--(40)]{prs2018}),
the first term on the right-hand side of~\eqref{mpslabel:eq:convergenceTerm} tends to $0$ as $h,k \to 0$.
Moreover, owing to the available convergence results and the convergence properties
of $a(\cdot,\cdot)$ and $\ppi(\cdot)$, it holds that
\begin{equation*}
\int_0^{kJ}  \dual{\heff(\mmhkbar(t))}{ \vvphi_h(t) \times \mmhkbar(t) } \dt
\to
\int_0^T  \dual{\heff(\mm(t))}{ \vvphi(t) \times \mm(t) } \dt
\quad
\text{as } h,k \to 0.
\end{equation*}
Hence, passing~\eqref{mpslabel:eq:varform_before_limit} to the limit as $h,k \to 0$, we obtain~\eqref{mpslabel:eq:weak:variational}
for any smooth test function $\vvphi$.
By density, we obtain the desired result.

Finally, the energy inequality~\eqref{mpslabel:eq:weak:energy} is obtained by passing to the limit
as $h,k \to 0$ the discrete energy identity~\eqref{mpslabel:eq:stability}
and using standard lower semicontinuity arguments
in combination with the available convergence results.
\end{proof}

%% file: sec_proof_practical_fp.tex
\section{Analysis of the practical midpoint scheme: constraint preserving fixed-point iteration} \label{mpslabel:sec:proofs2}

To start with, we recall that for quasi-uniform families of triangulations we have the inverse estimate
\begin{subequations}\label{mpslabel:eq:total}
\begin{align}
\label{mpslabel:eq:inverse}
\norm[\LL^2(\Omega)]{\Grad\pphi_h}
\le \Cinv h^{-1} \norm[\LL^2(\Omega)]{\pphi_h}
\quad
\text{for all } \pphi_h \in \Vh,
\end{align}
from which it follows that
\begin{align} \label{mpslabel:eq:boundednessPh}
\norm[h]{\Ph\pphi}
\le (1 + \Cinv^2 h^{-2})^{1/2} \norm[\HH^1(\Omega)^{\star}]{\pphi}
\quad
\text{for all } \pphi \in \HH^1(\Omega)^{\star}.
\end{align}
Here, $\Cinv>0$ depends only on $\kappa$.
Moreover, the following inequalities are direct consequences of~\eqref{mpslabel:eq:a:continuity}--\eqref{mpslabel:eq:b:continuity}:
\begin{align}
\label{mpslabel:eq:heff:lipschitz}
\norm[\HH^1(\Omega)^{\star}]{\heffloc(\ppsi) - \heffloc(\vvphi)}
&\le C_1 \norm[\HH^1(\Omega)]{\ppsi - \vvphi}
&\text{for all } \ppsi, \vvphi \in \HH^1(\Omega), \\
\label{mpslabel:eq:heff:affine}
\norm[\HH^1(\Omega)^{\star}]{\heffloc(\ppsi)}
&\le C_1 \norm[\HH^1(\Omega)]{\ppsi} + \norm[\LL^2(\Omega)]{\ff}
&\text{for all } \ppsi \in \HH^1(\Omega), \\
\label{mpslabel:eq:heffnof:linear}
\norm[\HH^1(\Omega)^{\star}]{\heffloc(\ppsi) - \ff}
&\le C_1 \norm[\HH^1(\Omega)]{\ppsi}
&\text{for all } \ppsi \in \HH^1(\Omega).
\end{align}
\end{subequations}

\subsection{Well-posedness}

We now prove Proposition~\ref{mpslabel:prop:fp}, which establishes the properties of
the constraint preserving fixed-point iteration proposed in Section~\ref{mpslabel:sec:fp}.

\begin{proof}[Proof of Proposition~\ref{mpslabel:prop:fp}{\textrm{(i)}}]
Since the bilinear form on the left-hand side of~\eqref{mpslabel:eq:mps_eta_fp} is elliptic with respect to the norm $\norm[h]{\cdot}$,
the variational problem admits a unique solution $\eeta_h^{i,\ell+1} \in \Vh$ for each $\ell \ge 0$.

Let $\ell \in \N_0$ and let $\zz \in \Nh$ be an arbitrary node.
Testing~\eqref{mpslabel:eq:mps_eta_fp} with $\pphih = \eeta_h^{i,\ell+1}(\zz) \vphi_{\zz} \in \Vh$,
we obtain that
\begin{equation*}
\beta_{\zz} \abs{\eeta_h^{i,\ell+1}(\zz)}^2
= \beta_{\zz} \eeta_h^{i,\ell+1}(\zz) \cdot \mm_{h\eps}^i(\zz).
\end{equation*}
Hence, $\abs{\eeta_h^{i,\ell+1}(\zz)} \leq \abs{\mm_{h\eps}^i(\zz)} = 1$.
We conclude that $\norm[\LL^{\infty}(\Omega)]{\eeta_h^{i,\ell+1}} \le 1$.
\end{proof}

\begin{proof}[Proof of Proposition~\ref{mpslabel:prop:fp}{\textrm{(ii)}}]
Let $\ell \in \N_0$.
Subtracting the equations satisfied by two consecutive iterates $\eeta_h^{i,\ell+1}, \eeta_h^{i,\ell+2} \in \Vh$
in~\eqref{mpslabel:eq:mps_eta_fp},
we obtain that
\begin{equation*}
\begin{split}
&\inner[h]{\eeta_h^{i,\ell+2} - \eeta_h^{i,\ell+1}}{\pphih} \\
& \quad \stackrel{\eqref{mpslabel:eq:mps_eta_fp}}{=}
- \frac{k}{2} \inner[h]{\eeta_h^{i,\ell+2} \times \Ph\heffloc(\eeta_h^{i,\ell+1})}{\pphih}
+ \frac{k}{2} \inner[h]{\eeta_h^{i,\ell+1} \times \Ph\heffloc(\eeta_h^{i,\ell})}{\pphih} \\
& \qquad -  \frac{k}{2} \inner[h]{(\eeta_h^{i,\ell+2} - \eeta_h^{i,\ell+1}) \times \Ph\PPi_h(\mm_{h\eps}^i,\mm_{h\eps}^{i-1})}{\pphih}
- \alpha \inner[h]{(\eeta_h^{i,\ell+2} - \eeta_h^{i,\ell+1}) \times \mm_{h\eps}^i}{\pphih}.
\end{split}
\end{equation*}
Choosing $\pphih = \eeta_h^{i,\ell+2} - \eeta_h^{i,\ell+1} \in \Vh$, we obtain that
\begin{equation*}
\begin{split}
&\norm[h]{\eeta_h^{i,\ell+2} - \eeta_h^{i,\ell+1}}^2 \\
& \ = - \frac{k}{2} \inner[h]{\eeta_h^{i,\ell+2} \times \Ph\heffloc(\eeta_h^{i,\ell+1})}{\eeta_h^{i,\ell+2} - \eeta_h^{i,\ell+1}}
+ \frac{k}{2} \inner[h]{\eeta_h^{i,\ell+1} \times \Ph\heffloc(\eeta_h^{i,\ell})}{\eeta_h^{i,\ell+2} - \eeta_h^{i,\ell+1}} \\
& \ = - \frac{k}{2} \inner[h]{\eeta_h^{i,\ell+1} \times \Ph(\heffloc(\eeta_h^{i,\ell+1}) - \heffloc(\eeta_h^{i,\ell}))}{\eeta_h^{i,\ell+2} - \eeta_h^{i,\ell+1}},
\end{split}
\end{equation*}
where the second equality can be seen by adding and subtracting 
the quantity
\begin{equation*}
\frac{k}{2} \inner[h]{\eeta_h^{i,\ell+1} \times \Ph\heffloc(\eeta_h^{i,\ell+1})}{\eeta_h^{i,\ell+2} - \eeta_h^{i,\ell+1}}.
\end{equation*}
It follows that
\begin{equation*}
\begin{split}
& \norm[h]{\eeta_h^{i,\ell+2} - \eeta_h^{i,\ell+1}}^2 \\
& \quad =
- \frac{k}{2} \inner[h]{\eeta_h^{i,\ell+1} \times \Ph(\heffloc(\eeta_h^{i,\ell+1}) - \heffloc(\eeta_h^{i,\ell}))}{\eeta_h^{i,\ell+2} - \eeta_h^{i,\ell+1}} \\
& \quad \le
\frac{k}{2} \norm[\LL^{\infty}(\Omega)]{\eeta_h^{i,\ell+1}}
\norm[h]{ \Ph(\heffloc(\eeta_h^{i,\ell+1}) - \heffloc(\eeta_h^{i,\ell})) }
\norm[h]{ \eeta_h^{i,\ell+2} - \eeta_h^{i,\ell+1} } \\
& \quad \le
\frac{k (1 + \Cinv^2 h^{-2})^{1/2}}{2}
\norm[\HH^1(\Omega)^{\star}]{ \heffloc(\eeta_h^{i,\ell+1}) - \heffloc(\eeta_h^{i,\ell}) }
\norm[h]{ \eeta_h^{i,\ell+2} - \eeta_h^{i,\ell+1} }
\end{split}
\end{equation*}
where the last inequality follows from
$\norm[\LL^{\infty}(\Omega)]{\eeta_h^{i,\ell+1}} \le 1$
and~\eqref{mpslabel:eq:boundednessPh}.
Moreover, it holds that
\begin{equation*}
\begin{split}
\norm[h]{\eeta_h^{i,\ell+2} - \eeta_h^{i,\ell+1}}
& \stackrel{\phantom{\eqref{mpslabel:eq:heff:lipschitz}}}{\le}
\frac{k (1 + \Cinv^2 h^{-2})^{1/2}}{2}
\norm[\HH^1(\Omega)^{\star}]{ \heffloc(\eeta_h^{i,\ell+1}) - \heffloc(\eeta_h^{i,\ell}) } \\
& \stackrel{\eqref{mpslabel:eq:heff:lipschitz}}{\le}
\frac{C_1 k (1 + \Cinv^2 h^{-2})^{1/2}}{2}
\norm[\HH^1(\Omega)]{ \eeta_h^{i,\ell+1} - \eeta_h^{i,\ell} } \\
& \stackrel{\eqref{mpslabel:eq:inverse}}{\le}
\frac{C_1 k (1 + \Cinv^2 h^{-2})}{2}
\norm[h]{ \eeta_h^{i,\ell+1} - \eeta_h^{i,\ell} }.
\end{split}
\end{equation*}
Since $k = o(h^2)$ as $h,k \to 0$, there exist $h_0, k_0 >0$ and a constant $0 < q < 1$
for which~\eqref{mpslabel:eq:contraction} holds for all $h < h_0$ and $k < k_0$.
\end{proof}

\begin{proof}[Proof of Proposition~\ref{mpslabel:prop:fp}{\textrm{(iii)}}]
Let $\ell \in \N_0$.
Using~\eqref{mpslabel:eq:boundednessPh},
\eqref{mpslabel:eq:heff:lipschitz},
and~\eqref{mpslabel:eq:normEquivalence}
as well as the fact that $\norm[\LL^{\infty}(\Omega)]{\eeta_h^{i,\ell+1}} \le 1$, 
we obtain that
\begin{equation*}
\begin{split}
\norm[h]{\Interp[\eeta_h^{i,\ell+1} \times \Ph (\heffloc(\eeta_h^{i,\ell+1}) - \heffloc(\eeta_h^{i,\ell}))]}
& \leq
\norm[h]{\Ph (\heffloc(\eeta_h^{i,\ell+1}) - \heffloc(\eeta_h^{i,\ell}))} \\
& \stackrel{\phantom{\eqref{mpslabel:eq:contraction}}}{\le}
C_1 (1 + \Cinv^2 h^{-2}) \norm[h]{\eeta_h^{i,\ell+1} - \eeta_h^{i,\ell}} \\
& \stackrel{\eqref{mpslabel:eq:contraction}}{\le}
C_1 (1 + \Cinv^2 h^{-2}) q^{\ell} \norm[h]{\eeta_h^{i,1} - \eeta_h^{i,0}} \\
& \ \le
2 C_1 \abs{\Omega}^{1/2}(1 + \Cinv^2 h^{-2}) q^{\ell}.
\end{split}
\end{equation*}
Hence, $\norm[h]{\Ph (\heffloc(\eeta_h^{i,\ell+1}) - \heffloc(\eeta_h^{i,\ell}))} \leq \eps$
for all $\ell \in \N_0$
satisfying
\begin{equation*}
\ell \geq \frac{\log(2 C_1 \abs{\Omega}^{1/2}(1 + \Cinv^2 h^{-2}) / \eps)}{\log(1/q)}.
\end{equation*}

Since $\mm_{h\eps}^{i+1} := 2 \, \eeta_h^{i,\ell^*+1} - \mm_{h\eps}^i$, there holds $\mm_{h\eps}^{i+1/2} = \eeta_h^{i,\ell^*+1}$.
From~\eqref{mpslabel:eq:mps_eta_fp} it follows that $\mm_{h\eps}^{i+1}$ solves
\begin{equation*}
\begin{split}
\inner[h]{d_t \mm_{h\eps}^{i+1}}{\pphih}
& = - \inner[h]{ \mm_{h\eps}^{i + 1/2} \times \Ph (\heffloc(\eeta_h^{i, \ell^*}) + \PPi_h(\mm_{h\eps}^i,\mm_{h\eps}^{i-1}))}{\pphih} \\
& \qquad+ \alpha \inner[h]{\mm_{h\eps}^{i+1/2} \times d_t \mm_{h\eps}^{i+1}}{\pphih}
\end{split}
\end{equation*}
for all $\pphi_h \in \Vh$.
Testing with $\pphih = \mm_{h\eps}^{i+1/2}(\zz) \vphi_{\zz} \in \Vh$ then reveals that $\mm_{h\eps}^{i+1} \in \Mh$
(see the proof of Theorem~\ref{mpslabel:thm:main}{\textrm{(i)}}).
\end{proof}

\subsection{Stability and weak convergence}
\label{mpslabel:sec:proof:mps:fp:stability}

Next, we provide the proof of Theorem~\ref{mpslabel:thm:main:fp},
which establishes the stability and convergence of Algorithm~\ref{mpslabel:alg:mps_fp}.

\begin{proof}[Proof of Theorem~\ref{mpslabel:thm:main:fp}]
Part~{\textrm{(i)}} is a direct consequence of Proposition~\ref{mpslabel:prop:fp}.
The proof of part~{\textrm{(ii)}} follows the lines of the one of Theorem~\ref{mpslabel:thm:main}{\textrm{(ii)}}.

Let us now consider the proof of part~{\textrm{(iii)}}.
Testing~\eqref{mpslabel:eq:mps:fp} with $\pphi_h = \alpha d_t\mm_{h\eps}^{i+1} - \Ph\heff(\mm_{h\eps}^{i+1/2}) + \Ph[\ppi(\mm_{h\eps}^{i+1/2}) - \PPi_h(\mm_{h\eps}^i, \mm_{h\eps}^{i-1})]$ yields
\begin{equation*}
\begin{split}
&\alpha\norm[h]{d_t\mm_{h\eps}^{i+1}}^2 
- \inner[h]{d_t\mm_{h\eps}^{i+1}}{\Ph\heff(\mm_{h\eps}^{i+1/2})} 
+ \inner[h]{d_t\mm_{h\eps}^{i+1}}{\Ph[\ppi(\mm_{h\eps}^{i+1/2}) - \PPi_h(\mm_{h\eps}^i, \mm_{h\eps}^{i-1})]}
\\
&\quad = \inner[h]{\mm_{h\eps}^{i+1/2}\times \rr_{h\eps}^i}{\alpha d_t\mm_{h\eps}^{i+1} - \Ph\heff(\mm_{h\eps}^{i+1/2}) + \Ph[\ppi(\mm_{h\eps}^{i+1/2}) - \PPi_h(\mm_{h\eps}^i, \mm_{h\eps}^{i-1})]}.
\end{split}
\end{equation*}
Using~\eqref{mpslabel:eq:aux_eq} and rearranging the terms, we obtain that
\begin{equation*}
\begin{split}
& \E(\mm_{h\eps}^{i+1}) 
+ \alpha k\norm[h]{d_t\mm_{h\eps}^{i+1}}^2 
= \E(\mm_{h\eps}^i) 
- k\inner[h]{d_t\mm_{h\eps}^{i+1}}{\Ph[\ppi(\mm_{h\eps}^{i+1/2}) - \PPi_h(\mm_{h\eps}^i, \mm_{h\eps}^{i-1})]}
\\
&\quad + k\inner[h]{\mm_{h\eps}^{i+1/2}\times \rr_{h\eps}^i}{\alpha d_t\mm_{h\eps}^{i+1} - \Ph\heff(\mm_{h\eps}^{i+1/2}) + \Ph[\ppi(\mm_{h\eps}^{i+1/2}) - \PPi_h(\mm_{h\eps}^i, \mm_{h\eps}^{i-1})]}.
\end{split}
\end{equation*}
Let $1 \leq j \leq J$.
Summation over $i = 0, \dots, j-1$ leads to
\begin{equation*}
\begin{split}
&\E(\mm_{h\eps}^j) 
+ \alpha k\sum_{i=0}^{j-1}\norm[h]{d_t\mm_{h\eps}^{i+1}}^2 
= \E(\mm_h^0) 
- k\sum_{i=0}^{j-1}\inner[h]{d_t\mm_{h\eps}^{i+1}}{\Ph[\ppi(\mm_{h\eps}^{i+1/2}) - \PPi_h(\mm_{h\eps}^i, \mm_{h\eps}^{i-1})]}
\\
&\quad + k\sum_{i=0}^{j-1}\inner[h]{\mm_{h\eps}^{i+1/2}\times \rr_{h\eps}^i}{\alpha d_t\mm_{h\eps}^{i+1} - \Ph\heff(\mm_{h\eps}^{i+1/2}) + \Ph[\ppi(\mm_{h\eps}^{i+1/2}) - \PPi_h(\mm_{h\eps}^i, \mm_{h\eps}^{i-1})]}.
\end{split}
\end{equation*}
Applying the G{\aa}rding inequality~\eqref{mpslabel:eq:gaarding} and continuity~\eqref{mpslabel:eq:a:continuity}
and using the fact that $\norm[\LL^{\infty}(\Omega)]{\mm_{h\eps}^j} = \norm[\LL^{\infty}(\Omega)]{\mm_h^0} = 1$,
we obtain that
\begin{equation*}
\begin{split}
&C_2\norm[\HH^1(\Omega)]{\mm_{h\eps}^{j}}^2
+ 2\alpha k\sum_{i=0}^{j-1}\norm[h]{d_t\mm_{h\eps}^{i+1}}^2 \\
& \quad \le
(C_1 + \norm[L(\LL^2(\Omega), \LL^2(\Omega))]{\ppi})\norm[\HH^1(\Omega)]{\mm_h^0}^2 
+ 2 \norm[\LL^2(\Omega)]{\ff} \big( \norm[\LL^2(\Omega)]{\mm_{h\eps}^j} + \norm[\LL^2(\Omega)]{\mm_h^0} \big) \\
&\qquad + C_3\norm[\LL^2(\Omega)]{\mm_{h\eps}^j}^2
- 2k\sum_{i=0}^{j-1}\inner[h]{d_t\mm_{h\eps}^{i+1}}{\Ph[\ppi(\mm_{h\eps}^{i+1/2}) - \PPi_h(\mm_{h\eps}^i, \mm_{h\eps}^{i-1})]}
\\
&\quad\!\! + 2k\sum_{i=0}^{j-1}\inner[h]{\mm_{h\eps}^{i+1/2}\times \rr_{h\eps}^i}{\alpha d_t\mm_{h\eps}^{i+1} - \Ph\heff(\mm_{h\eps}^{i+1/2}) + \Ph[\ppi(\mm_{h\eps}^{i+1/2}) - \PPi_h(\mm_{h\eps}^i, \mm_{h\eps}^{i-1})]}.
\end{split}
\end{equation*}
Using Young's inequality, the first sum on the right-hand side can be estimated as
\begin{equation*}
\begin{split}
&- 2k\sum_{i=0}^{j-1}\inner[h]{d_t\mm_{h\eps}^{i+1}}{\Ph[\ppi(\mm_{h\eps}^{i+1/2}) - \PPi_h(\mm_{h\eps}^i, \mm_{h\eps}^{i-1})]}
\\
& \quad = - 2k\sum_{i=0}^{j-1}\inner{d_t\mm_{h\eps}^{i+1}}{\ppi(\mm_{h\eps}^{i+1/2}) - \PPi_h(\mm_{h\eps}^i, \mm_{h\eps}^{i-1})}
\\
& \quad \le \alpha k\sum_{i=0}^{j-1}\norm[\LL^2(\Omega)]{d_t\mm_{h\eps}^{i+1}}^2
+ \frac{k}{\alpha} \sum_{i=0}^{j-1}\norm[\LL^2(\Omega)]{\ppi(\mm_{h\eps}^{i+1/2}) - \PPi_h(\mm_{h\eps}^i, \mm_{h\eps}^{i-1})}^2.
\end{split}
\end{equation*}
Since $\mm_{h\eps}^i \in \Mh$ for all $i=0, \dots, j-1$ it holds that
\begin{equation*}
\norm[\LL^2(\Omega)]{\ppi(\mm_{h\eps}^{i+1/2}) - \PPi_h(\mm_{h\eps}^i,\mm_{h\eps}^{i-1})}
\le \big(\norm[L(\LL^2(\Omega); \LL^2(\Omega))]{\ppi} + 2\Cpi\big) |\Omega|^{1/2}\,,
\end{equation*}
and hence
\begin{equation*}
\frac{k}{\alpha} \sum_{i=0}^{j-1}\norm[\LL^2(\Omega)]{\ppi(\mm_{h\eps}^{i+1/2}) - \PPi_h(\mm_{h\eps}^i, \mm_{h\eps}^{i-1})}^2
\leq \frac{C \abs{\Omega} (T + k_0)}{\alpha},
\end{equation*}
where $C>0$ depends only on $\ppi$ and $\Cpi$.
Hence,
using the norm equivalence~\eqref{mpslabel:eq:normEquivalence}, we obtain the estimate
\begin{equation*}
\begin{split}
&- 2k\sum_{i=0}^{j-1}\inner[h]{d_t\mm_{h\eps}^{i+1}}{\Ph[\ppi(\mm_{h\eps}^{i+1/2}) - \PPi_h(\mm_{h\eps}^i, \mm_{h\eps}^{i-1})]}
\\
& \quad \le \alpha k\sum_{i=0}^{j-1}\norm[h]{d_t\mm_{h\eps}^{i+1}}^2
+ \frac{C \abs{\Omega} (T + k_0)}{\alpha}.
\end{split}
\end{equation*}

Using the estimates
\begin{equation*}
\begin{split}
\norm[h]{\Ph \heff(\mm_{h\eps}^{i + 1/2})}
&
\stackrel{\eqref{mpslabel:eq:boundednessPh}}{\leq}
(1 + \Cinv^2 h^{-2})^{1/2}
\norm[\HH^{1}(\Omega)^\star]{\heff(\mm_{h\eps}^{i + 1/2})}
\\ &
\stackrel{\eqref{mpslabel:eq:heff:affine}}{\leq}
\big(1 + \Cinv^2 h^{-2})^{1/2} (C_1\norm[\HH^1(\Omega)]{\mm_{h\eps}^{i + 1/2}} + \norm[\LL^2(\Omega)]{\ff} \big),
\end{split}
\end{equation*}
\begin{equation*}
\begin{split}
\sum_{i=0}^{j-1} \norm[\HH^1(\Omega)]{\mm_{h\eps}^{i + 1/2}}
& 
\le
\frac{1}{2} \sum_{i=0}^{j-1} (\norm[\HH^1(\Omega)]{\mm_{h\eps}^{i + 1}} + \norm[\HH^1(\Omega)]{\mm_{h\eps}^i})
\\ &
=
\frac{1}{2} \norm[\HH^1(\Omega)]{\mm_h^0}
+ \sum_{i=1}^{j-1} \norm[\HH^1(\Omega)]{\mm_{h\eps}^i}
+ \frac{1}{2} \norm[\HH^1(\Omega)]{\mm_{h\eps}^j}
\\ &
\le \frac{j}{2}
+ \frac{1}{2} \norm[\HH^1(\Omega)]{\mm_{h\eps}^j}^2
+ \frac{1}{2} \sum_{i=0}^{j-1} \norm[\HH^1(\Omega)]{\mm_{h\eps}^i}^2,
\end{split}
\end{equation*}
and
\begin{equation*}
\norm[\LL^2(\Omega)]{d_t\mm_{h\eps}^{i+1}}
\leq
\frac{1}{4}
+ \norm[\LL^2(\Omega)]{d_t\mm_{h\eps}^{i+1}}^2,
\end{equation*}
together with the stopping criterion $\norm[h]{\Interp[\mm_{h\eps}^{i+1/2}\times\rr_{h\eps}^i]}\le \eps$ of Algorithm~\ref{mpslabel:alg:mps_fp},
if $h$ is sufficiently small,
we obtain that
\begin{equation*}
\begin{split}
&2k\sum_{i=0}^{j-1}\inner[h]{\mm_{h\eps}^{i+1/2}\times \rr_{h\eps}^i}{\alpha d_t\mm_{h\eps}^{i+1} - \Ph\heff(\mm_{h\eps}^{i+1/2}) + \Ph[\ppi(\mm_{h\eps}^{i+1/2}) - \PPi_h(\mm_{h\eps}^i, \mm_{h\eps}^{i-1})]} \\
&\quad
\le C' \eps (1 + h^{-1} )
+ 2k\eps\alpha\sum_{i=0}^{j-1} \norm[h]{d_t\mm_{h\eps}^{i+1}}^2 \\
& \qquad
+ C_1 k\eps \big(1 + \Cinv^2 h^{-2})^{1/2} \norm[\HH^1(\Omega)]{\mm_{h\eps}^j}^2
+ C_1 k\eps \big(1 + \Cinv^2 h^{-2})^{1/2} \sum_{i=0}^{j-1} \norm[\HH^1(\Omega)]{\mm_{h\eps}^i}^2,
\end{split}
\end{equation*}
where the constant $C' > 0$ depends only on $T$,
$\abs{\Omega}$,
$\ff$,
$\kappa$,
$\ppi$,
and $\Cpi$.
Altogether,
exploiting the assumption $\eps = \OO(h)$ as $h,\eps \to 0$,
there exist thresholds $0<h_0^* \leq h_0$, $0<k_0^* \leq k_0$,
and $\eps_0^*>0$ as well as constants $A, B > 0$
(depending only on $\kappa$, $T$, and the problem data)
such that
\begin{equation*}
\norm[\HH^1(\Omega)]{\mm_{h\eps}^j}^2
+ k \sum_{i=0}^{j-1} \norm[h]{d_t \mm_{h\eps}^i}^2
\le
A + B k \sum_{i=0}^{j-1} \norm[\HH^1(\Omega)]{\mm_{h\eps}^i}^2
\end{equation*}
for all $h<h_0^*$, $k<k_0^*$, and $\eps < \eps_0^*$.
Then, the discrete Gronwall lemma (see, e.g., \cite[Lemma~10.5]{thomee2006})
and the norm equivalence~\eqref{mpslabel:eq:normEquivalence}
yield~\eqref{mpslabel:eq:stability:fp}. 
This concludes the proof of part~{\textrm{(iii)}}.

The proof of part~{\textrm{(iv)}} follows the lines of~\cite{bp2006,bartels2006,cimrak2009,prs2018};
see also the proof of Theorem~\ref{mpslabel:thm:main}{\textrm{(iii)}}.
In particular, \eqref{mpslabel:eq:weak:variational} and~\eqref{mpslabel:eq:weak:energy}
are obtained by passing to the limit as $h,k,\eps \to 0$
the discrete identities~\eqref{mpslabel:eq:mps:fp} and~\eqref{mpslabel:eq:energylaw:fp}, respectively,
where the additional contributions arising from the linearization of the nonlinear system
(resp., from the explicit treatment of $\ppi$),
which do not appear in the proof of Theorem~\ref{mpslabel:thm:main}{\textrm{(iii)}},
vanish in the limit, because they are bounded by $\eps$
(resp., because $\ppi_h$ is assumed to be consistent with $\ppi$).
\end{proof}

%% file: sec_proof_practical_newton.tex
\section{Analysis of the practical midpoint scheme: Newton iteration} \label{mpslabel:sec:proof:newton}
\subsection[Stability and weak convergence]{Stability of Algorithm~\ref{mpslabel:alg:mps_newton}}\label{mpslabel:sec:newton:proof}
Lemma~\ref{mpslabel:le:newton:Linfty} and Theorem~\ref{mpslabel:thm:main:newton} assume well-posedness of Algorithm~\ref{mpslabel:alg:mps_newton} up to time-step $i < J$ , i.e., that for all $n = 0, \dots, i$ the Newton solver~\eqref{mpslabel:eq:mps_eta_newton} returns after finitely many iterations the solutions $\mm_{h\eps}^{n+1}, \rr_{h\eps}^n \in \Vh$ such that~\eqref{mpslabel:eq:mps:newton} holds with $\norm[h]{\rr_{h\eps}^n} \le \eps$.
Later, in Sections~\ref{mpslabel:sec:newton:proof:lipschitz}--\ref{mpslabel:sec:newton:proof:ellstar} Theorem~\ref{mpslabel:thm:main:newton:convergence} is proved, guaranteeing that, given appropriate CFL-conditions, this well-posedness assumption is always satisfied.
\subsubsection{Boundedness of magnetization length, Lemma~\ref{mpslabel:le:newton:Linfty}\textrm{(i)--(ii)}}\label{mpslabel:sec:newton:proof:linfty}
For $0 \le n \le i$ and $\zz\in\NN_h$, testing \eqref{mpslabel:eq:mps:newton} with $\pphi_h = \mm_{h\eps}^{n+1/2}(\zz)\varphi_{\zz} \in \Vh$ yields
\begin{equation*}
\begin{split}
\frac{1}{2k}\beta_{\zz}\left(|\mm_{h\eps}^{n+1}(\zz)|^2 - |\mm_{h\eps}^n(\zz)|^2\right) 
&= \inner[h]{\rr_{h\eps}^n}{\mm_{h\eps}^{n+1/2}(\zz)\varphi_{\zz}}
\le \norm[h]{\rr_{h\eps}^n}\norm[h]{\mm_{h\eps}^{n+1/2}(\zz)\varphi_{\zz}} \\
&\le \eps\beta_{\zz}^{1/2}|\mm_{h\eps}^{n+1/2}(\zz)| \\
&\le \eps\beta_{\zz}^{1/2}\left(\frac{1}{2}|\mm_{h\eps}^{n}(\zz)|^2 + \frac{1}{2}|\mm_{h\eps}^{n+1}(\zz)|^2 + \frac{1}{4}\right)\,.
\end{split}
\end{equation*}
Rearranging the terms and using $\eps\beta_{\zz}^{-1/2} \le C_1\eps h^{-3/2} =: C_{h\eps}$ uniformly for all $\zz \in \NN_h$, shows that for $k < 1 / (2C_{h\eps})$ it holds that
\begin{equation*}
\begin{split}
|\mm_{h\eps}^{n+1}(\zz)|^2 
&\le \frac{1 + C_{h\eps} k}{1 - C_{h\eps} k}|\mm_{h\eps}^n(\zz)|^2 + C_{h\eps} k \\
&= \left(1 + \frac{2C_{h\eps} k}{1 - C_{h\eps} k}\right)|\mm_{h\eps}^n(\zz)|^2 + C_{h\eps} k
\le (1 + 4C_{h\eps} k)|\mm_{h\eps}^n(\zz)|^2 + \frac{C_{h\eps} k}{2} \,.
\end{split}
\end{equation*}
Using $n \le i < J = T/k$ implies
\begin{equation*}
\begin{split}
|\mm_{h\eps}^{n+1}(\zz)|^2 
&\le (1+4C_{h\eps} k)^{n+1}|\mm_{h}^0(\zz)|^2 + C_{h\eps} k\sum_{p=0}^{n}(1 + 4C_{h\eps} k)^{p} \\
&\le \exp(4C_{h\eps} T)\left(|\mm_{h}^0(\zz)|^2 + C_{h\eps} T\right)\,.
\end{split}
\end{equation*}
Using $\mm_{h}^0 \in \Mh$ and uniform boundedness of $C_{h\eps}$ due to $\eps = \mathcal{O}(h^{3/2})$ concludes the proof of~{\textrm{(i)}}.
Analogously to the estimate above on $|\mm_{h\eps}^{n+1}(\zz)|^2$, starting from
\begin{align*}
\frac{1}{2k}\beta_{\zz}\left(|\mm_{h\eps}^{n+1}(\zz)|^2 - |\mm_{h\eps}^n(\zz)|^2\right) 
= \inner[h]{\rr_{h\eps}^n}{\mm_{h\eps}^{n+1/2}(\zz)\varphi_{\zz}}
\ge -\norm[h]{\rr_{h\eps}^n}\norm[h]{\mm_{h\eps}^{n+1/2}(\zz)\varphi_{\zz}}\,,
\end{align*}
by a similar computation one derives an estimate below via
\begin{align*}
|\mm_{h\eps}^{n+1}(\zz)|^2 
&\ge (1-4C_{h\eps} k)^{n+1}|\mm_{h}^0(\zz)|^2 - C_{h\eps} k\sum_{p=0}^{n}(1 - 4C_{h\eps} k)^{p} \\
&\ge \exp(-8C_{h\eps} T)|\mm_{h}^0(\zz)|^2 - \exp(-4C_{h\eps} T)C_{h\eps} T
\qquad\text{for all } 0 < k < k_0,
\end{align*}
where $k_0$ can be uniformly chosen since $\eps = \mathcal{O}(h^{3/2})$.
If $\eps = o(h^{3/2})$, then in both estimates $C_{h\eps}$ tends to zero as $h, \eps \to 0$.
Hence, also statement~{\textrm{(ii)}} holds true. \qed

\subsubsection{Stability and weak convergence, Theorem~\ref{mpslabel:thm:main:newton}\textrm{(i)--(iii)}}\label{mpslabel:sec:newton:proof:stability}
For $0 \le i < J$ testing~\eqref{mpslabel:eq:mps:newton} with $\pphi_h = \alpha d_t\mm_{h\eps}^{i+1} - \Ph\heff(\mm_{h\eps}^{i+1/2}) + \Ph(\ppi(\mm_{h\eps}^{i+1/2}) - \PPi_h(\mm_{h\eps}^i, \mm_{h\eps}^{i-1}))$ yields
\begin{equation*}
\begin{split}
&\alpha\norm[h]{d_t\mm_{h\eps}^{i+1}}^2 
- \inner[h]{d_t\mm_{h\eps}^{i+1}}{\Ph\heff(\mm_{h\eps}^{i+1/2})} 
+ \inner[h]{d_t\mm_{h\eps}^{i+1}}{\Ph(\ppi(\mm_{h\eps}^{i+1/2}) - \PPi_h(\mm_{h\eps}^i, \mm_{h\eps}^{i-1}))}
\\
&\quad = \inner[h]{\rr_{h\eps}^i}{\alpha d_t\mm_{h\eps}^{i+1} - \Ph\heff(\mm_{h\eps}^{i+1/2}) + \Ph(\ppi(\mm_{h\eps}^{i+1/2}) - \PPi_h(\mm_{h\eps}^i, \mm_{h\eps}^{i-1}))}.
\end{split}
\end{equation*}
Up to replacing $\rr_{h\eps}^i$ by $\mm_{h\eps}^{i+1/2} \times \rr_{h\eps}^i$, this identity resembles the first identity in Section~\ref{mpslabel:sec:proof:mps:fp:stability}, where Theorem~\ref{mpslabel:thm:main:fp}\textrm{(ii)--(iv)} is proved.
Hence, using $\LL^\infty(\Omega)$-boundedness of the iterates from Lemma~\ref{mpslabel:le:newton:Linfty}\textrm{(i)} and that the stopping criterion~\eqref{mpslabel:eq:stopping_newton} guarantees $\norm[h]{\rr_{h\eps}^i} \le \eps$, the proof of Theorem~\ref{mpslabel:thm:main:newton}{\textrm{(i)--(iii)}} directly follows the lines of Section~\ref{mpslabel:sec:proof:mps:fp:stability}. \qed

\subsection{Main theorem on Newton's method}
\label{mpslabel:sec:newton:main_thm}
Newton's method is an iterative scheme to generate a converging sequence of approximate solutions to the following problem:
Given $\FF\colon \R^n \to \R^n$,
\begin{equation}\label{mpslabel:eq:find_zero}
\text{find } \xx^\ast \in \R^n, \text{ such that } \FF(\xx^\ast) = \0.
\end{equation}
Here, $\FF$ is considered to be $C^1$-continuous on a convex open set $D \subseteq \R^n$ containing $\xx^\ast$ and the Jacobian of $\FF$ evaluated at $\xx \in \R^n$ is denoted by $\Grad\FF(\xx) \in \R^{n\times n}$.
Given a starting value $\xx^0 \in \R^n$, Newton's method applied to~\eqref{mpslabel:eq:find_zero} iterates for all $\ell \in \N_0$
\begin{equation}\label{mpslabel:eq:newton:method}
\begin{split}
 \textrm{solve } \quad& \Grad\FF(\xx^{\ell})\ddelta\xx^{\ell} = -\FF(\xx^{\ell})\,, \\
 \textrm{set } \quad& \xx^{\ell+1} = \xx^{\ell} + \ddelta\xx^{\ell}\,.
\end{split}
\end{equation}
Given a vector norm $\norm{\cdot}$ on $\R^n$, by $B(\norm{\cdot}; \xx, R)$ the open unit ball of radius $R > 0$ around $\xx \in \R^n$ with respect to the norm $\norm{\cdot}$ is denoted.
In accordance with \cite[Definition~1.20]{qss2007} a matrix norm $\norm[\R^{n\times n}]{\cdot}$ and a vector norm $\norm[\R^n]{\cdot}$ are called \emph{consistent}, if it holds that $\norm[\R^n]{\Amat\xx} \le \norm[\R^{n\times n}]{\Amat} \norm[\R^n]{\xx}$ for all $\Amat\in\R^{n\times n}$ and all $\xx\in\R^n$.
Clearly, any vector norm is consistent with the natural matrix norm induced by the vector norm defined as
\begin{equation}\label{mpslabel:eq:induced:norm}
\norm{\Amat} = \sup_{\xx\in\R^n\setminus\{\0\}} \frac{\norm{\Amat\xx}}{\norm{\xx}} \quad \text{for all } \Amat \in \R^{n\times n}\,.
\end{equation}
Using the above notation, we recall the classical local convergence result for Newton's method.

\begin{theorem}[{\cite[Theorem~7.1]{qss2007}}]\label{mpslabel:thm:newton_local}
For a convex open set $D \subseteq \R^n$ with $\xx^\ast \in D$, let $\FF \in C^1(D; \R^n)$ with $\FF(\xx^\ast) = \0$.
Suppose that $(\Grad\FF(\xx^\ast))^{-1} \in \R^{n\times n}$ exists and that there exist constants $C, R, L > 0$, such that 
\begin{subequations}
\begin{equation}\label{mpslabel:eq:newton_local:invertible}
\norm{(\Grad\FF(\xx^\ast))^{-1}} \le C\,,
\end{equation}
and
\begin{equation}\label{mpslabel:eq:newton_local:lipschitz}
\norm{\Grad\FF(\xx) - \Grad\FF(\yy)} \le L\norm{\xx - \yy} \quad\text{for all}\quad \xx, \yy \in B(\norm{\cdot}; \xx^\ast, R),
\end{equation}
\end{subequations}
where the symbol $\norm{\cdot}$ denotes two consistent vector and matrix norms.
Then, there holds that for any $\xx^{0} \in B(\norm{\cdot}; \xx^\ast, \min\{R, 1/(2CL)\})$, the sequence $(\xx^\ell)_{\ell\in\N}$ generated by Newton's method~\eqref{mpslabel:eq:newton:method}
is uniquely defined and converges to $\xx^\ast$ with
\begin{equation}\label{mpslabel:eq:newton:quadratic_convergence}
\norm{\xx^{\ell+1} - \xx^\ast} 
\le CL\norm{\xx^{\ell} - \xx^\ast}^2 \,.
\end{equation}
\end{theorem}

\begin{remark}\label{mpslabel:re:newton:error_decay}
In particular~\eqref{mpslabel:eq:newton:quadratic_convergence} and $\xx^{0} \in B(\norm{\cdot}; \xx^\ast, \min\{R, 1/(2CL)\})$ imply
\begin{equation}\label{mpslabel:eq:newton:exponential_decay}
\norm{\xx^{\ell} - \xx^\ast} 
\le \left(\prod_{j=0}^{\ell-1}(CL)^{2^j}\right)\norm{\xx^{0} - \xx^\ast}^{2^{\ell}} 
= (CL)^{2^{\ell} - 1} \norm{\xx^0 - \xx^\ast}^{2^\ell}
\le \frac{2\norm{\xx^0 - \xx^\ast}}{2^{2^\ell}}\,,
\end{equation}
for all $\ell \in \N_0$.
Hence, there holds $\norm{\xx^\ell - \xx^\ast}, \norm{\ddelta\xx^\ell} \to 0$ for $\ell \to \infty$.
\end{remark}

\subsection{Newton's method applied to the nonlinear midpoint scheme}
\label{mpslabel:sec:apply_newton_to_mps}
We aim to apply Newton's method~\eqref{mpslabel:eq:newton:method} to the nonlinear system of equations~\eqref{mpslabel:eq:mps_eta1_imex}, i.e., to the IMEX version of the ideal midpoint scheme where the lower order terms are integrated explicitly in time $\ppi(\mm_h^{i+1/2}) \approx \PPi_h(\mm_h^i, \mm_h^{i-1})$.
Consider a numbering of the nodes $\{\zz_j\colon j=1, \dots, N\} = \NN_h$ of the mesh $\T_h$, and associate with a given vector $\xx \in (\R^3)^{N}$ the finite element function defined by $\widehat{\xx} := \sum_{j=1}^N \xx_j\varphi_{\zz_j} \in \Vh$.
Further, for a finite element function $\uu = \sum_{j=1}^N \uu(\zz_j)\varphi_{\zz_j} \in \Vh$, we write $[\uu] \in (\R^{3})^N \simeq \R^{3N}$ for the vector of nodal values, i.e., $[\uu]_j := \uu(\zz_j) \in \R^3$.

The mass lumped scalar product $\inner[h]{\cdot}{\cdot}$ gives rise to the matrix $\matrixM_h\in(\R^{3\times 3})^{N\times N} \simeq \R^{3N\times 3N}$, defined via $(\matrixM_h)_{jk} := \delta_{j, k}\beta_{\zz_j}\II_{3\times 3} \in \R^{3\times 3}$.
Given $\mm_{h}^i\in\Vh$, the solution $\mm_{h}^{i+1/2}$ of~\eqref{mpslabel:eq:mps_eta1_imex} satisfies $\FF([\mm_{h}^{i+1/2}]) = \0$, with
\begin{equation}\label{mpslabel:eq:mps:newton:F}
\FF(\xx) := \matrixM_h\Big(\xx - [\mm_{h}^i] + \big[\Interp\big(\frac{k}{2}\widehat{\xx}\times\Ph\big(\heffloc(\widehat{\xx}) + \PPi_h(\mm_{h}^i, \mm_{h}^{i-1})\big) + \alpha\widehat{\xx}\times\mm_{h}^i\big)\big]\Big) \,.
\end{equation}
The Jacobian $\Grad\FF \colon \R^{3N} \to \R^{3N \times 3N}$ satisfies for all $\xx, \uu, \vv \in \R^{3N}$ that
\begin{equation}\label{mpslabel:eq:mps:newton:jacobian}
\begin{split}
\uu^\top \Grad\FF(\xx) \vv 
&= \inner[h]{\widehat{\uu}}{\widehat{\vv}}
+ \frac{k}{2}\inner[h]{\widehat{\uu}\times\Ph\heffloc(\widehat{\xx})}{\widehat{\vv}}
+ \frac{k}{2}\inner[h]{\widehat{\xx}\times\Ph\big(\heffloc(\widehat{\uu}) - \ff\big)}{\widehat{\vv}} \\
&\quad + \frac{k}{2}\inner[h]{\widehat{\uu}\times\Ph\PPi_h(\mm_{h}^i, \mm_{h}^{i-1})}{\widehat{\vv}}
+ \alpha\inner[h]{\widehat{\uu}\times\mm_{h}^i}{\widehat{\vv}}\,.
\end{split}
\end{equation}
Newton's method~\eqref{mpslabel:eq:newton:method} applied to the system~\eqref{mpslabel:eq:mps_eta1_imex} in the $i$th time-step now can be written as:
Given $\mm_{h\eps}^i \in \Vh$ and initial value $\eeta_h^{i,0} \in \Vh$, for all $\ell \in \N_0$ compute $\uu_h^{i,\ell} \in \Vh$ such that, for all $\pphi_h \in \Vh$, it holds that
\begin{subequations}
\begin{align} \label{mpslabel:eq:mps_eta_newton:in_proof}
\notag & \inner[h]{\uu_h^{i,\ell}}{\pphi_h}
+ \frac{k}{2} \inner[h]{ \uu_h^{i,\ell} \times \Ph\heffloc(\eeta_h^{i,\ell})}{\pphi_h}
+ \frac{k}{2} \inner[h]{\eeta_h^{i,\ell} \times \Ph\big(\heffloc(\uu_h^{i,\ell}) - \ff\big)}{\pphi_h} \\
\notag &\quad + \frac{k}{2}\inner[h]{\uu_h^{i,\ell} \times \Ph\PPi_h(\mm_{h\eps}^i, \mm_{h\eps}^{i-1})}{\pphi_h}
+ \alpha \inner[h]{\uu_h^{i,\ell} \times \mm_{h\eps}^i}{\pphi_h}\\
& = \inner[h]{\mm_{h\eps}^i - \eeta_h^{i,\ell}}{\pphi_h}
- \frac{k}{2} \inner[h]{ \eeta_h^{i,\ell} \times \Ph\heffloc(\eeta_h^{i,\ell})}{\pphi_h} \\
\notag &\quad - \frac{k}{2} \inner[h]{\eeta_h^{i,\ell} \times \Ph\PPi_h(\mm_{h\eps}^i, \mm_{h\eps}^{i-1})}{\pphi_h}
- \alpha \inner[h]{\eeta_h^{i,\ell} \times \mm_{h\eps}^i}{\pphi_h}\,,
\end{align}
and define
\begin{equation} \label{mpslabel:eq:mps_eta_newton:in_proof:update}
\eeta_h^{i,\ell+1} := \eeta_h^{i,\ell} + \uu_h^{i,\ell}\,.
\end{equation}
\end{subequations}
In the remainder of this section, to improve readability we omit the $h$-subscript of the iteration variables $\eeta^{i,\ell}$ and $\uu^{i,\ell}$.
Note that by~\eqref{mpslabel:eq:mps:newton:F}--\eqref{mpslabel:eq:mps:newton:jacobian} we see that~\eqref{mpslabel:eq:mps_eta_newton:in_proof}--\eqref{mpslabel:eq:mps_eta_newton:in_proof:update} resembles Newton's method~\eqref{mpslabel:eq:newton:method} with $\xx^\ell = [\eeta^{i, \ell}]$ and $\ddelta\xx^\ell = [\uu^{i, \ell}]$.
Given some tolerance $\eps > 0$, the iteration will be stopped once
\begin{equation} \label{mpslabel:eq:stopping_newton:in_proof}
\norm[h]{\Interp\big( \uu^{i,\ell} \times \Ph\big(\heffloc(\uu^{i,\ell}) - \ff\big)\big)}
\le \eps.
\end{equation}

If $\ell^* \in \N_0$ is the first index for which the stopping criterion~\eqref{mpslabel:eq:stopping_newton:in_proof} is satisfied,
the approximate magnetization at the new time-step is defined as
$\mm_{h\eps}^{i+1} := 2 \, \eeta^{i,\ell^*+1} - \mm_{h\eps}^i$.

For all $i \in \N_0$, let $\rr_{h\eps}^i := \Interp\big(\uu^{i,\ell^*} \times \Ph\big(\heffloc(\uu^{i,\ell^*}) - \ff\big)\big) \in \Vh$, so that $\inner[h]{\rr_{h\eps}^i}{\pphi_h}$ equals the difference of~\eqref{mpslabel:eq:mps_eta1_imex} and \eqref{mpslabel:eq:mps_eta_newton:in_proof}.
In view of the stopping criterion~\eqref{mpslabel:eq:stopping_newton:in_proof}, it holds that $\norm[h]{\rr_{h\eps}^i} \leq \eps$.
With this definition, the proposed linearization of one iteration of Algorithm~\ref{mpslabel:alg:mps} based on the Newton method is covered by Algorithm~\ref{mpslabel:alg:mps_newton}.

\subsection[Well-posedness]{Well-posedness of Algorithm~\ref{mpslabel:alg:mps_newton}}
\label{mpslabel:sec:proof:newton:well_posedness}
We show Theorem~\ref{mpslabel:thm:main:newton:convergence}{\textrm{(i)}} by induction:
For $0 \le i < J$ assume that Algorithm~\ref{mpslabel:alg:mps_newton} is well-defined for all $n = 0, \dots, i-1$.
In particular, by Lemma~\ref{mpslabel:le:newton:Linfty} and Theorem~\ref{mpslabel:thm:main:newton} we have the bounds
\begin{align}\label{mpslabel:eq:proof:induction:bounds}
\norm[\LL^\infty(\Omega)]{\mm_{h\eps}^i} \le C_\infty \quad\text{and}\quad \E(\mm_{h\eps}^i) \le C\,.
\end{align}
Now, the inductive step is to prove convergence of the Newton iteration \eqref{mpslabel:eq:mps_eta_newton}--\eqref{mpslabel:eq:stopping_newton} for time-step $n = i$.
We do this by verifying the assumptions~\eqref{mpslabel:eq:newton_local:invertible}--\eqref{mpslabel:eq:newton_local:lipschitz} of Theorem~\ref{mpslabel:thm:newton_local} for the Newton solver~\eqref{mpslabel:eq:mps_eta_newton:in_proof} with the initial value chosen as $\eeta^{i, 0} := \mm_{h\eps}^i$.
In Section~\ref{mpslabel:sec:newton:proof:lipschitz} we verify the Lipschitz continuity~\eqref{mpslabel:eq:newton_local:lipschitz}.
Invertibility~\eqref{mpslabel:eq:newton_local:invertible} is shown in Section~\ref{mpslabel:sec:newton:proof:invertibility}.
In Section~\ref{mpslabel:sec:newton:proof:x0} we prove that under the assumed CFL-conditions the initial guess $\eeta^{i, 0} := \mm_{h\eps}^i$ is an appropriate choice, which guarantees convergence of Newton's method.
Finally, in Section~\ref{mpslabel:sec:newton:proof:ellstar} we conclude by estimating the maximum number of Newton iterations required to achieve the required tolerance~\eqref{mpslabel:eq:stopping_newton:in_proof}, in particular showing that the number is finite.
Hence, Sections~\ref{mpslabel:sec:newton:proof:lipschitz}--\ref{mpslabel:sec:newton:proof:ellstar} prove Theorem~\ref{mpslabel:thm:main:newton:convergence}.\\ \\
\noindent Throughout the proof, we use the notation of Section~\ref{mpslabel:sec:newton:main_thm}--\ref{mpslabel:sec:apply_newton_to_mps} and consider the $\ell^2$-norm on $(\R^3)^N \simeq \R^{3N}$ defined by $\norm[2]{\xx} = \sum_{j=1}^{N} |\xx_j|^2$, as well as the induced matrix norm on $(\R^{3\times 3})^{N \times N} \simeq \R^{3N \times 3N}$ also denoted by $\norm[2]{\cdot}$, cf.\ \eqref{mpslabel:eq:induced:norm}.
\subsubsection{Lipschitz continuity of $\Grad\FF$}\label{mpslabel:sec:newton:proof:lipschitz}
By~\eqref{mpslabel:eq:induced:norm} it holds for arbitrary $\xx, \yy \in \R^{3N}$ that
\begin{equation*}
\begin{split}
\norm[2]{\Grad\FF(\xx) - \Grad\FF(\yy)} 
= \sup_{\uu, \vv \in \R^{3N}\setminus\{\0\}}
\frac{\uu^\top (\Grad\FF(\xx) - \Grad\FF(\yy)) \vv}{\norm[2]{\uu}\norm[2]{\vv}}\,.
\end{split}
\end{equation*}
With the representation~\eqref{mpslabel:eq:mps:newton:jacobian} and the estimates~\eqref{mpslabel:eq:total} we see
\begin{equation*}
\begin{split}
\uu^\top (\Grad\FF(\xx) - \Grad\FF(\yy)) \vv
&= \frac{k}{2}\inner[h]{\widehat{\uu}\times\Ph\big(\heffloc(\widehat{\xx})-\heffloc(\widehat{\yy})\big)}{\widehat{\vv}} \\
&\quad + \frac{k}{2}\inner[h]{(\widehat{\xx}-\widehat{\yy})\times\Ph\big(\heffloc(\widehat{\uu}) - \ff\big)}{\widehat{\vv}} \\
&\lesssim k\norm[\LL^\infty(\Omega)]{\widehat{\uu}} \norm[h]{\Ph\big(\heffloc(\widehat{\xx})-\heffloc(\widehat{\yy})\big)} \norm[h]{\widehat{\vv}}\\
&\quad+ k\norm[\LL^\infty(\Omega)]{\widehat{\xx}-\widehat{\yy}} \norm[h]{\Ph\big(\heffloc(\widehat{\uu}) - \ff\big)} \norm[h]{\widehat{\vv}} \\
&\lesssim kh^{-2}(\norm[2]{\uu} \norm[\LL^2(\Omega)]{\widehat{\xx}-\widehat{\yy}} + \norm[2]{\xx - \yy} \norm[\LL^2(\Omega)]{\widehat{\uu}}) \norm[h]{\widehat{\vv}}\,,
\end{split}
\end{equation*}
With the norm equivalence $h^{3/2}\norm[2]{\cdot} \simeq \norm[h]{\,\widehat{\cdot}\,} \simeq \norm[\LL^2(\Omega)]{\,\widehat{\cdot}\,}$ on $\R^{3N}$, we get uniformly for all $\xx, \yy \in \R^{3N}$ that
\begin{equation*}
\norm[2]{\Grad\FF(\xx) - \Grad\FF(\yy)} 
\lesssim kh \norm[2]{\xx - \yy}\,.
\end{equation*}
In particular, \eqref{mpslabel:eq:newton_local:lipschitz} holds for $\norm[2]{\cdot}$ with $R=+\infty$ and $L \simeq kh$.\qed
\subsubsection{Invertibility of $\Grad\FF(\xx^\ast)$}\label{mpslabel:sec:newton:proof:invertibility}
The unknown $\xx^\ast \in \R^{3N}$ is defined by $\FF(\xx^\ast) = \0$.
Hence,
\begin{equation*}
\begin{split}
0 
= [\widehat{\xx^\ast}(\zz_j)\varphi_{\zz_j}]^\top\FF(\xx^\ast)
\stackrel{\eqref{mpslabel:eq:mps:newton:F}}{=} \beta_{\zz_j}\big(|\xx_j^\ast|^2 - \xx_j^\ast\cdot\mm_{h\eps}^i(\zz_j)\big) \quad \text{for all}\quad j = 1, \dots, N
\end{split}
\end{equation*}
together with~\eqref{mpslabel:eq:proof:induction:bounds} guarantees boundedness
\begin{align}\label{mpslabel:eq:xstar_bounded}
\norm[\LL^\infty(\Omega)]{\widehat{\xx^\ast}} \le \norm[\LL^\infty(\Omega)]{\mm_{h\eps}^i} \le C_\infty\,.
\end{align}
Now the assumption $k=o(h^2)$ guarantees invertibility of $\Grad\FF(\xx^\ast)$ by ellipticity
\begin{equation*}
\yy^\top\Grad\FF(\xx^\ast)\yy
\stackrel{\eqref{mpslabel:eq:mps:newton:jacobian}}{=} \norm[h]{\widehat{\yy}}^2 
+ \frac{k}{2}\inner[h]{\widehat{\xx^\ast}\times\Ph\big(\heffloc(\widehat{\yy}) - \ff\big)}{\widehat{\yy}}
\gtrsim (1 - kh^{-2})\norm[h]{\widehat{\yy}}^2
\gtrsim \norm[h]{\widehat{\yy}}^2\,,
\end{equation*}
where we used~\eqref{mpslabel:eq:xstar_bounded} and the Cauchy--Schwarz inequality together with the estimates~\eqref{mpslabel:eq:total}.
To show boundedness of $(\Grad\FF(\xx^\ast))^{-1}$ we write
\begin{equation*}
\norm[2]{(\Grad\FF(\xx^\ast))^{-1}} 
= \sup_{\xx \in \R^{3N}\setminus\{\0\}}\frac{\norm[2]{(\Grad\FF(\xx^\ast))^{-1}\xx}}{\norm[2]{\xx}}
= \sup_{\xx \in \R^{3N}\setminus\{\0\}}\frac{\norm[2]{\yy(\xx)}}{\norm[2]{\xx}}\,,
\end{equation*}
with $\yy := \yy(\xx) := (\Grad\FF(\xx^\ast))^{-1}\xx$.
Using~\eqref{mpslabel:eq:mps:newton:jacobian}, it holds that
\begin{equation*}
\yy^\top\xx 
= \yy^\top\Grad\FF(\xx^\ast)\yy
= \norm[h]{\widehat{\yy}}^2 
+ \frac{k}{2}\inner[h]{\widehat{\xx^\ast}\times\Ph\big(\heffloc(\widehat{\yy}) - \ff\big)}{\widehat{\yy}}\,.
\end{equation*}
Using norm equivalences $h^{3/2}\norm[2]{\cdot} \simeq \norm[h]{\,\widehat{\cdot}\,} \simeq \norm[\LL^2(\Omega)]{\,\widehat{\cdot}\,}$ on $\R^{3N}$ and an inverse estimate, it follows that
\begin{equation*}
\begin{split}
h^{3}\norm[2]{\yy}^2 
&\lesssim \norm[h]{\widehat{\yy}}^2
= \yy^\top\xx 
- \frac{k}{2}\inner[h]{\widehat{\xx^\ast}\times\Ph\big(\heffloc(\widehat{\yy}) - \ff\big)}{\widehat{\yy}}\\
&\lesssim \norm[2]{\yy}\norm[2]{\xx} 
+ k\norm[\LL^\infty(\Omega)]{\widehat{\xx^\ast}}\norm[h]{\Ph\big(\heffloc(\widehat{\yy}) - \ff\big)}\norm[h]{\widehat{\yy}} \\
&\!\stackrel{\eqref{mpslabel:eq:total}}{\lesssim} \norm[2]{\yy}\norm[2]{\xx} 
+ k h^{-2}\norm[h]{\widehat{\yy}}^2
\lesssim \norm[2]{\yy}\norm[2]{\xx} 
+ k h\norm[2]{\yy}^2\,.
\end{split}
\end{equation*}
With the CFL condition $k = o(h^2)$ we estimate $h^3(1 - kh^{-2})\norm[2]{\yy} \lesssim h^3\norm[2]{\yy} \lesssim \norm[2]{\xx}$ and conclude that
\begin{equation*}
\norm[2]{(\Grad\FF(\xx^\ast))^{-1}} \lesssim h^{-3} \,.
\end{equation*}
In particular it holds~\eqref{mpslabel:eq:newton_local:invertible} for $\norm[2]{\cdot}$ with $C \simeq h^{-3}$.\qed
\subsubsection{Initial guess leads to convergence}\label{mpslabel:sec:newton:proof:x0}
We recall the results from Section~\ref{mpslabel:sec:newton:proof:lipschitz} and Section~\ref{mpslabel:sec:newton:proof:invertibility}:
The Newton iteration~\eqref{mpslabel:eq:mps_eta_newton:in_proof} satisfies the assumptions of Theorem~\ref{mpslabel:thm:newton_local} for $\norm[2]{\cdot}$ with $C \simeq h^{-3}$, $R = +\infty$ and $L \simeq kh$.
The theorem now guarantees convergence $\xx^\ell \to \xx^\ast$ in $\norm[2]{\cdot}$ as $\ell\to\infty$ of the Newton iteration~\eqref{mpslabel:eq:mps_eta_newton:in_proof} for any initial guess $\xx^0 \in \R^{3N}$ with $\norm[2]{\xx^\ast - \xx^0} \le 1 / (2CL) \simeq h^2 / k$.\\
\noindent Given $\mm_{h\eps}^i \in \Vh$, Algorithm~\ref{mpslabel:alg:mps_newton} defines the initial guess as $\xx^0 := [\eeta^{i, 0}] := [\mm_{h\eps}^i]$.
Let $\xx^\ast \in \R^{3N}$ be the solution of~\eqref{mpslabel:eq:find_zero}, i.e., by~\eqref{mpslabel:eq:mps:newton:F} it holds for all $\pphi_h\in\Vh$ that
\begin{equation*}
\inner[h]{\widehat{\xx^\ast} - \mm_{h\eps}^i}{\pphi_h}
= - \frac{k}{2}\inner[h]{\widehat{\xx^\ast} \times \Ph\big(\heffloc(\widehat{\xx^\ast}) + \PPi_h(\mm_{h\eps}^i, \mm_{h\eps}^{i-1})\big)}{\pphi_h}
+ \alpha \inner[h]{\widehat{\xx^\ast} \times (\widehat{\xx^\ast} - \mm_{h\eps}^{i})}{\pphi_h}\,.
\end{equation*}
Using $\pphi_h = \alpha(\widehat{\xx^\ast} - \mm_{h\eps}^i) - (k/2)\Ph\big(\heffloc(\widehat{\xx^\ast}) + \PPi_h(\mm_{h\eps}^i, \mm_{h\eps}^{i-1})\big) \in \Vh$ shows
\begin{equation*}
\begin{split}
\alpha\norm[h]{\widehat{\xx^\ast} - \mm_{h\eps}^i}^2 
&= \frac{k}{2}\inner[h]{\widehat{\xx^\ast} - \mm_{h\eps}^i}{\Ph\big(\heffloc(\widehat{\xx^\ast}) + \PPi_h(\mm_{h\eps}^i, \mm_{h\eps}^{i-1})\big)} \\
&= \frac{k}{2}\inner[h]{\widehat{\xx^\ast} - \mm_{h\eps}^i}{\Ph\heff(\widehat{\xx^\ast})}
+ \frac{k}{2}\inner[h]{\widehat{\xx^\ast} - \mm_{h\eps}^i}{\Ph\big(\PPi_h(\mm_{h\eps}^i, \mm_{h\eps}^{i-1}) - \ppi(\widehat{\xx^\ast})\big)}
\end{split}
\end{equation*}
We rewrite the first term on the right hand side as
\begin{equation*}
\begin{split}
4\inner[h]{\widehat{\xx^\ast} - \mm_{h\eps}^i}{\Ph\heff(\widehat{\xx^\ast})}
&\stackrel{\eqref{mpslabel:eq:pseudo-projection}}{=} 4\inner[]{\widehat{\xx^\ast} - \mm_{h\eps}^i}{\heff(\widehat{\xx^\ast})} \\
&\stackrel{\eqref{mpslabel:eq:heff}}{=} -4a(\widehat{\xx^\ast} - \mm_{h\eps}^i, \widehat{\xx^\ast}) 
+ 4\inner[\Omega]{\widehat{\xx^\ast} - \mm_{h\eps}^i}{\ff} \\
&= - a\big((2\widehat{\xx^\ast} - \mm_{h\eps}^i) - \mm_{h\eps}^i, (2\widehat{\xx^\ast} - \mm_{h\eps}^i) + \mm_{h\eps}^i\big) \\
&\quad\, +2\inner[\Omega]{2\widehat{\xx^\ast} - \mm_{h\eps}^i}{\ff}
-2\inner[\Omega]{\mm_{h\eps}^i}{\ff} \\
&\stackrel{\eqref{mpslabel:eq:llg:energy}}{=} 2\E(\mm_{h\eps}^i)
- a(2\widehat{\xx^\ast} - \mm_{h\eps}^i, 2\widehat{\xx^\ast} - \mm_{h\eps}^i)
+ 2\inner[\Omega]{\widehat{\xx^\ast} - \mm_{h\eps}^i}{\ff} \,.
\end{split}
\end{equation*}
With~\eqref{mpslabel:eq:xstar_bounded} and the G{\aa}rding inequality~\eqref{mpslabel:eq:gaarding} we estimate
\begin{equation*}
\begin{split}
- a(2\widehat{\xx^\ast} - \mm_{h\eps}^i, 2\widehat{\xx^\ast} - \mm_{h\eps}^i)
\le C_3 \norm[\LL^2(\Omega)]{2\widehat{\xx^\ast} - \mm_{h\eps}^i}^2
- C_2 \norm[\HH^1(\Omega)]{2\widehat{\xx^\ast} - \mm_{h\eps}^i}^2
\le 3 C_3 C_\infty^2 |\Omega|\,.
\end{split}
\end{equation*}
Now combination with the generous estimates
\begin{equation*}
\begin{split}
\inner[h]{\widehat{\xx^\ast} - \mm_{h\eps}^i}{\Ph\big(\PPi_h(\mm_{h\eps}^i, \mm_{h\eps}^{i-1}) - \ppi(\widehat{\xx^\ast})\big)}
&\le 2C_\infty^2|\Omega| (2C_{\ppi} + \norm[L(\LL^2(\Omega); \LL^2(\Omega))]{\ppi})\,, \\
\inner[\Omega]{2\widehat{\xx^\ast} - \mm_{h\eps}^i}{\ff}
&\le 3C_\infty\norm[\LL^2(\Omega)]{\ff}
\end{split}
\end{equation*}
yields
\begin{equation*}
\begin{split}
\norm[h]{\widehat{\xx^\ast} - \mm_{h\eps}^i}^2 
&\lesssim \E(\mm_{h\eps}^i) k 
+ C\big(C_3, C_\infty, |\Omega|, C_{\ppi}, \norm[L(\LL^2(\Omega); \LL^2(\Omega))]{\ppi}, \norm[\LL^2(\Omega)]{\ff}\big) k
\stackrel{\eqref{mpslabel:eq:proof:induction:bounds}}{\lesssim} k \,.
\end{split}
\end{equation*}
Due to the norm equivalence $h^{3/2}\norm[2]{\cdot} \simeq \norm[h]{\,\widehat{\cdot}\,}$ on $\R^{3N}$, the claim $\xx^0\in B(\norm[2]{\cdot}; \xx^\ast, 1/(2CL))$ follows for $h, k \to 0$ from $k = o(h^{7/3})$ via
\begin{equation*}
\norm[2]{\xx^\ast -\xx^0} 
\simeq h^{-3/2}\norm[h]{\widehat{\xx^\ast} - \mm_{h\eps}^i}
\lesssim k^{1/2}h^{-3/2}
= k^{-1}h^2 (k^{3/2}h^{-7/2})
< k^{-1}h^2
\simeq 1/(2CL)\,.
\end{equation*}
Hence, the choice $\eeta^{i, 0} := \mm_{h\eps}^i$ implies convergence $\xx^\ell \to \xx^\ast$ in $\norm[2]{\cdot}$ as $\ell\to\infty$.
\qed
\subsubsection{Finite number of Newton iterations}\label{mpslabel:sec:newton:proof:ellstar}
In the previous section we showed $\xx^\ell \to \xx^\ast$ in $\norm[2]{\cdot}$ and therefore also $\norm[2]{\ddelta\xx^\ell} \to 0$ as $\ell \to \infty$.
Now let $\ell^\ast \in \N$ be the smallest integer, such that~\eqref{mpslabel:eq:stopping_newton:in_proof} is satisfied.
The index $\ell^\ast$ is well defined due to $\widehat{\ddelta\xx^\ell} = \uu^{i, \ell}$ and
\begin{equation*}
\begin{split}
\norm[h]{\Interp\big(\widehat{\ddelta\xx^{\ell}}\times\Ph\big(\heffloc(\widehat{\ddelta\xx^{\ell}}) - \ff\big)\big)}
\lesssim \norm[\infty]{\ddelta\xx^{\ell}}h^{-2}\norm[h]{\widehat{\ddelta\xx^{\ell}}}
\lesssim h^{-1/2}\norm[2]{\ddelta\xx^{\ell}}^2 \to 0
\;\; \text{as} \;\; \ell \to \infty.
\end{split}
\end{equation*}
Recalling that by Remark~\ref{mpslabel:re:newton:error_decay} it holds that
\begin{equation*}
\norm[2]{\xx^\ast - \xx^{\ell}} \le  \frac{2\norm[2]{\xx^0 - \xx^\ast}}{2^{2^\ell}}\,,
\end{equation*}
we estimate the index $\ell^\ast \in \N$:
With the estimates~\eqref{mpslabel:eq:total} and the norm equivalence $h^{3/2}\norm[2]{\cdot} \simeq \norm[h]{\,\widehat{\cdot}\,}$ on $\R^{3N}$, it holds for the error $\rr_{h\eps}^i$ that
\begin{equation*}
\begin{split}
\norm[h]{\rr_{h\eps}^i} 
&= \norm[h]{\Interp\big(\uu^{i, \ell^\ast}\times\Ph\big(\heffloc(\uu^{i, \ell^\ast}) - \ff\big)\big)}
= \norm[h]{\Interp\big(\widehat{\ddelta\xx^{\ell^\ast}}\times\Ph\big(\heffloc(\widehat{\ddelta\xx^{\ell^\ast}}) - \ff\big)\big)} \\
&\le \norm[\LL^\infty(\Omega)]{\widehat{\ddelta\xx^{\ell^\ast}}}\norm[h]{\Ph\big(\heffloc(\widehat{\ddelta\xx^{\ell^\ast}}) - \ff\big)}
\lesssim \norm[\infty]{\ddelta\xx^{\ell^\ast}}h^{-2}\norm[h]{\widehat{\ddelta\xx^{\ell^\ast}}} \\
&\lesssim h^{-1/2}\norm[2]{\ddelta\xx^{\ell^\ast}}^2
= h^{-1/2}\norm[2]{\xx^{\ell^\ast + 1} - \xx^{\ell^\ast}}^2
\lesssim h^{-1/2}\Big(\norm[2]{\xx^{\ast} - \xx^{\ell^\ast + 1}}^2
+ \norm[2]{\xx^{\ast} - \xx^{\ell^\ast}}^2\Big) \\
&\lesssim 2^{-2^{\ell^\ast}}h^{-1/2}\norm[2]{\xx^0 - \xx^\ast}^2
\lesssim 2^{-2^{\ell^\ast}}kh^{-7/2}\,.
\end{split}
\end{equation*}
Since $\ell^\ast \in \N$ is defined as the smallest integer, such that~\eqref{mpslabel:eq:stopping_newton:in_proof} is satisfied, $\ell^\ast$ is estimated from above by $\log_2\log_2(C_\newton kh^{-7/2}\eps^{-1})$ with a generic constant $C_\newton > 0$.
\qed

%% file: llg+dmi+mps.bbl
\newcommand{\etalchar}[1]{$^{#1}$}
\begin{thebibliography}{ADMN21}

\bibitem[AD15]{ad2015}
F.~Alouges and G.~{Di Fratta}.
\newblock Homogenization of composite ferromagnetic materials.
\newblock {\em Proc. R. Soc. Lond. A}, 471(2182):20150365, 2015.

\bibitem[ADMN21]{abmn2021}
F.~Alouges, A.~{De Bouard}, B.~Merlet, and L.~Nicolas.
\newblock Stochastic homogenization of the {L}andau--{L}ifshitz--{G}ilbert
  equation.
\newblock {\em Stoch. Partial Differ. Equ. Anal. Comput.}, 9:789--818, 2021.

\bibitem[AFKL21]{afkl2021}
G.~Akrivis, M.~Feischl, B.~Kov\'{a}cs, and C.~Lubich.
\newblock Higher-order linearly implicit full discretization of the
  {L}andau-{L}ifshitz-{G}ilbert equation.
\newblock {\em Math. Comp.}, 90(329):995--1038, 2021.

\bibitem[AFM06]{afm2006}
E.~Acerbi, I.~Fonseca, and G.~Mingione.
\newblock Existence and regularity for mixtures of micromagnetic materials.
\newblock {\em Proc. R. Soc. Lond. A}, 462(2072):2225--2243, 2006.

\bibitem[AHP{\etalchar{+}}14]{ahpprs2014}
C.~Abert, G.~Hrkac, M.~Page, D.~Praetorius, M.~Ruggeri, and D.~Suess.
\newblock Spin-polarized transport in ferromagnetic multilayers: {A}n
  unconditionally convergent {FEM} integrator.
\newblock {\em Comput. Math. Appl.}, 68(6):639--654, 2014.

\bibitem[AJ06]{aj2006}
F.~Alouges and P.~Jaisson.
\newblock Convergence of a finite element discretization for the
  {L}andau--{L}ifshitz equation in micromagnetism.
\newblock {\em Math. Models Methods Appl. Sci.}, 16(2):299--316, 2006.

\bibitem[AKST14]{akst2014}
F.~Alouges, E.~Kritsikis, J.~Steiner, and J.-C. Toussaint.
\newblock A convergent and precise finite element scheme for
  {L}andau--{L}ifschitz--{G}ilbert equation.
\newblock {\em Numer. Math.}, 128(3):407--430, 2014.

\bibitem[AKT12]{akt2012}
F.~Alouges, E.~Kritsikis, and J.-C. Toussaint.
\newblock A convergent finite element approximation for
  {L}andau--{L}ifschitz--{G}ilbert equation.
\newblock {\em Physica B}, 407(9):1345--1349, 2012.

\bibitem[Alo08]{alouges2008a}
F.~Alouges.
\newblock A new finite element scheme for {L}andau--{L}ifchitz equations.
\newblock {\em Discrete Contin. Dyn. Syst. Ser. S}, 1(2):187--196, 2008.

\bibitem[AS92]{as1992}
F.~Alouges and A.~Soyeur.
\newblock On global weak solutions for {L}andau--{L}ifshitz equations:
  Existence and nonuniqueness.
\newblock {\em Nonlinear Anal.}, 18(11):1071--1084, 1992.

\bibitem[Bar06]{bartels2006}
S.~Bartels.
\newblock Constraint preserving, inexact solution of implicit discretizations
  of {L}andau--{L}ifshitz--{G}ilbert equations and consequences for
  convergence.
\newblock {\em PAMM}, 6(1):19--22, 2006.

\bibitem[Bar15]{bartels2015}
S.~Bartels.
\newblock {\em Numerical methods for nonlinear partial differential equations}.
\newblock Springer, Cham, 2015.

\bibitem[BBNP14]{bbnp2014}
L'. Ba{\v n}as, Z.~Brze{\'z}niak, M.~Neklyudov, and A.~Prohl.
\newblock {\em Stochastic ferromagnetism: Analysis and numerics}, volume~58 of
  {\em Studies in Mathematics}.
\newblock De Gruyter, 2014.

\bibitem[BFF{\etalchar{+}}14]{bffgpprs2014}
F.~Bruckner, M.~Feischl, T.~F{\"u}hrer, P.~Goldenits, M.~Page, D.~Praetorius,
  M.~Ruggeri, and D.~Suess.
\newblock Multiscale modeling in micromagnetics: {E}xistence of solutions and
  numerical integration.
\newblock {\em Math. Models Methods Appl. Sci.}, 24(13):2627--2662, 2014.

\bibitem[BKP08]{bkp2008}
S.~Bartels, J.~Ko, and A.~Prohl.
\newblock Numerical analysis of an explicit approximation scheme for the
  {L}andau--{L}ifshitz--{G}ilbert equation.
\newblock {\em Math. Comp.}, 77(262):773--788, 2008.

\bibitem[BP06]{bp2006}
S.~Bartels and A.~Prohl.
\newblock Convergence of an implicit finite element method for the
  {L}andau--{L}ifshitz--{G}ilbert equation.
\newblock {\em SIAM J. Numer. Anal.}, 44(4):1405--1419, 2006.

\bibitem[Bro63]{brown1963}
W.~F. Brown.
\newblock {\em Micromagnetics}.
\newblock Interscience Publishers, New York, 1963.

\bibitem[Cim08]{cimrak2008b}
I.~Cimr{\'a}k.
\newblock A survey on the numerics and computations for the
  {L}andau--{L}ifshitz equation of micromagnetism.
\newblock {\em Arch. Comput. Methods Eng.}, 15(3):277--309, 2008.

\bibitem[Cim09]{cimrak2009}
I.~Cimr{\'a}k.
\newblock Convergence result for the constraint preserving mid-point scheme for
  micromagnetism.
\newblock {\em J. Comput. Appl. Math.}, 228(1):238--246, 2009.

\bibitem[DD20]{dd2020}
E.~Davoli and G.~{Di Fratta}.
\newblock Homogenization of chiral magnetic materials - {A} mathematical
  evidence of {D}zyaloshinskii's predictions on helical structures.
\newblock {\em J. Nonlinear Sci.}, 30:1229--1262, 2020.

\bibitem[DMRS20]{dmrs2020}
G.~{Di Fratta}, C.~B. Muratov, F.~N. Rybakov, and V.~V Slastikov.
\newblock Variational principles of micromagnetics revisited.
\newblock {\em SIAM J. Math. Anal.}, 52(4):3580--3599, 2020.

\bibitem[DPP{\etalchar{+}}20]{dpprs2017}
G.~{Di Fratta}, C.-M. Pfeiler, D.~Praetorius, M.~Ruggeri, and B.~Stiftner.
\newblock Linear second-order {IMEX}-type integrator for the (eddy current)
  {L}andau--{L}ifshitz--{G}ilbert equation.
\newblock {\em IMA J. Numer. Anal.}, 40(4):2802--2838, 2020.

\bibitem[Dzy58]{dzyaloshinskii1958}
I.~Dzyaloshinskii.
\newblock A thermodynamic theory of `weak' ferromagnetism of
  antiferromagnetics.
\newblock {\em J. Phys. Chem. Solids}, 4(4):241--255, 1958.

\bibitem[FCS13]{fcs2013}
A.~Fert, V.~Cros, and J.~Sampaio.
\newblock Skyrmions on the track.
\newblock {\em Nat. Nanotechnol.}, 8(3):152--156, 2013.

\bibitem[FRC17]{frc2017}
A.~Fert, N.~Reyren, and V.~Cros.
\newblock Magnetic skyrmions: advances in physics and potential applications.
\newblock {\em Nat. Rev. Mater.}, 2:17031, 2017.

\bibitem[FT17]{ft2017}
M.~Feischl and T.~Tran.
\newblock The {Eddy Current-LLG} equations: {FEM-BEM} coupling and a priori
  error estimates.
\newblock {\em SIAM J. Numer. Anal.}, 55(4):1786--1819, 2017.

\bibitem[GC07]{gc2007}
C.~J. Garc{\'i}a-Cervera.
\newblock Numerical micromagnetics: {A} review.
\newblock {\em Bol. Soc. Esp. Mat. Apl. SeMA}, 39:103--135, 2007.

\bibitem[Gil55]{gilbert1955}
T.~L. Gilbert.
\newblock A {L}agrangian formulation of the gyromagnetic equation of the
  magnetization fields.
\newblock {\em Phys. Rev.}, 100:1243, 1955.
\newblock Abstract only.

\bibitem[GR86]{gr1986}
V.~Girault and P.-A. Raviart.
\newblock {\em Finite element methods for {N}avier--{S}tokes equations:
  {T}heory and algorithms}, volume~5 of {\em Springer Series in Computational
  Mathematics}.
\newblock Springer, 1986.

\bibitem[HPP{\etalchar{+}}19]{hpprss2019}
G.~Hrkac, C.-M. Pfeiler, D.~Praetorius, M.~Ruggeri, A.~Segatti, and
  B.~Stiftner.
\newblock Convergent tangent plane integrators for the simulation of chiral
  magnetic skyrmion dynamics.
\newblock {\em Adv. Comput. Math.}, 45(3):1329--1368, 2019.

\bibitem[HS98]{hs1998}
A.~Hubert and R.~Sch{\"a}fer.
\newblock {\em Magnetic domains: {T}he analysis of magnetic microstructures}.
\newblock Springer, 1998.

\bibitem[KP06]{kp2006}
M.~Kruzik and A.~Prohl.
\newblock Recent developments in the modeling, analysis, and numerics of
  ferromagnetism.
\newblock {\em SIAM Rev.}, 48(3):439--483, 2006.

\bibitem[KW18]{kw2018}
E.~Kim and J.~Wilkening.
\newblock Convergence of a mass-lumped finite element method for the
  {L}andau--{L}ifshitz equation.
\newblock {\em Quart. Appl. Math.}, 76:383--405, 2018.

\bibitem[LL35]{ll1935}
L.~Landau and E.~Lifshitz.
\newblock On the theory of the dispersion of magnetic permeability in
  ferromagnetic bodies.
\newblock {\em Phys. Zeitsch. der Sow.}, 8:153--168, 1935.

\bibitem[Mor60]{moriya1960}
T.~Moriya.
\newblock Anisotropic superexchange interaction and weak ferromagnetism.
\newblock {\em Phys. Rev.}, 120(91):91, 1960.

\bibitem[ngs]{ngsolve}
{N}etgen/{NGS}olve {F}inite {E}lement {L}ibrary.
\newblock \url{https://ngsolve.org/}.
\newblock Accessed on March 12, 2022.

\bibitem[Pfe]{commics}
C.-M. Pfeiler.
\newblock Commics -- {A} {P}ython module for {C}omputational {M}icromagnetics.
\newblock \url{https://gitlab.asc.tuwien.ac.at/cpfeiler/commics}.
\newblock Accessed on March 12, 2022.

\bibitem[Pra04]{praetorius2004}
D.~Praetorius.
\newblock Analysis of the operator {$\Delta^{-1}\mathrm{div}$} arising in
  magnetic models.
\newblock {\em Z. Anal. Anwend.}, 23(3):589--605, 2004.

\bibitem[Pro01]{prohl2001}
A.~Prohl.
\newblock {\em Computational micromagnetism}.
\newblock Teubner, Wiesbaden, 2001.

\bibitem[PRS18]{prs2018}
D.~Praetorius, M.~Ruggeri, and B.~Stiftner.
\newblock Convergence of an implicit-explicit midpoint scheme for computational
  micromagnetics.
\newblock {\em Comput.\ Math.\ Appl.}, 75(5):1719--1738, 2018.

\bibitem[PRS{\etalchar{+}}20]{prsehhsmp2020}
C.-M. Pfeiler, M.~Ruggeri, B.~Stiftner, L.~Exl, M.~Hochsteger, G.~Hrkac,
  J.~Sch\"oberl, N.~J. Mauser, and D.~Praetorius.
\newblock Computational micromagnetics with {C}ommics.
\newblock {\em Comput. Phys. Commun.}, 248:106965, 2020.

\bibitem[QSS07]{qss2007}
A.~Quarteroni, R.~Sacco, and F.~Saleri.
\newblock {\em Numerical mathematics}, volume~37 of {\em Texts in Applied
  Mathematics}.
\newblock Springer-Verlag, Berlin, second edition, 2007.

\bibitem[SCR{\etalchar{+}}13]{scrtf2013}
J.~Sampaio, V.~Cros, S.~Rohart, A.~Thiaville, and A.~Fert.
\newblock Nucleation, stability and current-induced motion of isolated magnetic
  skyrmions in nanostructures.
\newblock {\em Nat. Nanotechnol.}, 8(11):839--844, 2013.

\bibitem[Tho06]{thomee2006}
V.~Thom{\'e}e.
\newblock {\em Galerkin finite element methods for parabolic problems},
  volume~25 of {\em Springer Series in Computational Mathematics}.
\newblock Springer, second edition, 2006.

\bibitem[TRJF12]{trjcf2012}
A.~Thiaville, S.~Rohart, V.~Ju\'{e}, \'{E}.~Cros, and A.~Fert.
\newblock Dynamics of {D}zyaloshinskii domain walls in ultrathin magnetic
  films.
\newblock {\em Europhys. Lett.}, 100(5):57002, 2012.

\end{thebibliography}
